\documentclass[12pt,a4paper]{amsart}

\usepackage[text={400pt,650pt},centering]{geometry}

\usepackage{color}
\usepackage{esint,amssymb}
\usepackage{graphicx}
\usepackage{MnSymbol}
\usepackage{mathtools}
\usepackage[colorlinks=true, pdfstartview=FitV, linkcolor=blue, citecolor=blue, urlcolor=blue,pagebackref=false]{hyperref}
\usepackage{microtype}

\newtheorem{theorem}{Theorem}
\newtheorem{proposition}[theorem]{Proposition}
\newtheorem{lemma}[theorem]{Lemma}

\theoremstyle{definition}
\newtheorem{remark}[theorem]{Remark}
\newtheorem{example}[theorem]{Example}
\newtheorem{definition}[theorem]{Definition}

\newcommand{\lref}[1]{Lemma~\ref{l.#1}}
\newcommand{\cref}[1]{Corollary~\ref{c.#1}}

\numberwithin{equation}{section}
\numberwithin{theorem}{section}

\newcommand{\N}{\mathbb{N}}
\newcommand{\R}{\mathbb{R}}

\newcommand{\E}{\mathbb{E}}
\renewcommand{\P}{\mathbb{P}}
\newcommand{\F}{\mathcal{F}}
\newcommand{\Zd}{\mathbb{Z}^d}
\newcommand{\Rd}{\mathbb{R}^d}

\newcommand{\ep}{\varepsilon}
\newcommand{\eps}{\varepsilon}

\newcommand{\Fz}{F}

\newcommand{\J}{\mathcal{J}}
\renewcommand{\a}{\mathbf{a}}
\newcommand{\h}{\mathbf{h}}
\newcommand{\g}{\mathbf{g}}
\newcommand{\T}{\mathsf{T}}
\renewcommand{\subset}{\subseteq}

\newcommand{\uhom}{\overline u}
\newcommand{\uhomr}{\overline u_r}
\newcommand{\ghom}{\overline \g}

\newcommand{\cu}{\square}
\newcommand{\cut}{\boxbox}

\newcommand{\ext}[1]{\tilde{#1}}

\newcommand{\Ls}{L^2_{\mathrm{sol}}}
\newcommand{\Lso}{L^2_{\mathrm{sol},0}}

\renewcommand{\L}{\mathcal{L}}
\renewcommand{\fint}{\strokedint}

\newcommand{\Ll}{\left}
\newcommand{\Rr}{\right}

\DeclareMathOperator{\dist}{dist}
\DeclareMathOperator*{\essinf}{ess\,inf}

\DeclareMathOperator*{\osc}{osc}
\DeclareMathOperator{\var}{var}
\DeclareMathOperator{\cov}{cov}

\DeclareMathOperator{\divg}{div}
\DeclareMathOperator{\supp}{supp}

\newcommand{\X}{\mathcal{X}}  
\newcommand{\Y}{\mathcal{Y}}

\renewcommand{\tilde}{\widetilde}
\renewcommand{\div}{\divg}

\begin{document}

\title[Lipschitz regularity for equations with random coefficients]{Lipschitz regularity for elliptic equations with random coefficients}

\begin{abstract}
We develop a higher regularity theory for general quasilinear elliptic equations and systems in divergence form with random coefficients. The main result is a large-scale $L^\infty$-type estimate for the gradient of a solution. The estimate is proved with optimal stochastic integrability under a one-parameter family of mixing assumptions, allowing for very weak mixing with non-integrable correlations to very strong mixing (e.g., finite range of dependence). We also prove a  quenched $L^2$ estimate for the error in homogenization of Dirichlet problems. The approach is based on subadditive arguments which rely on a variational formulation of general quasilinear divergence-form equations.
\end{abstract}

\author[S.N. Armstrong]{Scott N. Armstrong}
\address[S.N. Armstrong]{Ceremade (UMR CNRS 7534), Universit\'e Paris-Dauphine, Paris, France}
\email{armstrong@ceremade.dauphine.fr}

\author[J.-C. Mourrat]{Jean-Christophe Mourrat}
\address[J.-C. Mourrat]{Ecole normale sup\'erieure de Lyon, CNRS, Lyon, France}
\email{jean-christophe.mourrat@ens-lyon.fr}

\date{\today}

\maketitle

\section{Introduction}

\subsection{Informal summary of the main results}
We are interested in the stochastic homogenization of divergence-form elliptic equations and systems which take the general quasilinear form
\begin{equation} \label{e.pde}
- \nabla \cdot \left( \a\!\left(\nabla u(x), x \right) \right) = 0  \quad \mbox{in} \ U \subseteq \Rd.
\end{equation}
The precise assumptions are stated in detail below, but let us mention here that the map $\a:\Rd \times\Rd \to \R$ is uniformly monotone and Lipschitz in its first argument. In addition,~$x \mapsto \a(\, \cdot \, ,x)$ is a stationary random field satisfying a quantitative ergodic assumption. 

\smallskip

In this paper, we prove a ``quenched" Lipschitz estimate for solutions of~\eqref{e.pde} which is optimal in terms of the stochastic integrability of the random variables appearing in the estimates. While the results are new for  linear equations and even under strongly mixing conditions (see however the discussion of~\cite{AS} and~\cite{GNO3} below), they hold for the most general nonlinear divergence-form elliptic operators and under quite weak mixing conditions. We are primarily motivated by the fact that strong gradient bounds play a central role in the quantitative theory of stochastic homogenization for elliptic equations. This is because the magnitude of the gradient of a solution controls how sensitively it depends on the coefficients -- and it turns out that good estimates of the latter, combined with appropriate concentration inequalities, yield optimal quantitative bounds in homogenization. This was shown for linear equations by~Gloria and Otto~\cite{GO1,GO2,GO3} and Gloria, Neukamm and Otto~\cite{GNO2,GNO}. The work in this paper is the crucial step to developing an optimal quantitative theory of stochastic homogenization for general quasilinear equations and systems in divergence form, as well as for improving the stochastic integrability of the current linear theory, as we will show in a future work. 

\smallskip

To give a better idea of the gradient estimate we prove, we recall that a solution $u$ of the linear equation
\begin{equation*} \label{}
-\nabla \cdot \left( A(x)\nabla u\right) = \nabla \cdot \mathbf{f}(x) \quad \mbox{in} \ B_1,
\end{equation*}
under the assumption that the coefficient matrix $A(x)$ and vector field $\mathbf{f}(x)$ are H\"older continuous, satisfies the following pointwise  bound:
\begin{equation} \label{e.modeleq}
\left| \nabla u(0) \right|^2 \leq C\left( 1+ \fint_{B_1} \left| \nabla u(x) \right|^2\,dx \right),
\end{equation}
where the constant $C$ depends in particular on the regularity of the coefficients. Of course, the assumption of H\"older continuity of $A$ is crucial, as the best regularity for general equations with rapidly oscillating coefficients is $W^{1,2+\ep}$ (Meyers' estimate, for Sobolev regularity) and $C^{0,\ep}$ (De Giorgi-Nash-Moser, for H\"older regularity) for some $\ep>0$ depending on the ellipticity. Therefore the estimate~\eqref{e.modeleq} is not scale invariant, and in particular at very large scales (i.e., with the ball $B_1$ replaced by $B_R$ with radius $R\gg1$) the estimate is false, in general, even if $A$ is smooth. 

\smallskip

The primary purpose of this paper is to show that, nevertheless, an equation \emph{with random coefficients} has better regularity than a general equation, and we can in fact prove that, due to statistical effects, an estimate like~\eqref{e.modeleq} holds on large scales. The main result, stated precisely below in Theorem~\ref{t.lip}, states roughly that, under appropriate mixing conditions on the coefficients, for any $R\gg1$, a solution~$u$ of~\eqref{e.pde} in a domain $U\supseteq B_{R}$ satisfies
\begin{equation} \label{e.introLip}
\fint_{B_r} \left| \nabla u(x) \right|^2\,dx \leq C\left( 1 + \fint_{B_{R}} \left| \nabla u(x) \right|^2\,dx \right) \quad \mbox{for every} \ r \in \left[ \X, \frac12 R\right].
\end{equation}
The main difference between~\eqref{e.modeleq} and~\eqref{e.introLip} is that the latter estimate holds only for balls with radii larger than a random ``minimal radius" $\X$, while the former holds in every ball (hence pointwise). The random variable~$\X$ is almost surely finite -- but not bounded -- and thus the central task is to estimate the probability that $\X$ is very large. That is, we would like to specify which of the stochastic moments of~$\X$ are bounded under various quantitative ``mixing" assumptions on the coefficients. In this paper, we prove~\eqref{e.introLip}, with optimal stochastic integrability estimates on~$\X$, under a continuum of mixing conditions on the coefficients, ranging from relatively weak mixing (allowing for non-integrable correlations) to very strong assumptions. For instance, if the coefficients satisfy a finite range of dependence condition, then we have the bound
\begin{equation} \label{e.opt}
\E\left[ \exp\left( \X^s \right) \right] < \infty \quad \mbox{for every} \ s<d,
\end{equation}
which is optimal in that such an estimate is certainly false for $s>d$. On the other hand, if the coefficients satisfy a weaker ``$\alpha$--mixing" condition with an algebraic rate given by an exponent~$\beta>0$ (see~(P3) in the next subsection for the precise statement of this assumption), then we give the (optimal) bound
\begin{equation} \label{e.opt2}
\E\left[ \X^\theta \right] < \infty \quad \mbox{for every} \ \theta<\beta.
\end{equation}

\smallskip

We call~\eqref{e.introLip} a ``quenched" estimate because the random variable $\X$ does not depend on the solution~$u$, and thus in particular on what random dependence $u$ may (or may not) possess. We call it a ``Lipschitz" estimate because~$R$ is very large and in this scaling the microscopic scale is $O(1)$, and thus we think of the left side as approximating $|\nabla u(0)|$. Notice that~\eqref{e.introLip} implies
\begin{equation*} \label{}
\fint_{B_1} \left| \nabla u(x) \right|^2\,dx \leq C \X^d\left( 1 + \fint_{B_{R}} \left| \nabla u(x) \right|^2\,dx \right),
\end{equation*}
so we may also think of $\X$ as a random prefactor constant in an estimate with deterministic scales. 
Under the additional assumption that the coefficients are smooth on the microscopic scale, the left side of the latter inequality controls~$|\nabla u(0)|$ (by the nonlinear version of the Schauder estimate~\eqref{e.modeleq}, in fact) and, in that case, we then obtain a true pointwise Lipschitz estimate at the origin. However, we prefer to think of control over the very small scales as a separate issue. Indeed, questions about the behavior of solutions on scales smaller than the microscopic scale are outside of the realm of homogenization, as they cannot be due to statistical effects. An estimate like~\eqref{e.introLip} is the strongest possible estimate for the control of microscopic-scale fluctuations in terms of large-scale fluctuations for solutions of~\eqref{e.pde}; higher regularity such as $C^{1,\alpha}$ cannot hold at large scales without some reinterpretation of the meaning of such an estimate. In fact, for equations with periodic coefficients (a special case of what we consider here), there cannot exist uniform bounds on the modulus of continuity of the gradient unless the coefficients are constant.

\smallskip

While we use scalar notation throughout the paper, we only use arguments which apply to systems under the assumption that the constant-coefficient homogenized system admits $C^{1,\alpha}$ estimates (which is a necessary condition, satisfied both for linear systems as well as in the nonlinear scalar case, for example). Moreover, for the error estimate presented in Theorem~\ref{t.subopt}, which is of independent interest, this assumption is not required. Therefore, the methods in this paper are quite flexible in their application to general elliptic systems.

\smallskip

Finally, we point out that the arguments leading to the Lipschitz estimate (Theorem~\ref{t.lip}) apply with only very minor modifications to yield Lipschitz estimates up to the boundary for Dirichlet problems (even with possibly oscillating boundary data).

\subsection{Motivation and previous works}
There has been much recent interest in quantitative results for the stochastic homogenization of uniformly elliptic equations, prompted by the striking results of Gloria and Otto~\cite{GO1,GO2}. Partially inspired by the previous works of Naddaf and Spencer~\cite{NS} and Conlon and Naddaf~\cite{CN}, they proved optimal estimates on the typical size of the fluctuations of approximate correctors and their energy density for linear elliptic equations under strong independence assumptions on the coefficients. Expounding on this work and some ideas developed independently in~\cite{M}, these authors and Neukamm~\cite{GNO2,GNO} then proved optimal estimates and a quantitative two-scale expansion for uniformly elliptic equations under quantitative assumptions built on spectral gap-type concentration inequalities. 

\smallskip

At the analytic heart of all of these results are estimates on large stochastic moments of the gradient of the approximate correctors as well as the gradient of the Green's functions. Later, much stronger bounds on the Green's functions were obtained by Marahrens and Otto~\cite{MO} for coefficient fields satisfying a logarithmic Sobolev inequality. They also proved the first ``quenched" result: a $C^{0,\alpha}$ estimate for $0<\alpha< 1$, with bounded (stochastic) moments (see also Gloria and Marahrens~\cite{GM}). These estimates, and the latter result especially, suggested the possibility of developing a large scale Lipschitz regularity theory for random elliptic equations, which could potentially not only generalize these results but substantially improve the quantitative theory of stochastic homogenization.

\smallskip

The first quenched Lipschitz estimate for random elliptic equations was proved recently by the first author and Smart~\cite{AS}. This result was obtained for solutions of variational quasilinear equations under the assumption that the coefficients satisfy a finite range of dependence, and gave optimal estimates on the stochastic integrability of the random variable~$\X$ in~\eqref{e.introLip}, that is, they obtained~\eqref{e.opt}. 

\smallskip

There were two separate ideas central to the arguments of~\cite{AS}, which differed substantially from those of the previous works mentioned above: the first was, following the philosophy of Avellaneda and Lin~\cite{AL1} in the context of periodic homogenization, to ``borrow" the higher regularity of the constant-coefficient homogenized equation. The technique is to prove a Lipschitz estimate by implementing a Campanato-type~$C^{1,\alpha}$ iteration and comparing, at each dyadic scale, the heterogeneous solution to the homogenized one with the same boundary values. Since the latter is more regular, the ``flatness" of the solution (the normalized $L^2$ difference between the solution and the nearest affine function) improves at each step of the iteration by a universal factor less than one. It is necessary to keep track of the error made by substituting the solution for that of the homogenized equation at each iteration. It was shown in~\cite[Lemma 5.1]{AS} that one obtains a Lipschitz bound on scales larger than the minimal scale above which this error in homogenization is controlled by a fixed \emph{algebraic} (or at least Dini)  modulus. This quantitative variation of the compactness argument of~\cite{AL1} does not depend on the precise form of the equation (in particular, it does not require the equation to be variational) and thus in certain cases it may give more information than the compactness arguments of~\cite{AL1} (for example, in a recent preprint, the first author and Shen~\cite{ASh} used the idea to generalize the results of~\cite{AL1,AL2} as well as those of Kenig, Lin and Shen~\cite{KLS2,KLS1} to certain linear systems with almost periodic coefficients).
 
\smallskip
 
The key step in quantifying the stochastic moments of~$\X$ in the Lipschitz estimate, if we follow this method, is to obtain a quantitative estimate on the error in homogenization for the Dirichlet problem which is strong in stochastic integrability: here the estimate may be suboptimal in the typical size of the error (provided it is algebraic or Dini) in order to bound the probability of deviations as strongly as possible.

\smallskip

This brings us to the second central theme of~\cite{AS}, which is that the method of subadditivity, introduced in stochastic homogenization by Dal Maso and Modica~\cite{DM1,DM2}, is very well-suited for obtaining estimates which (although suboptimal in the typical size of the error) are strong in stochastic integrability. It is here that variational techniques play a central role, as the natural subadditive quantities for analyzing the problem arise from the variational characterization of the equation. Therefore, the results in~\cite{AS} were stated for local minimizers of uniformly convex integral functionals, which is equivalent to the special case of~\eqref{e.pde} in which~$\a(\cdot,x)$ is the gradient of a uniformly convex function. If the equation is linear (i.e.,~$\a(p,x) = A(x)p$ for a square matrix $A$ with positive eigenvalues), then this reduces to the requirement that~$A$ be symmetric.

\smallskip

More recently, Gloria, Neukamm and Otto~\cite{GNO3}, partially inspired by~\cite{AS}, obtained a quenched Lipschitz estimate for linear equations and systems. In particular, they give the first such estimate for linear equations and systems with nonsymmetric coefficients, although their techniques do not seem to generalize in a straightforward way to nonlinear equations. Their arguments use a similar scheme based on a Campanato iteration and approximation of the Dirichlet problem for the heterogeneous equation by the homogenized problem. They introduce a random variable $r_*$ which is related to the quantity we denote by~$\X$ (and the analogous random variable in~\cite{AS}) but is defined using harmonic coordinates, and is therefore perhaps more intuitive and geometric. It also allows to establish qualitative results (without imposing any mixing conditions) and gives control of higher regularity of the solutions after subtracting the corrector (i.e., the regularity of the error in the two-scale expansion). 

\smallskip

The arguments in~\cite{GNO3} and~\cite{AS} are very different in how the integrability of~$\X$ or~$r_*$ is obtained (that is, in how the error in homogenization of the Dirichlet problem is quantified in each step of the Campanato iteration), and those of~\cite{GNO3} so far yield suboptimal estimates of the stochastic integrability of~$\X$ under strong mixing conditions (for example, under independence assumptions, they get~\eqref{e.opt} for some $s > 0$ rather than every $s<d$).

\smallskip

We conclude this subsection by mentioning two more recent related works: both Ben-Artzi, Marahrens and Neukamm~\cite{BMN} and Bella and Otto~\cite{BO} obtained moment bounds on the gradient of the approximate correctors for linear elliptic systems using concentration inequalities, and Lamacz, Neukamm and Otto~\cite{LNO} obtained such estimates for a percolation model using similar methods. 

\subsection{An overview of the approach}
Before stating the main results, it is necessary to give a brief overview of the ideas leading to the proofs (since the latter affects the former).

\smallskip

First, the Campanato scheme for proving a Lipschitz estimate from~\cite{AS} reduces our task to that of obtaining an estimate for the homogenization error for the Dirichlet problem which is optimal in stochastic integrability. This estimate, stated below in Theorem~\ref{t.subopt}, is of independent interest as it is the first quantitative result in stochastic homogenization for equations taking the general quasilinear form~\eqref{e.pde} and, except for the result in~\cite{AS} that it generalizes, the only quantitative result in stochastic homogenization for nonlinear elliptic equations in divergence form. The proof of the error estimate is an adaption of the variational, subadditive approach of~\cite{AS}, although here we do not have the same variational structure. We instead start from the fact that solutions of
\begin{equation}
\label{e.genquasi}
-\nabla \cdot \left( \a\left( \nabla u,x \right) \right) = f \quad \mbox{in} \ U
\end{equation}
may be characterized as the null minimizers in $H^1(U)$ of a nonnegative functional taking the form
\begin{multline} \label{e.fun}
w\mapsto \J \left[ w,f \right] := \inf\bigg\{ \int_U \left( F(\nabla w(x),\mathbf{g}(x),x) -\nabla w(x) \cdot \g(x) \right)\,dx \,: \\ 
 \mathbf{g} \in L^2(U;\Rd), \ -\nabla \cdot \mathbf{g} = f  \bigg\},
\end{multline}
where the function~$F(p,q,x)$ in the integrand is uniformly convex in~$(p,q)$. Moreover, if~$u$ is a null minimizer of~ $\J[\, \cdot \ , f]$, then the corresponding vector field~$\g$ which attains the infimum is the flux, that is, $\g = \a(\nabla u,x)$, and the integrand vanishes identically. This variational formulation does not coincide with the classical one in the case that $\a$ is the gradient of a Lagrangian.

\smallskip

This more general variational principle first appeared in special cases in the work of Brezis and Ekeland~\cite{BE} and Neyroles~\cite{N}. Later, major progress was made by~Fitzpatrick~\cite{F}, who showed that general maximal monotone maps admit a variational characterization. Variational formulations of monotone maps were subsequently explored in some details by Ghoussoub and his collaborators (see~\cite{Gh} and the references therein). In Section~\ref{s.variational}, we review the connection between variational problems of this form and general quasilinear, divergence-form elliptic equations.

\smallskip

We consider the functional $\J\left[\cdot,\cdot\right]$ to be the fundamental object of study in this paper and we seldom refer to the PDE or the coefficients~$\a$. Thus, rather than looking to identify effective coefficients~$\overline\a$ directly, we search for an effective variational problem characterized by an effective integrand $\overline F$. The latter, as we will see, may be defined by the ($\P$--almost sure) limit
\begin{equation*}
\overline F(p,q) = \lim_{n\to\infty} \ \inf_{u \in H^1_0(\cu_n), \, \g \in \Lso(\cu_n)} \  \fint_{\cu_n} F\left( p + \nabla u(x), q + \g(x),x \right)\,dx.
\end{equation*}
Here $\cu_n$ is the cube centered at the origin with side length $3^n$, and $\Lso(U)$ denotes the set of solenoidal (i.e.,~divegence-free) $L^2$ vector fields in $U$ with normal component vanishing on the boundary of $U$.

\smallskip

In the context of periodic homogenization, such an approach to homogenization has already been explored by Ghoussoub, Moameni and S\'aiz~\cite{GMS}, who gave an analogous formula for $\overline F$ and a qualitative proof of homogenization, as well as by Visintin~\cite{V2,V1}. Since the (normalized) energy 
\begin{equation*}
\mu_0(U,p,q) := \inf\left\{ \fint_{U} F\left( p + \nabla u(x), q + \g(x),x \right)\,dx \,:\, u \in H^1_0(U), \, \g \in \Lso(U) \right\}
\end{equation*}
of the minimizing pair with given affine boundary data is a naturally monotone/subadditive quantity, the argument of Dal Maso and Modica~\cite{DM2} naturally extends to this setting to give a qualitative proof of homogenization in the general (stationary-ergodic) random setting as well.

\smallskip

Here however, our interest is in quantitative results, and for this $\mu_0$ is only half of the picture. We also need the naturally superadditive quantity, modeled on the analogous one from~\cite{AS}, which we denote by
\begin{multline*}
\mu(U,q^*,p^*) := \inf \left\{ \fint_{U} \left( F\left( \nabla u(x), \g(x),x \right) - q^*\cdot \nabla u(x) - p^*\cdot \g(x) \right) \,dx \,: \right. \\
\left. \phantom{\fint_U} u \in H^1(U), \, \g \in \Ls(U) \right\}.
\end{multline*}
Notice that $\mu$ imposes no boundary conditions on $u$ and $\g$, which is what gives it the property of superadditivity. As we will see, the large scale limit of $-\mu$ is the convex dual (Legendre-Fenchel transform) of that of $\mu_0$, and the variables $(p,q)$ are dual to $(q^*,p^*)$. The fact that $\mu_0$ and $\mu$ are monotone from different sides will allow us to ``trap" the limiting behavior between these two extremes, in effect  revealing the additive structure of the problem which renders a more precise quantitative analysis possible.

\subsection{Statement of hypotheses and main results}
As mentioned in the previous section, we consider the variational problem associated to~\eqref{e.pde} to be the fundamental object of study, and therefore we state our results in terms of null minimizers of the functional $\J[\, \cdot \, , f]$ rather than solutions of~\eqref{e.genquasi}. The link between the variational problem and the PDE is studied in detail in Section~\ref{s.variational}, where it becomes clear that our assumptions cover the case of~\eqref{e.pde} in full generality. 

\smallskip

For the rest of the paper, we assume that $d\geq 2$ and fix the parameters $K_0\geq 0$ and  $\Lambda \geq 3$.
We consider coefficient fields~$F=F(p,q,x)$ satisfying 
\begin{equation} \label{e.lebemeas}
F \in L^\infty_{\mathrm{loc}}(\Rd \times \Rd \times \Rd)
\end{equation}
such that $(p,q) \mapsto F(p,q,x)$ is uniformly convex, that is,   
\begin{equation}\label{e.UC}
(p,q) \mapsto F(p,q,x) - \frac1{2\Lambda}\left( |p|^2+|q|^2\right) \quad \mbox{is convex}
\end{equation}
as well as uniformly $C^{1,1}$, uniformly in $x$; in view of~\eqref{e.UC}, it suffices to assume 
\begin{equation} \label{e.upperUC}
(p,q) \mapsto F(p,q,x) -\frac\Lambda{2} \left( |p|^2+|q|^2\right) \quad \mbox{is concave.}
\end{equation}
We require also that, for every $p,x\in\Rd$, 
\begin{equation} \label{e.infFpq}
\inf_{q\in\Rd} \left(F(p,q,x) - p\cdot q \right) = 0.
\end{equation}
We define the set of all coefficient fields by 
\begin{equation} 
\label{e:def:Omega}
\Omega : = \left\{ F\,:\, \mbox{$F$ satisfies~\eqref{e.lebemeas},~\eqref{e.UC},~\eqref{e.upperUC} and~\eqref{e.infFpq}}  \right\}.
\end{equation}
We endow $\Omega$ with the translation group $\{T_y\}_{y\in\Rd}$, which acts on $\Omega$ via
\begin{equation*}
(T_zF)(p,q,x) := F(p,q,x+z),
\end{equation*}
and the family $\{ \F_U\}$ of $\sigma$--algebras, with $\F_U$ defined for each Borel $U\subseteq \Rd$ by
\begin{multline*} \label{}
\F_U := \mbox{$\sigma$--algebra on $\Omega$ generated by the family of maps} \\
F \mapsto \int_{U} F(p,q,x) \varphi(x) \,dx, \quad p,q \in\Rd, \ \varphi\in C^\infty_c(\Rd).
\end{multline*}
The largest of these $\sigma$--algebras is denoted by $\F:=\F_{\Rd}$. The translation group also acts naturally on $\F$ via the definition
\begin{equation*}
T_zA:= \left\{ T_zF\,:\, F\in A \right\}, \quad A\in\F.
\end{equation*}
Associated to each $F\in\Omega$ and bounded open $U\subseteq \Rd$ is the functional
\begin{equation*}
\J:H^1(U)\times H^{-1}(U) \to [0,\infty),
\end{equation*}
defined for each $(u,u^*) \in H^1(U)\times H^{-1}(U)$ by
\begin{multline} \label{e.J}
\J\!\left[ u,u^* \right] := \inf\bigg\{ \int_{U}
\left( F\left( \nabla u(x),\g(x) ,x\right) -\nabla u(x) \cdot \g(x)\right) \,dx \\
:\, \g\in L^2(U;\Rd), \ -\nabla \cdot \g = u^* \bigg\},
\end{multline}
where the condition $-\nabla \cdot \g = u^*$ is to be understood in the sense that
\begin{equation*} \label{}
\int_{U} \g(x) \cdot \nabla \phi(x) \,dx = \left\langle \phi, u^*\right\rangle \quad \mbox{for every} \ \phi\in H^1_0(U)
\end{equation*}
and $\langle\cdot,\cdot \rangle$ denotes the canonical pairing between $H^1_0(U)$ and $H^{-1}(U)$. Although~$\J$ depends on both~$F$ and~$U$, we keep this implicit in our notation as the identities of~$F$ and~$U$ are always clear from context.

\smallskip

We consider a probability measure~$\P$ on $(\Omega,\F)$ which is assumed to satisfy the following three conditions:
\begin{enumerate}
\item[(P1)] The random field $F$ is uniformly bounded below on the support of~$\P$, in the sense that
\begin{equation*} \label{}
\P \left[ \essinf_{x\in\Rd} \inf_{p,q\in\Rd}  F(p,q,x) \geq - K_0^2 \right] = 1.
\end{equation*}

\smallskip

\item[(P2)] $\P$ is stationary with respect to $\Zd$--translations: for every $z\in \Zd$ and $A\in \F$,
\begin{equation*} 
\P \left[ A \right] = \P \left[ T_z A \right].
\end{equation*}

\smallskip

\item[(P3)] $\P$ is ``$\alpha$--mixing" with an algebraic decorrelation rate: there exist an exponent $\beta >0$ and a constant $C_3> 0$ such that, for all Borel subsets $U,V\subseteq \Rd$, $A\in \F_U$ and $B \in \F_V$, we have 
\begin{equation*}
\quad \left| \P [ A \cap B ]- \P[A ] \, \P[B ] \right| \leq C_3 \left( 1+\dist(U,V) \right)^{-\beta}.
\end{equation*}
\end{enumerate}

The assumption~(P3) is a relatively weak mixing condition. The conclusions of the main result can be considerably strengthened, with minor modifications to the arguments, if one is willing to assume stronger mixing conditions. To make this point explicitly, we also state results under an $\alpha$--mixing condition with a faster rate of decorrelation:

\begin{enumerate}
\item[(P4)] $\P$ is ``$\alpha$--mixing" with a stretched exponential decorrelation rate: there exist an exponent $\gamma >0$ and a constant $C_4> 0$ such that, for all Borel subsets $U,V\subseteq \Rd$, $A\in \F_U$ and $B \in \F_V$, we have 
\begin{equation*}
\quad \left| \P [ A \cap B ]- \P[A ] \,\P[B ] \right| \leq C_4 \exp \left( -\dist(U,V)^{\gamma} \right).
\end{equation*}
\end{enumerate}
Notice that a finite range of dependence hypothesis on the probability measure~$\P$ (as assumed for instance in~\cite{AS})  implies~(P4) for every $\gamma<\infty$.

\smallskip

We emphasize that assumptions (P1)-(P3) are assumed to be in force throughout the paper. This is not the case for~(P4), which is only assumed to be in force where explicitly stated.

\smallskip

We now present the statement of the main result. 

\begin{theorem}[Lipschitz regularity]
\label{t.lip}
Assume that $\P$ satisfies~(P1),~(P2),~(P3). Fix $p>d$ and $\theta\in (0,\beta)$. There exist $C(d,\Lambda,\beta,C_3,p,\theta)\geq 1$ and a random variable~$\X \geq 1$ satisfying
\begin{equation}\label{e.weakSI}
\E \left[ \X^\theta \right] \leq C 
\end{equation}
such that the following holds: if $R\geq 1$, $f\in L^2(B_R)$ and $u\in H^1(B_R)$ satisfy
\begin{equation}\label{e.localminimizer}
\J \left[ u ,f \right]  = 0,
\end{equation}
and if we define
\begin{equation}\label{e.Mcondition}
M:=K_0 + \frac1{R}\inf_{a\in\R} \left( \fint_{B_R} \left| u(x) -a\right|^2\,dx\right)^{\frac12} + R\left( \fint_{B_R} \left|f(x) \right|^p\,dx \right)^{\frac1p} ,
\end{equation}
then we have the estimate
\begin{equation}\label{e.lipschitzest}
 \fint_{B_r} \left| \nabla u(x) \right|^2\,dx  \leq CM^2 \quad \mbox{for every} \quad \X (1+M)^{\frac{2d}{\theta}} \leq r \leq \frac12R.
\end{equation}
Moreover, if~$\P$ satisfies~(P4), then~\eqref{e.weakSI} may be improved to the statement that, for every $s\in(0,d\gamma/(d+\gamma) )$,
\begin{equation}\label{e.strongSI}
\E \left[ \exp\left( \X^s \right) \right] \leq C,
\end{equation}
where the constant $C$ depends additionally on $(\gamma,C_4,s)$, and~\eqref{e.lipschitzest} to the statement that
\begin{equation}\label{e.lipschitzest2}
 \fint_{B_r} \left| \nabla u(x) \right|^2\,dx  \leq CM^2 \quad \mbox{for every} \quad \X \log (1+M) \leq r \leq \frac12R.
\end{equation}
\end{theorem}

\begin{remark}
The condition~\eqref{e.localminimizer} says that $u$ is a null minimizer (on~$B_R$) of the functional $\J\left[\cdot,f \right]$. According to~Proposition~\ref{p.yesvariational}, this is equivalent to the statement that~$u$ is a solution of
\begin{equation} \label{e.PDEf}
-\nabla \cdot \left( \a(\nabla u(x),x) \right) = f(x) \quad \mbox{in} \ B_R,
\end{equation}
where~$\a$ is the uniformly monotone vector field represented by $F$ in the sense defined in Section~\ref{s.variational}. Moreover, according to Theorem~\ref{t:rep}, any uniformly monotone vector field~$\a$ admits a variational representation satisfying our assumptions, and thus there is no loss of generality in the variational formulation compared to the PDE itself. In particular, the conclusion of Theorem~\ref{t.lip} applies to solutions of equations of the form~\eqref{e.PDEf} with random coefficients which are uniformly monotone, Lipschitz and bounded. (It is obvious how to translate hypotheses~(P2) and~(P3), since these are essentially unchanged; the proper translation of~(P1) into a statement about $\a(p,x)$ is indicated in Lemma~\ref{l.K0stuff}, below.) 
\end{remark}

\begin{remark}
The integrability of $\X$ given in Theorem~\ref{t.lip} is optimal under each of assumptions~(P3) and~(P4) and for every value of the exponents~$\beta$ and~$\gamma$. This can be verified by the construction of explicit examples. Observe that if~$\P$ satisfies a finite range of dependence, then we may take any $\gamma<\infty$ in~(P4) and we thus get~\eqref{e.strongSI} for every exponent~$s<d$, which matches the result in~\cite{AS}. This is certainly optimal, because the probability of obtaining any \emph{particular} coefficient field (for the random checkerboard, say) in a ball of radius~$R$ is of order $\exp(-cR^d)$. 

\smallskip

While we do not give further details here, we also obtain appropriate bounds on~$\X$ under mixing conditions other than $\alpha$--mixing. For instance, if~$\P$ is assumed to satisfy a spectral gap inequality, then we obtain~\eqref{e.strongSI} for every $s<d/2$ and, under a logarithmic Sobolev inequality, for every $s<d$ (we also expect both of these to be optimal).
\end{remark}

\begin{remark}
The parameter $M$ in the lower bound for $r$ in~\eqref{e.lipschitzest} and~\eqref{e.lipschitzest2} may be removed in the case that $F$ is positively homogeneous of order two (e.g., in the case that the corresponding PDE is linear). The reason it arises in the general setting has to do with the fact that one may not obtain full information regarding the random fields $\{ F(p,q,\cdot) \}_{p,q\in\Rd}$ by looking at $(p,q)$ lying in a bounded set, and thus the random scale above which the Lipschitz estimate holds necessarily exhibits some dependence on the size of the solution. See also Remark~\ref{r.weirdness} below.
\end{remark}

\subsection{Outline of the paper}
In Section~\ref{s.variational}, we review the variational characterization of general divergence-form elliptic equations. In Section~\ref{s.lipschitz}, we explain the general scheme for obtaining the Lipschitz estimate and reduce the main result to an estimate on the error in homogenization for the Dirichlet problem. The rest of the paper is focused on obtaining the latter result, stated in Theorem~\ref{t.subopt}. We introduce the subadditive and superadditive quantities in Section~\ref{s.subadditive}, and review some of their basic properties. In Section~\ref{s.proof}, we reduce Theorem~\ref{t.subopt} to two ingredients: one is probabilistic and gives a quantitative estimate for the convergence of the subadditive and superadditive quantities, while the other is a deterministic statement that asserts that these quantities control Dirichlet problems with arbitrary boundary data. The proofs of these results are given in Sections~\ref{s.mu} and~\ref{s.blackbox}, respectively, which are the analytic core of the paper. We also discuss how to construct~$\overline\a$ from $\overline F$ in Section~\ref{s.variationalbar}. In Appendix~\ref{s.mixing}, we put the mixing conditions into a convenient form for their use in Section~\ref{s.mu}. In Appendix~\ref{s.reg}, we give the proofs of some regularity estimates (such as the Caccioppoli and Meyers estimates) needed at several points in the paper, which although well-known, we nevertheless could not find in the literature in the generality required here. 

\section{Variational formulation of divergence-form elliptic equations}
\label{s.variational}

At the core of the proof of Theorem~\ref{t.lip} is a subadditivity argument which generalizes the one used in~\cite{AS}. The purpose of this section is to set up this argument by first formulating the PDE as a variational problem. This section is thus somewhat independent of the rest of the paper, and we also prove a new result (Theorem~\ref{t:rep} below), which is of independent interest. Apart from this result, the material covered here is well-known and can be found for example in \cite{F,Gh,Penot,V1} and is included for the convenience of the reader.

\smallskip

\subsection{Variational characterization of monotone maps}
The variational formulation of maximal monotone operators, which was pioneered by Fitzpatrick~\cite{F}, makes it possible to interpret the quasilinear PDE~\eqref{e.pde} as a minimization problem for an integral functional. We now review the basics of this (relatively recently developed) theory. 

\smallskip

In this subsection, we drop $x$-dependence and consider a Lipschitz, uniformly monotone vector field $\a:\Rd \to \Rd$. That is, we assume that~$\a$ satisfies, for some constant $\lambda\geq 1$,
\begin{equation} \label{e.unifmonolip}
\left\{ \begin{aligned} 
& \left| \a(p_1) - \a(p_2) \right| \leq \lambda \left| p_1-p_2 \right|, \\
& \left( \a(p_1) - \a(p_2) \right) \cdot(p_1-p_2) \geq \frac1{\lambda}\left|p_1-p_2\right|^2.
\end{aligned} \right.
\end{equation}

We next recall that~\eqref{e.unifmonolip} implies that~$\a$ is \emph{maximal monotone}.

\begin{lemma}
\label{l.maxmono}
The map $\a$ is maximal monotone, that is, for every $p,q\in\Rd$, 
\begin{equation} \label{e.maxmono}
\inf_{\xi\in\Rd} \left( \a(\xi) - q \right) \cdot(\xi-p) \geq 0  \ \iff \ q = \a(p).
\end{equation}
\end{lemma}
\begin{proof}
Suppose the first statement in~\eqref{e.maxmono} holds. Fix $h\in\Rd$ and $\delta > 0$ and take $\xi = p +\delta h$ to obtain 
\begin{align*} \label{}
0 \leq \left( \a(\xi) - q \right) \cdot(\xi-p) & = \left( \a(p) - q \right) \cdot(\xi-p) + \left( \a(\xi) - a(p) \right)\cdot (\xi-p) \\
& \leq \delta \left( \a(p) - q \right) \cdot h + \lambda \delta^2 |h|^2.
\end{align*}
Dividing by $\delta$ and passing to the limit $\delta \to 0$, we obtain 
\begin{equation*} \label{}
\left( \a(p) - q \right) \cdot h \geq 0.
\end{equation*}
This holds for arbitrary $h\in\Rd$, thus $\a(p) = q$. This gives one implication of~\eqref{e.maxmono}, and the other is obvious. 
\end{proof}
\begin{remark}
\label{r:cond-l.maxmono}
An inspection of the proof of Lemma~\ref{l.maxmono} reveals that, for $\a$ to be maximal monotone, it suffices that $\a$ be merely monotone and continuous.
\end{remark}

According to \cite[Proposition~12.54]{RW} (or the Browder-Minty theorem), $\a$ is surjective. It then follows that the inverse $\a^{-1}$ of $\a$ is also Lipschitz and uniformly monotone (with the same constants), that is, for every $q_1,q_2\in\Rd$,
\begin{equation} \label{e.unifmonolipinv}
\left\{ \begin{aligned} 
& \left| \a^{-1}(q_1) - \a^{-1}(q_2) \right| \leq  \lambda\left| q_1-q_2 \right|, \\
& \left( \a^{-1}(q_1) - \a^{-1}(q_2) \right) \cdot(q_1-q_2) \geq \frac1\lambda \left|q_1-q_2\right|^2.
\end{aligned} \right.
\end{equation}

We look for a variational representation of the map $\a$ in the following sense.

\begin{definition}
\label{d:a-rep}
Let $F: \Rd \times \Rd \to \R \cup \{+\infty\}$ be a convex function.
We say that $F$ \emph{(variationally) represents} the monotone vector field $\a: \Rd \to \Rd$ if the following two conditions hold for every $(p,q) \in \Rd \times \Rd$:
\begin{enumerate}
\item[(i)] $F(p,q) \ge p\cdot q$,
\item[(ii)] $ F(p,q) = p \cdot q \quad \Leftrightarrow \quad q = \a(p)$.
\end{enumerate}
\end{definition}
Representatives of maximal monotone vector fields are not unique. It is sometimes desirable to find a representative that has the following additional property.
\begin{definition}
Let $F: \Rd \times \Rd \to \R \cup \{+\infty\}$ be a convex function, and let
\begin{equation}
\label{e:def:FL-two}
F^*(q^*,p^*) = \sup_{p,q\in\Rd} \, \big\{ (p,q)\cdot(q^*,p^*) - F(p,q)  \big\}
\end{equation}
be the Legendre-Fenchel transform of $F$ (where we write $ (p,q)\cdot(q^*,p^*) $ for $p\cdot q^* + q \cdot p^*$). We say that $F$ is \emph{self-dual} if $F^*(q,p) = F(p,q)$ for every $(p,q) \in \R^d \times \R^d$.
\end{definition}

We continue with some examples of representatives of monotone vector fields. 

\begin{example}[Gradients of convex functions]
Consider a uniformly convex function $L:\Rd \to \R$ satisfying
\begin{align}
\label{e:unif-conv}
\frac1{2\Lambda} |p_1-p_2|^2 \leq \frac12 L(p_1) + \frac12 L(p_2) - L\left(\frac12p_1+\frac12 p_2 \right) \leq \frac\Lambda2 |p_1-p_2|^2.
\end{align}
Then the uniformly monotone map given by
\begin{equation*}
\a(p) = \nabla L(p)
\end{equation*}
is represented by the self-dual function
$$
F(p,q) = L(p) + L^*(q),
$$
where $L^*$ is the Legendre-Fenchel transform of $L$. Moreover, property \eqref{e:unif-conv} ensures that $F$ is uniformly convex and $C^{1,1}$. 
\end{example}

\begin{example}[Linear modifications \cite{BE,Nay}]
\label{ex:BEN}
Suppose $\a_0$ is a monotone map represented by $F_0$ and let $M$ be a $d\times d$ matrix such that $p\cdot Mp \ge 0$ for every $p \in \R^d$. Then the monotone map
$$
\a(p) = \a_0(p) + Mp
$$
is represented by the function
$$
F(p,q) = F_0(p,q-Mp) + p \cdot Mp.
$$
Indeed, the fact that $F$ is convex is clear. We also have $$F(p,q) \ge p\cdot(q-Mp) + p\cdot Mp = p\cdot q,$$ with equality if and only if $q-Mp = \a_0(p)$, that is, $q = \a(p)$. Besides, if $F_0$ is uniformly convex and $C^{1,1}$, then so is $F$. In particular, if 
$$
\a(p) = \nabla L(p) + Mp
$$
with $M$ skew-symmetric, then $\a$ can be represented by the self-dual function
$$
F(p,q) = L(p) + L^*(q-Mp),
$$ 
and moreover, $F$ is uniformly convex and $C^{1,1}$ if $L$ satisfies \eqref{e:unif-conv}.

\smallskip

A fortiori, if $A$ is a symmetric positive definite matrix and $M$ is skew-symmetric, then the monotone map
$$
\a(p) = (A+M)p
$$
is represented by
\begin{equation}
\label{e.linearF}
F(p,q) = \frac{1}{2} \, p\cdot Ap + \frac{1}{2} \, (q-Mp) \cdot A^{-1} (q-Mp).
\end{equation}
Those readers interested only in the case that~\eqref{e.pde} is linear may skip the rest of this subsection and take~\eqref{e.linearF} to be the definition of~$F$ in the rest of the paper.
\end{example}

\begin{remark}
\label{r:switch}
For $G:\Rd \times \Rd\to\R \cup \{+\infty\}$, denote $G^{\T}(q,p) := G(p,q)$. The function~$G$ represents~$\a$ if and only if~$G^\T$ represents~$\a^{-1}$. 
\end{remark}

\begin{example}[Representing general maximal monotone maps]
Let $\a : \R^d \to \R^d$ be a maximal monotone map. The \emph{Fitzpatrick function} \cite{F} of~$\a$, defined by
\begin{equation}
\label{e:Fitz}
F(p,q) := \sup_{\xi\in\Rd} \, \big(q \cdot \xi - \a(\xi) \cdot (\xi - p) \big),
\end{equation}
represents~$\a$. Indeed, the facts that $F$ is convex and that $F(p,q) \ge p\cdot q$ for every $(p,q)$ are clear; the condition $F(p,q) \le p\cdot q$ is equivalent to 
$$
\inf_{\xi\in\Rd} \left( \a(\xi) - q \right) \cdot(\xi-p) \geq 0,
$$
which is equivalent to having $q = \a(p)$ since $\a$ is maximal monotone, see \eqref{e.maxmono}. As will be clear from Step 3 of the proof of Proposition~\ref{p:dualize} below, the Fitzpatrick function is the smallest representative of $\a$.
\end{example}

The previous example demonstrates that an arbitrary maximal monotone map is variationally representable. However, for Lipschitz, uniformly monotone maps, it is desirable (and necessary for our purposes) to find a representative that is uniformly convex and $C^{1,1}$. Unfortunately, the Fitzpatrick function does not satisfy this property in general -- in particular, it is never uniformly convex if the map is linear. This issue is resolved by the following theorem.

\begin{theorem}[Uniformly convex and $C^{1,1}$ representative]
\label{t:rep}
Assume $\a:\R^d \to \R^d$ satisfies~\eqref{e.unifmonolip} for some $\lambda \ge 1$. Then there exists a self-dual $F : \R^d \times \R^d \to \R$ that represents~$\a$ and satisfies
\begin{equation}
\label{e:rep-conv}
(p,q) \mapsto F(p,q) - \frac{1}{2 \Lambda}(|p|^2 + |q|^2)  \quad \text{is convex},
\end{equation}
\begin{equation}
\label{e:rep-C11}
(p,q) \mapsto F(p,q) - \frac{\Lambda}{2}(|p|^2 + |q|^2)  \quad \text{is concave},
\end{equation}
where we have set
\begin{equation}
\label{e:def:lambda}
\Lambda : = 2\lambda +1. 
\end{equation}
\end{theorem}

The proof of Theorem~\ref{t:rep} uses the following result.

\begin{proposition}
\label{p:dualize}
Let $\a: \R^d \to \R^d$ be a Lipschitz, uniformly monotone vector field, and let $F : \R^d \times \R^d \to \R \cup \{+\infty\}$ represent $\a$. The dual function $F^*$ represents $\a^{-1}$.
\end{proposition}
\begin{proof}
We decompose the proof into four steps. In the first step, we introduce a function $F_\infty$ on $\R^d \times \R^d$ that can be thought of as an indicator function of the graph of~$\a$ \cite{Penot}. It is designed to be larger than any representative of $\a$. We then note that the Fitzpatrick function introduced in \eqref{e:Fitz} is $F_\infty^{*\T}$. In the second step, we show that $F_\infty^{**}$ represents $\a$. In the third step, we show that $F_\infty^{*\T} \le F \le F_\infty$. In the last step, we deduce that $F_\infty^* \le F^* \le F_\infty^{**\T}$ and complete the proof. 

\smallskip

\emph{Step 1.} Let $F_\infty : \R^d \times \R^d \to \R \cup \{+\infty\}$ be the (possibly non-convex) function defined by
\begin{equation*}
F_\infty(p,q) := \left\{ 
\begin{aligned}
& p\cdot q && \text{if} \ q = \a(p), \\
& +\infty && \text{otherwise}.
\end{aligned}
\right.
\end{equation*}
It is clear that any $F$ that represents $\a$ must satisfy $F \leq F_\infty$. Moreover,
$$
F_\infty^{*}(q,p) = \sup_{\xi \in\Rd} \Ll( q\cdot \xi + a(\xi)\cdot p - a(\xi)\cdot \xi \Rr).
$$
In other words, $F_\infty^{*\T}$ is the Fitzpatrick function introduced in \eqref{e:Fitz}. Since the Fitzpatrick function represents $\a$, Remark~\ref{r:switch} implies that $F_\infty^{*}$ represents~$\a^{-1}$.

\smallskip

\emph{Step 2.} We now show that $F_\infty^{**}$ represents $\a$. The function is convex. Since $F_\infty^{**} \le F_\infty$, it is also clear that if $q = \a(p)$, then $F_\infty^{**}(p,q) \le F_\infty(p,q) = p\cdot q$. It thus suffices to demonstrate the following, for every $p,q\in\Rd$:
\begin{equation}
\label{e:step2-1}
F_\infty^{**}(p,q) \ge p\cdot q \quad \mbox{and} \quad q \neq \a(p) \Rightarrow F_\infty^{**}(p,q) > p\cdot q.
\end{equation}
Since the Fitzpatrick function $F_\infty^{*\T}$ represents $\a$, we have $F_\infty^{*\T} \le F_\infty$. By duality, $F_\infty^{*\T} \le F_\infty^{**}$. Using again that $F_\infty^{*\T}$ represents $\a$, we obtain \eqref{e:step2-1}.

\smallskip

\emph{Step 3.} We now show that 
\begin{equation}
\label{e:comp-step3}
F_\infty^{*\T} \le F \le F_\infty.
\end{equation}
The second inequality is clear since $F$ represents $\a$. For the first one, the convexity of $F$ ensures that for every $p,q,\xi \in \R^d$ and every $\eps > 0$,
$$
F(p,q) - F(\xi,\a(\xi)) \ge \eps^{-1} \Ll[ F\big(\xi+\eps(p-\xi),\a(\xi)+\eps(q-\a(\xi))\big) -F(\xi,\a(\xi)) \Rr].
$$
Since $F(p,q) \ge p\cdot q$ for every $(p,q)$, the latter quantity is larger than
$$
\a(\xi)\cdot(p-\xi) + (q-\a(\xi))\cdot\xi + O(\eps) \qquad (\eps \to 0).
$$
Sending $\eps\to 0$ and recalling that $F(\xi,\a(\xi)) = \a(\xi)\cdot \xi$, we arrive at
$$
F(p,q) \ge q\cdot \xi - \a(\xi) \cdot(\xi-p).
$$
Taking the supremum over $\xi \in \R^d$ gives $F \ge F_\infty^{*\T}$.

\smallskip

\emph{Step 4.} By duality, we deduce from \eqref{e:comp-step3} that $F_\infty^{*} \le F^* \le F_\infty^{**\T}$. We have seen in Steps 1 and 2 respectively that $F_\infty^*$ and $F_\infty^{**\T}$ represent $\a^{-1}$. So $F^*$ represents $\a^{-1}$ as well, and the proof is complete.
\end{proof}
\begin{remark}
\label{r:max-monot}
The assumption that $\a$ be Lipschitz and uniformly monotone is not necessary for Proposition~\ref{p:dualize}. The result can be proved along the same lines assuming only that~$\a$ is a (possibly set-valued) maximal monotone map.
\end{remark}
\begin{proof}[Proof of Theorem~\ref{t:rep}]
We decompose the proof into four steps. In the first step, we construct a preliminary representative~$\tilde F$ of~$\a$ which is uniformly convex. In the second step, we use~$\tilde F$ to construct a representative~$F$ of~$\a$ that is also self-dual, following an approach due to~\cite{BW}. In the third and fourth steps, we check that $F$ is uniformly convex and~$C^{1,1}$.

\smallskip

\emph{Step 1.} We construct a uniformly convex representative of $\a$. For convenience, denote $\tau :=1/\lambda$.
In view of \eqref{e.unifmonolip} and \eqref{e.unifmonolipinv}, we can write 
\begin{equation}
\label{e:decompaa}
\a(p) = \tau p + \a_0(p) \quad \text{and} \quad \a^{-1}(q) = \tau q + \a_1(q),
\end{equation}
where $\a_0$ and $\a_1$ are maximal monotone operators (see~\lref{maxmono} and~Remark~\ref{r:cond-l.maxmono}). Let $F_0$ and $F_1$ be the Fitzpatrick functions representing $\a_0$ and $\a_1$, respectively:
$$
F_0(p,q) := \sup_{\xi\in\Rd} \, \big(q \cdot \xi - \a_0(\xi) \cdot (\xi - p) \big) \quad \mbox{and} \quad F_1(q,p) := \sup_{\xi\in\Rd} \, \big(p \cdot \xi - \a_1(\xi) \cdot (\xi - q) \big).
$$
By~Example~\ref{ex:BEN}, the functions 
\begin{equation*}
(p,q) \mapsto F_0(p,q-\tau p) + \tau |p|^2 \quad \mbox{and} \quad (q,p) \mapsto F_1(q,p-\tau q) + \tau |q|^2
\end{equation*}
represent $\a$ and $\a^{-1}$ respectively. It then follows that~$\a$ is represented by 
\begin{equation}
\label{e:bang}
\ext{F}(p,q) := \frac{1}{2} \left( F_0(p,q-\tau p) + \tau |p|^2 + F_1(q,p-\tau q) + \tau |q|^2 \right).
\end{equation}
Observe that~$\ext{F}$ can be written in the form
\begin{equation}
\label{e:F-G}
\ext{F}(p,q) = \frac{\tau}{2} (|p|^2 + |q|^2) + G(p,q),
\end{equation}
with $G$ convex. 

\smallskip

\emph{Step 2.} 
We define~$F$ to be the proximal average of~$\ext{F}$ and $\ext{F}^*$:
\begin{equation}
\label{e:defF}
F(p,q) = \inf\Ll\{ \frac12 \ext{F}(p_1,q_1) + \frac12 \ext{F}^*(q_2,p_2) + \frac18 |p_1-p_2|^2 + \frac18 |q_1 - q_2|^2 \Rr\},
\end{equation}
where the infimum is taken over all $p_1,q_1, p_2,q_2\in\Rd$ such that
\begin{equation}
\label{e:decompp}
(p,q) = \frac12 (p_1,q_1) + \frac12(p_2,q_2) 
\end{equation}
and the Legendre-Fenchel transform $\ext{F}^*$ of $\ext{F}$ is defined in~\eqref{e:def:FL-two}. 
According to~\cite[Theorem~2.2]{BW}, the function $(p,q) \mapsto F(p,q)$ is self-dual. Using the notation of Remark~\ref{r:switch}, we let $H = (\ext{F} + \ext{F}^{*\T})/2$. It is clear from \eqref{e:defF} that $F \le H$. By duality, we also have $H^{*\T} \le F^{*\T} = F$. Moreover, we learn from Proposition~\ref{p:dualize} (and Remark~\ref{r:switch}) that both $H$ and $H^{*\T}$ represent $\a$. Since $H^{*\T} \le F \le H$, it thus follows that $F$ represents $\a$.

\smallskip

\emph{Step 3.} We now show that $F$ is uniformly convex using the representation \eqref{e:F-G}, that is,
$$
\ext{F}(p_1,q_1) = \frac{\tau}{2}(|p_1|^2 + |q_1|^2) + G(p_1,q_1)
$$
with $G$ convex. In order to lighten notation, let us denote $\xi := (p_1 - p_2)/2$ and $\eta := (q_1 - q_2)/2$. For $p_1,p_2,q_1,q_2\in\Rd$ satisfying \eqref{e:decompp}, since $p_1 = p + \xi$ and $q_1 = q + \eta$, we have
\begin{align*}
& \tau(|p_1|^2 + |q_1|^2) +  \frac12|p_1-p_2|^2 + \frac12|q_1 - q_2|^2  \\
& \qquad = \tau \Ll(|p|^2 + |q|^2 + 2 p\cdot \xi + 2 q\cdot \eta + |\xi|^2 + |\eta|^2\Rr) + 2|\xi|^2 + 2|\eta|^2 \\
& \qquad = \frac{2\tau}{2+\tau}\Ll(|p|^2 + |q|^2\Rr) + \underbrace{(2+\tau)\Ll|\xi + \frac{\tau}{2+\tau} p\Rr|^2 + (2+\tau)\Ll|\eta + \frac{\tau}{2+\tau} q\Rr|^2}_{=:\, 2 \ext{G}(p,q,\xi,\eta)}.
\end{align*}
Defining $\Lambda:=(2+\tau)/\tau$ (which matches~\eqref{e:def:lambda} as $\tau = 1/\Lambda$), we thus obtain
\begin{multline}
F(p,q) - \frac{1}{2\Lambda} \Ll(|p|^2 + |q|^2\Rr) \\
= \inf_{\xi,\eta\in\Rd} \, \Ll\{ \frac12 \Ll(G(p+\xi,q+\eta) + \ext{G}(p,q,\xi,\eta) + \ext{F}^*(q+\eta,p+\xi)\Rr) + \frac12 |\eta|^2  \Rr\}.
\end{multline}
The expression inside the infimum on the right side above is a convex function of $(p,q,\xi,\eta)$. It follows that the infimum itself is a convex function of $(p,q)$; see for example~\cite[Proposition~2.22]{RW}. This yields the uniform convexity of~$F$ or, more precisely,~\eqref{e:rep-conv}.

\smallskip

\emph{Step 4.} We show that $F$ is $C^{1,1}$ as a consequence of its self-duality and uniform convexity. In fact, this follows from the general fact that the Legendre-Fenchel transform of a uniformly convex function is $C^{1,1}$ from above. Since $F^*(q,p) = F(p,q)$, this will complete the proof. Letting $z = (p,q)$, we have shown that there exists a convex $\hat{G}$ such that 
$$
F(z) - \frac{1}{2\Lambda} |z|^2 = \hat{G}(z).
$$
We now observe that for every $z^* \in \R^{2d}$,
\begin{align*}
F^*(z^*) & = \sup_{z\in\R^{2d}} \, \Ll( z\cdot z^* - F(z) \Rr) \\
& = \sup_{z\in\R^{2d}} \, \Ll(z \cdot z^* - \frac{1}{2\Lambda} |z|^2 - \hat{G}(z) \Rr) \\
& = \frac{\Lambda}{2} |z^*|^2 - \inf_{z\in\R^{2d}} \Ll(  \frac{1}{2\Lambda}\Ll|z-\Lambda z^*\Rr|^2 + \hat{G}(z)\Rr).
\end{align*}
Since the function $(z,z^*) \mapsto \Ll|z-\Lambda z^*\Rr|^2/2\Lambda + \hat{G}(z)$ is convex, the infimum on the right side of the previous line is also convex, and we obtain \eqref{e:rep-C11}.
\end{proof}

\begin{remark}
An alternative proof of Theorem~\ref{t:rep} that does not rely on a proximal average (and thus gives a somewhat more explicit formula), but does not give a self-dual representative, is as follows. Choosing $\tau = 1/(2\Lambda)$ in \eqref{e:decompaa}, we see that $\a_0$ and $\a_1$ are uniformly convex and $C^{1,1}$. Hence, each has a uniformly convex representative by Step 1 of the above proof. We can then construct representatives $F_0$ and $F_1$ of $\a_0$ and $\a_1$ respectively that are $C^{1,1}$, by considering the dual functions. The function $\ext{F}$ defined in \eqref{e:bang} is then uniformly convex and $C^{1,1}$, and it represents $\a$.
\end{remark}

We also need a converse of Theorem~\ref{t:rep}, which is easier to prove. 

\begin{lemma}
\label{l.werepresent} 
Suppose that $F:\Rd\times\Rd \to \R$ satisfies~\eqref{e:rep-conv} for some $\Lambda \geq1$ and has the property that, for every $p\in\Rd$, 
\begin{equation}\label{e.Ftouch}
\inf_{q\in\Rd}\left( F(p,q) -p\cdot q \right)= 0. 
\end{equation}
Let $\a(p)\in\Rd$ be the point at which this infimum is attained, which is necessarily unique since~$q\mapsto F(p,q)$ is uniformly convex. Then~$\a$ is variationally represented by~$F$ and satisfies~\eqref{e.unifmonolip} with $\lambda=4\Lambda$. 
\end{lemma}
\begin{proof}
By the definition of~$\a$, we have that, for every $p,q\in\Rd$,
\begin{equation*}
F(p,q) = p\cdot q \quad \iff \quad \a(p)=q.
\end{equation*}
We have left to show that~$\a$ satisfies~\eqref{e.unifmonolip} with $\lambda=4\Lambda$.
Let $p_1,p_2\in\Rd$. By~\eqref{e:rep-conv},~\eqref{e.Ftouch} and the definition of~$\a$, 
\begin{align*}
\lefteqn{ (\a(p_1)-\a(p_2))\cdot (p_1-p_2) } \qquad  & \\
& = F(p_1,\a(p_1)) + F(p_2,\a(p_2)) - p_1\cdot \a(p_2) - p_2\cdot \a(p_1)  \\
& \geq 2F\left(\frac12 (p_1+ p_2),\frac12(\a(p_1)+\a(p_2))\right) - p_1\cdot \a(p_2) - p_2\cdot \a(p_1) \\
& \qquad + \frac{1}{8\Lambda}\left( |p_1-p_2|^2 +|\a(p_1)-\a(p_2)| \right) \\ 
& \geq \frac12 (p_1+p_2) \cdot \left( \a(p_1)+\a(p_2) \right) - p_1\cdot \a(p_2) - p_2\cdot \a(p_1)  \\
& \qquad + \frac{1}{8\Lambda}\left( |p_1-p_2|^2 +|\a(p_1)-\a(p_2)| \right) \\
& = \frac12 (\a(p_1)-\a(p_2))\cdot (p_1-p_2) + \frac{1}{8\Lambda}\left( |p_1-p_2|^2 +|\a(p_1)-\a(p_2)| \right). 
\end{align*}
Rearranging this gives
\begin{equation*}
 \frac1{4\Lambda} \left( |p_1-p_2|^2 +|\a(p_1)-\a(p_2)|^2  \right)\leq (\a(p_1)-\a(p_2))\cdot (p_1-p_2) .
\end{equation*}
Discarding the second term on the left side gives the first inequality of~\eqref{e.unifmonolip} with $\lambda=4\Lambda$; discarding the first term on the left side and using Young's inequality gives the second inequality of~\eqref{e.unifmonolip} with $\lambda=2\Lambda$.
\end{proof}

We conclude this subsection by stating a connection between the minimum of $F$ and $\a(0)$, for $F$ and $\a$ as in the statement of Theorem~\ref{t:rep}.

\begin{lemma}
\label{l.K0stuff}
Suppose that $F$ represents $\a$ and satisfies~\eqref{e:rep-conv} and~\eqref{e:rep-C11}. Then there exist $C,c>0$, depending only on $\Lambda$, such that 
\begin{equation} \label{e.K0stuff}
-C \left| \a(0) \right|^2 \leq \inf_{p,q\in\Rd} F(p,q)  \leq -c \left| \a(0) \right|^2. 
\end{equation}
\end{lemma}
\begin{proof}
Since $F$ is uniformly elliptic, $\nabla F$ exists in the classical sense and is itself a uniformly monotone and Lipschitz map on $\Rd\times \Rd$ with constant $\Lambda$, as is its inverse. Moreover, there exists a unique point $(p_0,q_0)\in \Rd\times \Rd$ at which $F$ attains its minimum. We will show that these facts imply the lemma.

\smallskip

First, using that the inverse of $\nabla F$ is Lipschitz with constant $\Lambda$, we get
\begin{multline} \label{e.sillystuff}
|p_0|^2 + |\a(0) -q_0|^2 = \left|(0,\a(0))-(p_0,q_0) \right|^2 \\ \leq \Lambda  \left| \nabla F (0,\a(0)) - \nabla F(p_0,q_0) \right|^2  = \Lambda \left| \a(0) \right|^2.
\end{multline}
In particular,  by Young's inequality,
\begin{equation*} \label{}
|p_0|^2 + |q_0|^2 \leq C \left| \a(0) \right|^2,
\end{equation*}
and therefore
\begin{equation*} \label{}
F(p_0,q_0) \geq p_0\cdot q_0 \geq -\frac12 \left( |p_0|^2+|q_0|^2 \right) \geq - C\left| \a(0) \right|^2.
\end{equation*}
This yields the first inequality of~\eqref{e.K0stuff}.

\smallskip

Similarly to~\eqref{e.sillystuff}, we next use that $\nabla F$ is Lipschitz with constant $\Lambda$ to get
\begin{equation*} \label{}
\left| \a(0) \right|^2\leq \Lambda \left( |p_0|^2 + |\a(0) -q_0|^2 \right),
\end{equation*}
and thus, by the uniform convexity of $F$, 
\begin{align*} \label{}
0 = F(0,\a(0)) & \geq F(p_0,q_0) + \frac{1}{2\Lambda} \left( |p_0|^2 + |\a(0) -q_0|^2 \right) \\
& \geq \inf_{p,q\in\Rd} F(p,q) + \frac{1}{2\Lambda^2} \left| \a(0) \right|^2.
\end{align*}
Rearranging gives the second inequality of~\eqref{e.K0stuff} and completes the proof of the lemma. 
\end{proof}

\subsection{Variational formulation of quasilinear elliptic PDEs}
\label{s.varchar}

A uniformly elliptic operator on a bounded domain $U$ can be thought of as a uniformly monotone mapping from $H^1(U)$ into $H^{-1}(U)$. From this point of view, it is natural to seek a variational representative in terms of a given representative of the underlying vector field~$\a$. This is provided below in Proposition~\ref{p.yesvariational} which, together with Theorem~\ref{t:rep}, completes the link between general equations of the form~\eqref{e.pde} and null minimizers of functions of the form~\eqref{e.fun}. 

\smallskip

We now reintroduce $x$-dependence, consider $K_0 \geq 0$ and $\Lambda \geq 3$ as in the hypotheses in the introduction and consider a Lebesgue measurable
\begin{equation*} \label{}
\a: \Rd \times \Rd \to \R
\end{equation*}
such that for each $x \in \R$, $\a(\cdot,x)$ satisfies~\eqref{e.unifmonolip} for $\lambda := \frac12(\Lambda -1)$ and
\begin{equation*} \label{}
\left| \a(0,x) \right| \leq K_0.
\end{equation*}
By Theorem~\ref{t:rep} and Lemma~\ref{l.K0stuff}, there exists $F \in \Omega$ such that, for each $x\in\Rd$, 
\begin{equation}\label{e.Frepa}
F(\cdot,\cdot,x) \quad \mbox{variationally represents the monotone map} \quad \a(\cdot,x),
\end{equation} 
and such that for every $p,q,x\in\Rd$,
\begin{equation}\label{e.Fgrow}
\frac1{2\Lambda} \left( |p|^2+|q|^2\right)-C K_0^2 \leq F(p,q,x) \leq \frac{\Lambda}2 \left( |p|^2+|q|^2 \right) + CK_0^2,
\end{equation}
where the constant $C$ depends only on~$\Lambda$.

\smallskip

The next proposition asserts that the functional $\J+\langle\cdot,\cdot \rangle$ defined in~\eqref{e.J} variationally represents the quasilinear elliptic operator $u\mapsto -\nabla \cdot \a(\nabla u,\cdot)$ as a monotone function from $H^1(U)$ to $H^{-1}(U)$ (compare with Definition~\ref{d:a-rep}). This result, like many of those in this section, is well-known (cf.~\cite{Gh}).

\begin{proposition}[Variational formulation of \eqref{e.pde}] 
\label{p.yesvariational}
Fix $U\subseteq \Rd$ and $F \in \Omega$ satisfying~\eqref{e.Frepa} and~\eqref{e.Fgrow}, and let $\J$ be given by~\eqref{e.J}. For every $(u,u^*) \in H^1(U) \times H^{-1}(U)$,
\begin{equation*}
\J\!\left[ u,u^* \right] \geq 0. 
\end{equation*}
Moreover, the following statements are equivalent:
\begin{enumerate}
\item[(i)] $(u,u^*)$ is a solution of the quasilinear equation
\begin{equation*} \label{}
-\nabla \cdot \left( \a( \nabla u,\cdot) \right)  =  u^* \quad \mbox{in} \ U;
\end{equation*}

\item[(ii)] $\J\!\left[ u,u^* \right] = 0$;

\item[(iii)] $u$ is the unique minimizer of the functional
\begin{equation*}
w\mapsto  \J\!\left[ w,u^* \right], \quad w\in u+H^1_0(U);
\end{equation*}

\item[(iv)] $u^*$ is the unique minimizer of the functional 
\begin{equation*}
w^* \mapsto  \J\!\left[ u,w^* \right], \quad w^*\in H^{-1}(U).
\end{equation*}
\end{enumerate} 
\end{proposition}
\begin{proof}
By~$F(p,q,\cdot) \geq p\cdot q$, we have, for any $u \in H^1(U)$ and $\g\in L^2(U;\Rd)$,
\begin{equation} \label{e.Fpq0}
 \int_{U} \left( F\left( \nabla u(x),\g(x) ,x\right) - \nabla u(x)\cdot \g(x) \right) \,dx \geq 0.
\end{equation}
Moreover, equality holds in the previous inequality if and only if
\begin{equation*} 
F\left( \nabla u(x),\g(x) ,x\right) = \nabla u(x)\cdot \g(x) \quad \mbox{a.e. in} \ U,
\end{equation*}
which by~\eqref{e.Frepa} is equivalent to
\begin{equation} \label{e.gistheflux}
\g(x) = \a\left( \nabla u(x) \right) \quad \mbox{a.e. in} \ U. 
\end{equation}
This yields the equivalence of statements~(i) and~(ii) in the proposition.

\smallskip

The equivalence with (iv) also follows from these facts and the observation that the minimum of the functional appearing in~(iv) is zero, that is, for every $u\in H^1(U)$,
\begin{equation}\label{e.nullstar}
\inf_{u^*\in H^{-1}(U)} \J\left[ u,u^* \right] = 0. 
\end{equation}
This can be seen for a given $u\in H^1(U)$ by simply taking $u^*:=-\nabla \cdot \left( \a(\nabla u(x),x) \right)$. That this~$u^*$ is unique minimizer follows from the equivalence between equality holding in~\eqref{e.Fpq0} and~\eqref{e.gistheflux}.

\smallskip

It remains to establish the equivalence between (iii) and the other statements. Fix $f\in H^1(U)$ and $u^*\in H^{-1}$ and notice that, for $u \in f+H^1_0(U)$, we have 
\begin{multline*}
\J\left[u,u^*\right] = \inf\bigg\{ \int_U \left( F(\nabla u(x),\g(x),x) -\nabla f(x) \cdot \g(x) \right)\,dx - \langle u-f, u^* \rangle\\
:\, \g\in L^2(U;\Rd), \ -\nabla \cdot \g = u^* \bigg\}.
\end{multline*}
It is clear that, for fixed $f\in H^1(U)$ and $u^*\in H^{-1}(U)$, the quantity in the infimum on the right side is bounded from below (by $0$), uniformly convex as well as lower semicontinuous on the linear space $(f+H^1_0(U)) \times \{ \g\in L^2(U;\Rd)\,:\, -\nabla \cdot \g = u^* \}$. It therefore has a unique minimizer. It suffices then to show that this minimization problem is null, that is, for every $f\in H^1(U)$ and $u^*\in H^{-1}(U)$,
\begin{equation} \label{e.touchzero}
 \inf_{u_0\in H^1_0(U)} \J\left[ u_0+f,u^* \right] = 0. 
\end{equation}
This follows from the solvability of the Dirichlet problem for the PDE and the equivalence of (i) and (ii). Alternatively, we give a variational proof of~\eqref{e.touchzero} by considering the dual convex optimization problem (following, e.g.,~the arguments in~\cite[Proposition~III.2.1]{ET} or~\cite[Proposition 6.1]{Gh}).

\smallskip

Fix $f\in H^1(U)$ and $u^* \in H^{-1}(U)$ and define $G:H^{-1}(U) \to \R$ by
\begin{equation*}
G(v^*) := \inf_{u_0\in H^1_0(U)} \left( \J\left[ u_0+f,v^*+u^*\right] +\langle u_0,v^* \rangle \right). 
\end{equation*}
We want to argue that $G(0) = 0$. It is clear that $G(0) \geq 0$, so it suffices to argue that $G(0) \leq 0$. It is easy to check that $G$ is bounded below, convex and lower semicontinuous. The Legendre-Fenchel transform of 
$G$ satisfies, for every $v \in H^1_0(U)$, 
\begin{align*}
G^*(v) & := \sup_{ v^*\in H^{-1}(U) } \left( \langle v,v^* \rangle - G(v^*) \right)\\
& = \sup_{v^*\in H^{-1}(U)}  \sup_{u_0\in H^1_0(U)} \left( \langle v-u_0,v^* \rangle - \J \left[ u_0+f, v^* + u^* \right] \right)  \\
& \geq \sup_{v^*\in H^{-1}(U)} \left( - \J \left[ v+f, v^* + u^* \right] \right) \\
& = 0,
\end{align*}
where we used~\eqref{e.nullstar} in the last step.
By duality, using the fact that $G$ is convex and lower semicontinuous, we have $G^{**} = G$. Thus
\begin{equation*}
G(0) = G^{**}(0) = \sup_{v\in H^1_0(U)} \left( -G^*(v) \right) \leq 0.  \qedhere
\end{equation*}
\end{proof}

\section{Lipschitz regularity: structure of the proof of Theorem~\ref{t.lip}}
\label{s.lipschitz}

In this section, we present the scheme for obtaining a Lipschitz estimate from Campanato iterations. We also state the error estimate in homogenization of the Dirichlet problem and argue that these two elements imply Theorem~\ref{t.lip}.

\subsection{Two ingredients in the proof of Theorem~\ref{t.lip}}
The primary ingredient in the proof of Theorem~\ref{t.lip} is the following estimate for the error in homogenization of the Dirichlet problem. Under strong mixing conditions, it is necessarily \emph{suboptimal} in its estimate of the typical size of the error because it is \emph{optimal} in terms of stochastic integrability. The functional $\J$ is defined in~\eqref{e.J} and the homogenized functional $\overline \J$ referred to in the statement of the theorem is defined below in Section~\ref{s.homcoeff} (see~\eqref{e.barJ}). 

\begin{theorem}[Quenched $L^2$ error estimate in homogenization]

Fix a bounded Lipschitz domain~$U_0\subseteq \Rd$ and exponents $\delta \in (0,1]$ and $\theta \in (0,\beta)$. There exist $s_0(d,\Lambda,\beta,\delta,\theta)>0$, $C(d,\Lambda,\beta,C_3,\delta,U_0,\theta)\geq 1$ and, for every $s\in [s_0,\infty)$, a nonnegative random variable $\X_s$ satisfying
\begin{equation} \label{e.Xsst}
\E \left[ \X_s^s\right] \leq C
\end{equation}
such that the following holds.
For $R\geq 1$, $U:= RU_0$, $f\in W^{1,2+\delta}(U)$ and setting
\begin{equation}
\label{e.condMM}
M:= K_0 + \Ll(\fint_U \left| \nabla f(x) \right|^{2+\delta}\,dx \right)^{\frac1{2+\delta}},
\end{equation}
the unique functions $u,\uhom \in f+H^1_0(U) $   such that
\begin{equation} \label{e.Dir}
 \J \left[u, 0 \right] = \overline\J\left[ \uhom,0 \right] =  0,
\end{equation}
satisfy the estimate
\begin{equation} \label{e.wkee}
R^{-2} \fint_{U} \left| u(x) - \uhom(x) \right|^2\,dx \leq CM^{2} (1+M)^{2d/s} \X_s R^{-\theta/s}.
\end{equation}
Moreover, under assumption~(P4), for every $s \in (0,d\gamma/(d+\gamma))$, there exist $\alpha(d,\Lambda,\gamma,\delta,s)>0$, $C(d,\Lambda,\delta,U_0,\gamma,C_4,s)\geq 1$ and a random variable $\Y_s$ satisfying
\begin{equation} \label{e.Xsexp}
\E \left[ \exp\left( \Y_s \right) \right]  \leq C
\end{equation}
such that, for $u$ and $\uhom$ as above,
\begin{equation} \label{e.stree}
R^{-2} \fint_{U} \left| u(x) - \uhom(x) \right|^2\,dx \leq CM^2\left( 1+ \Y_s  R^{-s} \log (1+M) \right) R^{-\alpha}.
\end{equation}
\label{t.subopt}
\end{theorem}
\begin{remark}
\label{r.weirdness}
The non-quadratic dependence of the estimates~\eqref{e.wkee} and~\eqref{e.stree} on the parameter $M$ is a nonlinear phenomenon. Indeed, in the linear case (as usual we mean the PDE is linear, i.e., $F$ is quadratic) it may be removed by homogeneity. However, in the nonlinear case, this is not an artifact of our method: thinking in terms of the PDE, it arises because it is not necessarily the case that the span of the random fields $\{\a(p,\cdot)\}_{p\in\Rd}$ lie in a  finite dimensional space. Thus different $p$'s may exhibit different (i.e., independent) randomness, and since the estimates are quenched, we are required to control them all at once, up to the fixed size~$M$. Thus the larger $M$ is, the larger we may expect the random part of the right sides in the estimates to be. 
\end{remark}

The following proposition encapsulates the general scheme for proving Lipschitz estimates introduced in~\cite{AS}. It reduces Theorem~\ref{t.lip} to Theorem~\ref{t.subopt} by a Campanato-type iteration, which allows us to focus most of the effort in this paper on the proof of the latter. It is almost the same as~\cite[Lemma 5.1]{AS}, but we have included  a slight variation to allow us to handle equations with non-zero right hand sides. We also formulate the result using the $L^2$ norm rather than the $L^\infty$ norm, but this makes no difference in the argument.

\smallskip

Before giving the statement, we briefly introduce some notation. We take $\L$ to be the set of affine functions on $\Rd$ and define, for each $\sigma \in (0,\tfrac12]$ and $r>0$, the set
\begin{multline*} \label{}
\mathcal A (r,\sigma) := \Bigg\{ v \in L^2(B_r) \, : \, \frac{1}{\sigma r} \inf_{l\in\L} \left( \fint_{B_{\sigma r}} \left| v(x) - l(x) \right|^2\,dx \right)^{\frac12}  \\
\leq \frac 12 \left( \frac{1}{r} \inf_{l\in\L} \left( \fint_{B_{r}} \left| v(x) - l(x) \right|^2\,dx \right)^{\frac12}  \right)   \Bigg\}.
\end{multline*}
In words, the set $\mathcal A(r,\sigma)$ consists of those $L^2(B_r)$ functions~$u$  which satisfy one step of a $C^{1,\alpha}$ Campanato iteration with dilation factor~$\sigma$: the flatness of~$u$ in $B_{\sigma r}$ is improved from its flatness in $B_r$ by a factor of two.

\begin{proposition}[Campanato scheme]
\label{p.campanato}
For $\sigma \in (0,\frac12]$ and $\gamma \in (0,1]$ such that $\sigma^\gamma \geq \frac23$ and $\alpha > 0$, there exists $C(\sigma,\gamma,\alpha)\geq 1$ such that the following holds. Suppose that $R\geq 1$, $K,L\geq 0$, $r_0 \in [1,R/8]$ and $u\in L^2(B_R)$ have the property that for every $r\in [r_0,R/8]$,
\begin{multline} \label{e:osc}
 \inf_{v \in \mathcal A(r,\sigma) } \frac1r \left( \fint_{B_{r}} \left| u(x) - v(x) \right|^2\,dx \right)^{\frac12} \\ \leq r^{-\alpha} \left( K + \frac{1}{2r}\inf_{a \in \R} \left( \fint_{B_{2r}}  |u(x) - a|^2\,dx \right)^{\frac12} \right) + Lr^\gamma.
\end{multline}
Then for every $r\in [r_0,R/2]$, 
\begin{multline} \label{e.c01it}
\frac1r \inf_{a\in\R} \left( \fint_{B_{r}}  \left|u(x)-a\right|^2\,dx \right)^{\frac12} \\
 \leq C \left( \frac{1}{R} \inf_{a\in\Rd} \left( \fint_{B_{R}}  \left|u(x)-a\right|^2\,dx \right)^{\frac12} + K \left( \frac{r}{R} \right)^\alpha + LR^\gamma \right).
\end{multline}
\end{proposition}

Proposition~\ref{p.campanato} asserts that if a function~$u\in L^2(B_R)$ (we are thinking of~$R$ very large) has the property that, in every ball~$B_r$ with radius~$r$ between~$R/8$ and a ``minimal radius"~$r_0$, it can be well-approximated by a function in~$\mathcal A(r,\sigma)$, then in fact $u$ does not oscillate too much on scales larger than the minimal radius. Its proof appears in Section~\ref{ss.campanato} below. 

\smallskip

In order to make use of Proposition~\ref{p.campanato}, we need to check that null minimizers of the homogenized functional belong to $\mathcal A (r,\sigma)$. This is handled by the following simple lemma, which is a reflection of the well-known fact that a family of scale-invariant functions satisfies a $C^{1,\alpha}$ estimate if and only if it satisfies an improvement of flatness property. 

\begin{lemma}
\label{l.belongA}
Suppose that $\alpha \in (0,1]$, $K\geq 0$, and $u\in C^{1,\alpha}(B_R)$ has the property that, for every $r \in (0,R/2]$, 
\begin{equation*} \label{}
\left[ \nabla u \right]_{C^{0,\alpha}(B_r)} \leq K r^{-\alpha} \left( \frac1{2r}\inf_{l\in \mathcal L} \left( \fint_{B_{2r}} \left| u(x) - l(x) \right|^2\, dx \right)^{\frac12} \right).
\end{equation*}
Then there exists $\sigma(\alpha,K)\in (0,\frac12]$ such that $u\in \mathcal A(r,\sigma)$ for every $r\in(0,R/2]$. 
\end{lemma}

The proof of Lemma~\ref{l.belongA} is also given in Section~\ref{ss.campanato} below.

\subsection{Proof of Theorem~\ref{t.lip}}

In this subsection, we prove the main result of the paper by showing that the combination of Theorem~\ref{t.subopt} and Proposition~\ref{p.campanato} yields a pointwise Lipschitz estimate at large scales.

\smallskip

\begin{proof}[{Proof of Theorem~\ref{t.lip}}]
Throughout we fix $p>d$, $\theta\in (0,\beta)$. We take $\delta(d,\Lambda) > 0$ to be the exponent~$\delta_0$ in the statement of the interior Meyers estimate given in Proposition~\ref{p.interiorMeyers} in the appendix.
 For each $r\in [1,R/8]$, let $v_r \in u+H^1_0(B_{4r})$ denote the unique minimizer of $\J[\cdot,0]$ belonging to $u+H^1_0(B_{4r})$, and $\uhomr \in v_r+H^1_0(B_r)$ the unique minimizer of $\overline \J[\cdot,0]$ in~$v_r+H^1_0(B_r)$.

\smallskip

By the definition of $\overline \J$ and $\overline \a$ in Sections~\ref{s.homcoeff} and~\ref{s.variationalbar}, respectively, and using~Lemma~\ref{l.werepresent} and Proposition~\ref{p.yesvariational}, we see that $\uhomr$ is a solution of the constant-coefficient equation
\begin{equation*}
-\nabla \cdot \overline \a\left( \nabla \uhomr \right) = 0 \quad \mbox{in} \ B_{r}.
\end{equation*}
Note that $\overline \a$ is uniformly monotone and Lipschitz with constants depending only on~$\Lambda$. 
Applying Lemma~\ref{l.belongA} and~\cite[Theorem 8.9]{Giu}, we have that $\uhomr\in \mathcal A(r,\sigma)$ for some $\sigma(d,\Lambda) \in(0,\frac12]$. Note that while the assumptions in~\cite[Theorem 8.9]{Giu} require that the coefficients $\overline \a$ be differentiable,  this assumption is not quantitative and therefore the theorem holds for Lipschitz coefficients by approximation.

\smallskip

By making $p$ smaller if necessary, we may assume without loss of generality that $\sigma^{1-d/p} \geq \frac23$, which is convenient in view of the hypothesis of Proposition~\ref{p.campanato}.

\smallskip

Throughout the argument, we let $C$ denote a positive constant which may vary in each occurrence and depends only on $(d,\Lambda,\beta,C_3,p,\theta)$ or, in the case that~(P4) holds, depends only on $(d,\Lambda,\beta,C_3,p,\gamma,C_4,s)$.

\smallskip

\emph{Step 1.} We record some preliminary estimates involving $v_r$ and $u-v_r$. The interior Meyers estimate (see Proposition~\ref{p.interiorMeyers}) and the Caccioppoli inequality (see Proposition~\ref{p.caccioppoli}) imply that 
\begin{align*} \label{}
\left( \fint_{B_{r}} \left| \nabla v_r(x) \right|^{2+\delta}\,dx \right)^{\frac1{2+\delta}}  & \leq C \left( K_0 +\left( \fint_{B_{2r}} \left| \nabla v_r(x) \right|^2\,dx \right)^{\frac12} \right) \\
& \leq C \left( K_0  + \frac1{4r}\inf_{a\in\R} \left( \fint_{B_{4r}} \left| v_r(x) -a \right|^2\,dx\right)^{\frac12} \right).
\end{align*}
We next apply Proposition~\ref{p.L2estimate} in the appendix and then use H\"older's inequality to obtain 
\begin{multline} \label{e.graduvr}
 \fint_{B_r} \left| \nabla u(x) - \nabla v_r(x) \right|^2\,dx \\  \leq Cr^2  \fint_{B_{r}} \left| f(x) \right|^2\,dx \leq CR^{\frac{2d}p} r^{2-\frac{2d}p} \left( \fint_{B_{R}} \left| f(x) \right|^p\,dx \right)^{\frac 2p} \leq C M^2.
\end{multline}
Then by the Poincar\'e inequality, we get, for every $r\leq R/8$, 
\begin{equation} \label{e.throwawayf}
\frac1{r^2} \fint_{B_r} \left|  u(x) -  v_r(x) \right|^2\,dx  \leq CR^{\frac{2d}p} r^{2-\frac{2d}p} \left( \fint_{B_{R}} \left| f(x) \right|^p\,dx \right)^{\frac 2p} \leq CM^2.
\end{equation}
Combining these with the triangle inequality yields that 
\begin{equation} \label{e.okaytogo}
\left( \fint_{B_{r}} \left| \nabla v_r(x) \right|^{2+\delta}\,dx \right)^{\frac1{2+\delta}} \leq C_0 \left( M  + \frac1{4r}\inf_{a\in\R} \left( \fint_{B_{4r}} \left| u(x) -a \right|^2\,dx\right)^{\frac12} \right), 
\end{equation}
where we write $C = C_0$ for future reference.

\smallskip

\emph{Step 2.} We plan to appeal to Theorem~\ref{t.subopt} in a sequence of dyadic balls. In order to prepare the ground for this, we fix $C'\geq 1$ to be selected below and observe that letting
$c = 1/(2C_0) > 0$, for any $r \le R/8$, the condition
\begin{equation} \label{e.uglycondition}
\frac1{4r}\inf_{a\in\R} \left(\fint_{B_{4r}} \left| u(x) -a \right|^2\,dx  \right)^{\frac12} \leq c\, C' M
\end{equation}
implies by \eqref{e.okaytogo} that
\begin{equation}
\label{e.matchcond}
M_r := K_0 + \left( \fint_{B_{r}} \left| \nabla v_r(x) \right|^{2+\delta}\,dx \right)^{\frac1{2+\delta}} \le K_0 + C_0M(1+c\,C') \le C'M
\end{equation}
provided that $C' \ge 2+2C_0$. 

\smallskip

\emph{Step 3.} We define the random minimal radius~$\X \geq 1$ in the general case that only~(P3) holds. Set $\theta'= \theta+\frac12(\beta-\theta)$. With $s=s_0(d,\Lambda,\beta,\delta,\theta')\geq 1$ and $\X_s$ as in the statement of Theorem~\ref{t.subopt} with $\theta'$ in place of $\theta$, 
we set
\begin{equation*} \label{}
\mathcal X:= [(C')^{2+2d/s}\X_s]^{\frac s\theta}.
\end{equation*}
According to the conclusions of Theorem~\ref{t.subopt}, we then have that
\begin{equation*} \label{}
\E \left[ \X^{\theta} \right] \leq C(C')^{2s+2d}
\end{equation*}
and, for every $r \in \left[ \X (1+M)^{\frac{2d}{\theta}},\, \frac18 R\right]$, 
\begin{equation} \label{e.EE0}
\frac1{r^2} \fint_{B_r} \left| \uhomr(x)-v_r (x) \right|^2\,dx 
\leq \frac{C}{(C')^2}M_r^2 \, r^{-(\theta'-\theta)/s} \leq CM^2 r^{-\alpha},
\end{equation}
provided that $r$ satisfies \eqref{e.uglycondition},  since \eqref{e.uglycondition} implies \eqref{e.matchcond}. Note that the constant $C$ in \eqref{e.EE0} does not depend on our choice of $C'$, and that the exponent $\alpha>0$ depends only on~$(d,\Lambda,\beta,\theta)$.

\smallskip

\emph{Step 4.} In the case that assumption~(P4) holds, we define~$\X$ differently. Here we  fix an exponent $s\in (0,d\gamma/(d+\gamma))$ and allow the constant $C$ to depend additionally on $(\gamma,C_4,s)$. Letting $\mathcal{Y}_{s}$ and $\alpha > 0$ be as in the second statement of Theorem~\ref{t.subopt} with $M_0$ replaced by $C'M_0$, we define
\begin{equation*} \label{}
\X:= \max\Ll\{\mathcal Y_{s}^{\frac1s}, (C')^{\frac4{\alpha}}\Rr\},
\end{equation*}
so that
\begin{equation*} \label{}
\E\left[ \exp\left( \X^s \right) \right]  \leq Ce^{(C')^{\frac{4s}{\alpha}}}(C'M_0)^{2d},
\end{equation*}
and for every $r \in [\X,\frac18R]$,
$$
\frac1{r^2} \fint_{B_r} \left| \uhomr(x)-v_r (x) \right|^2\,dx 
\leq \frac{C}{(C')^2}M_r^2 r^{-\alpha/2} \leq CM^2 r^{-\alpha/2},
$$
provided that~\eqref{e.uglycondition} holds. Up to a redefinition of $\alpha > 0$, this is \eqref{e.EE0}, with $\alpha$ depending additionally on $(\gamma,s)$.

\smallskip

\emph{Step 5.} We observe that, for every $r\in \left[ \X(1+M)^{\frac{2d}{\theta}}, \frac18R \right]$ such that~\eqref{e.uglycondition} holds,
\begin{multline} \label{e.approxstep}
\frac1r \left( \fint_{B_r} \left| u(x) - \uhomr(x) \right|^2\,dx \right)^{\frac12} \\ 
\leq C_1 \Ll[r^{-\alpha} \, M 
+ R^{\frac dp} r^{1-\frac dp} \left( \fint_{B_R} \left| f(x) \right|^p\,dx  \right)^{\frac1p} \Rr],
\end{multline}
where we write $C_1 = C$ for future reference. 
Indeed, this follows immediately from~\eqref{e.throwawayf} and~\eqref{e.EE0}. 

\smallskip

\emph{Step 6.} Let $C'' := 8^{1+d/2} + C_2(1+2 C_1)$, where $C_2 = C(\sigma, 1-d/p, \alpha)$ is the constant given by Proposition~\ref{p.campanato} (and $\sigma$ was defined in Step 1, above). We show that, for every $r\in \left[\X (1+M)^{\frac{2d}{\theta}},\,\frac12R\right]$,
\begin{equation} \label{e.keystep}
\frac1r \inf_{a\in\Rd}  \left( \fint_{B_r} \left| u(x) -a \right|^2\,dx \right)^{\frac12}  \leq C'' M.
\end{equation}
We argue by induction. The estimate is obvious for $r \in[ R/8,R/2]$. Suppose that \eqref{e.keystep} holds for every radius $r \in [4 r_0, R/2]$, with $r_0$ such that $r_0 \ge \X(1+M)^{\frac{2d}{\theta}}$. Then~\eqref{e.uglycondition} holds for every $r\in \left[r_0,R/8\right]$, provided we choose $C' \ge C''/c$. In view of~\eqref{e.approxstep}, an application of Proposition~\ref{p.campanato} yields that, for every $r\in \left[r_0,R/2\right]$,
\begin{equation*} \label{}
\frac1{r}  \inf_{a\in\Rd}  \left( \fint_{B_{r}} \left| u(x) -a \right|^2\,dx \right)^{\frac12} \le  C_2(M +2 C_1 M) \le C''M.
\end{equation*} 
This completes the proof of~\eqref{e.keystep}.

\smallskip

\emph{Step 7.} We conclude. By the Caccioppoli inequality (Proposition~\ref{p.caccioppoli}),~\eqref{e.keystep} and the H\"older inequality, we have for every $r\in\left[\X (1+M)^{\frac{2d}{\theta}},\frac12R\right]$,
\begin{align*} \label{}
\fint_{B_{r}} \left| \nabla u(x) \right|^2\,dx  & \leq C \left(  K_0^2+ \frac1{r^2} \inf_{a\in\R} \fint_{B_{2r}} \left| u(x) -a \right| ^2\,dx + r^2  \fint_{B_{2r}} \left| f(x) \right|^2\,dx \right)\\
& \leq CM^2 + CR^{\frac {2d}p} r^{2-\frac {2d}p} \left( \fint_{B_R} \left| f(x) \right|^p\,dx  \right)^{\frac2p} \\
& \leq CM^2.
\end{align*}
This completes the proof of the theorem.
\end{proof}

\subsection{The Campanato \texorpdfstring{$C^{1,\alpha}$}{C1a}--type iteration}
\label{ss.campanato}
In this subsection we give the proofs of Proposition~\ref{p.campanato} and Lemma~\ref{l.belongA}.

\begin{proof}[{Proof of Proposition~\ref{p.campanato}}]
\emph{Step 1.} We first argue that we may assume \eqref{e:osc} to hold for every $r \in [r_0,R/2]$. Indeed, assume that the proposition is proved with this stronger assumption. Then we can use the result with $R$ replaced by $R/4$ to get \eqref{e.c01it} for $r \in [r_0, R/8]$, up to a redefinition of $C$. That \eqref{e.c01it} holds for $r \in [R/8,R]$ is obvious (up to a redefinition of $C$). We may therefore assume that \eqref{e:osc} holds for every $r \in [r_0,R/2]$. 

\medskip

Throughout, $C$ and $c$ denote positive constants depending only on~$(\alpha,\sigma,\gamma)$ which may vary in each occurrence. Define $s_0:= R$ and, for each $j \in\N$, set $s_j:= \sigma^{j-1} R/4$. Pick $m\in\N$ such that $s_{m+1} < r_0/2 \leq s_m$. Denote, for each $j\in\{ 0,\ldots,m\}$,
\begin{align*} \label{}
 F_j & := \frac1{s_j} \inf_{l\in\L} \left( \fint_{B_{s_j}} \left| u(x) - l(x) \right|^2\,dx \right)^{\frac12}, \\
 H_{j} & := \frac{1}{s_j} \inf_{a\in\R} \left( \fint_{B_{s_j}} \left| u(x) - a \right|^2\,dx \right)^{\frac12}.
\end{align*}

\smallskip

\emph{Step 2.} 
We show that for every $j \in\{ 1,\ldots, m\}$,
\begin{equation} \label{e.bb2iter}
F_{j+1} \leq \frac12 F_j + Cs_j^{-\alpha} (K+H_{j-1}) + CLs_j^\gamma.
\end{equation}
By~\eqref{e:osc} (and $\sigma \le 1/2$) that there exists $v \in \mathcal A(s_j,\sigma)$ such that
\begin{equation}
\label{e:v-opt}
\frac1{s_j} \left( \fint_{B_{s_j}} \left| u(x) - v(x) \right|^2\,dx \right)^{\frac12}  \leq s_j^{-\alpha} \left( K + H_{j-1}\right) + Ls_j^\gamma.
\end{equation}
By the triangle inequality,
$$
F_{j+1} \le \frac{1}{s_{j+1}}  \left( \fint_{B_{s_{j+1}}} \left| u(x) - v(x) \right|^2\,dx \right)^{\frac12} + \frac{1}{s_{j+1}} \inf_{l \in \L} \left( \fint_{B_{s_{j+1}}} \left| v(x) - l(x) \right|^2\,dx \right)^{\frac12}.
$$
The first term is bounded by $\sigma^{d/2-1} \le C$ times the left side of \eqref{e:v-opt}, while since $v \in \mathcal A(s_j,\sigma)$, the second term is bounded by
$$
\frac{1}{2} \frac{1}{s_{j}}  \inf_{l \in \L} \left( \fint_{B_{s_{j}}} \left| v(x) - l(x) \right|^2\,dx \right)^{\frac12} \le \frac{1}{2} F_j + \frac12 \frac1{s_j} \left( \fint_{B_{s_j}} \left| u(x) - v(x) \right|^2\,dx \right)^{\frac12},
$$
and we use \eqref{e:v-opt} again to estimate the right-most term. This yields~\eqref{e.bb2iter}.

\smallskip

\emph{Step 3.} We show that, for every $j\in\{ 0,\ldots,m-1\}$,
\begin{equation} \label{e.bb2iter2}
F_{j+1} \leq \frac12 F_j + Cs^{-\alpha}_j \left(K+H_0+\sum_{i=0}^j F_i\right) + CLs_j^\gamma.
\end{equation}
Select $p_j\in \Rd$ such that 
\begin{equation*} \label{}
F_j = \frac{1}{s_j} \inf_{a\in\Rd}\left( \fint_{B_{s_j}} \left| u(x) - p_j\cdot x - a \right|^2\,dx \right)^{\frac12}.
\end{equation*}
The triangle inequality yields, for every $j\in \{ 0,\ldots,m\}$,
\begin{equation} \label{e.FjHj}
F_j\leq H_j \leq 2|p_j| + F_j.
\end{equation}
Moreover, for any $a,b \in \R$,
\begin{multline*}
|p_j| \leq \frac C{s_j} \left( \fint_{B_{s_j}} \left| p_j\cdot x\right|^2\,dx \right)^{\frac12} \le \frac C{s_j} \left( \fint_{B_{s_j}} \left| p_j\cdot x +a-b\right|^2\,dx \right)^{\frac12} \\
\le \frac C{s_j} \left( \fint_{B_{s_j}} \left| u(x)-p_j\cdot x -a\right|^2\,dx \right)^{\frac12}  +  \frac C{s_j} \left( \fint_{B_{s_j}} \left| u(x) - b\right|^2\,dx \right)^{\frac12}  ,
\end{multline*}
so that
\begin{equation} \label{e.pjtrivbnd}
|p_j| \leq C(F_j + H_j) \leq CH_j.
\end{equation}
Similarly, for every $j\in \{ 0,\ldots,m-1\}$,
\begin{equation} \label{e.pjtriang}
\left|p_{j+1}-p_j \right| \leq  \frac C{s_{j}} \left( \fint_{B_{s_j}} \left| (p_{j+1}-p_j)\cdot x\right|^2\,dx \right)^{\frac12} \le C(F_{j+1}+F_j).
\end{equation}  
By iterating~\eqref{e.pjtriang} and using~\eqref{e.pjtrivbnd} with $j=0$, we get
\begin{equation} \label{e.pjbndsumFi}
|p_j| \leq H_0 + C\sum_{i=0}^j  F_i.
\end{equation}
By~\eqref{e.bb2iter},~\eqref{e.FjHj} and~\eqref{e.pjbndsumFi}, we obtain \eqref{e.bb2iter2}.

\smallskip

\emph{Step 4.} We show that, for every $j\in \{0,\ldots,m\}$,
\begin{equation} \label{e.bb3iterbby}
F_j \leq 2^{-j} F_0 + Cs_j^{-\alpha} \left( K+H_0 \right) + CL\left( s_j^\gamma+R^\gamma s_j^{-\alpha} \right).
\end{equation}
We argue by strong induction. Fix $A,B\geq 1$ to be selected below and suppose that $k\in \{ 0,\ldots,m-1\}$ is such that, for every $j\in \{0,\ldots,k\}$, 
\begin{equation} \label{e.bb3iterindhyp}
F_j \leq  2^{-j} F_0 + As_j^{-\alpha} \left( K+H_0 \right) + L\left( As_j^\gamma+ BR^\gamma s_j^{-\alpha} \right).
\end{equation}
Using~\eqref{e.bb2iter2} and this assumption (and then $F_0 \le H_0$), we find that 
\begin{align*}
F_{k+1} & \leq \frac12 \left( 2^{-k} F_0 + As_k^{-\alpha} \left( K+H_0 \right) + L \left( As_k^\gamma+ BR^\gamma s_k^{-\alpha}  \right) \right) + CLs_k^\gamma \\
& \qquad + Cs_k^{-\alpha} \left(K+H_0+\sum_{j=0}^k \left( 2^{-j} F_0 + As_j^{-\alpha} \left( K+H_0 \right) + L \left(A s_j^\gamma+ BR^\gamma s_j^{-\alpha}  \right) \right) \right) \\
& \leq 2^{-(k+1)} F_0 + s_{k+1}^{-\alpha} (K+H_0) \left(\frac12 A + CAs_k^{-\alpha} +C  \right) \\
& \qquad + Ls_{k+1}^\gamma \left(  \frac12\sigma^{-\gamma}A+C \right)  + L R^\gamma s_{k+1}^{-\alpha}  \left(   \frac12 B + CA + CBs_k^{-\alpha} \right).
\end{align*}
Now suppose in addition that $k\leq n$ with $n$ such that $Cs_n^{-\alpha} \leq \frac14$. Then using this and $\sigma^{\gamma} \geq \frac23$, we may select $A$ large enough that 
\begin{equation*} \label{}
\frac12 A + C A s_k^{-\alpha} + C\leq \frac34 A + C \leq A \quad \mbox{and} \quad \frac12 \sigma^{-\gamma} A + C \leq A
\end{equation*}
and then select $B$ large enough that 
\begin{equation*} \label{}
\frac12 B + CA + CBs_k^{-\alpha}  \leq B.
\end{equation*}
We obtain
\begin{equation*} \label{}
F_{k+1} \leq 2^{-(k+1)} F_0 + As_{k+1}^{-\alpha} \left( K+H_0 \right) + ALs_{k+1}^\gamma + BLR^\gamma s_{k+1}^{-\alpha}.
\end{equation*}
By induction, we obtain~\eqref{e.bb3iterindhyp} for every $j \leq n$. For every $j \in \{ n+1,\ldots,m\}$, we have $1 \leq s_j / s_m \leq C$ and hence $F_j \leq CF_n$. We thus obtain~\eqref{e.bb3iterindhyp} for every $j\in \{ 0,\ldots,m\}$. This completes the proof of~\eqref{e.bb3iterbby}.

\smallskip

\emph{Step 5.} We conclude the argument. To obtain~\eqref{e.c01it}, we need to estimate~$H_j$. According to~\eqref{e.FjHj}, \eqref{e.pjbndsumFi} and~\eqref{e.bb3iterbby}, we have
\begin{align*}
H_j \leq F_j + 2|p_j| & \leq 2H_0 + C\sum_{i=0}^j F_i \\
&\leq 2H_0+ C\sum_{i=0}^j \left( 2^{-i} F_0 + s_i^{-\alpha} \left( K+H_0 \right) +L\left( s_i^\gamma + R^\gamma s_i^{-\alpha} \right) \right) \\
& \leq 2H_0 + CF_0 + Cs_j^{-\alpha} (K+H_0) +CLR^\gamma\left( 1+s_j^{-\alpha} \right) \\
& \leq C \left( H_0 + s_j^{-\alpha} K +LR^\gamma \right).
\end{align*}
This completes the proof of the proposition. 
\end{proof}

\begin{proof}[{Proof of Lemma~\ref{l.belongA}}]
Denote $l_0(x):= u(0) + x\cdot \nabla u(0)$. Then the hypothesis of the lemma yields, for every $r\in (0,R/2]$ and $s\in (0,r/2]$,
\begin{align*}
 \left( \fint_{B_{s}} \left| u(x) - l_0(x) \right|^2\, dx \right)^{\frac12} & \leq \osc_{x\in B_s} \left| u(x) - l_0(x) \right| \\
 & \leq s^{1+\alpha} \left[ \nabla u \right]_{C^{0,\alpha}(B_s)} \\
& \leq K s^{1+\alpha} r^{-1-\beta} \inf_{l\in\L}  \left( \fint_{B_{r}} \left| u(x) - l(x) \right|^2\, dx \right)^{\frac12}.
\end{align*}
Taking $s=\sigma r$, this gives  
\begin{align*} \label{}
\frac1s \left( \fint_{B_{s}} \left| u(x) - l_0(x) \right|^2\, dx \right)^{\frac12} 
& \leq K \sigma^{\alpha}\left( \frac1r \inf_{l\in\L}  \left( \fint_{B_{2r}} \left| u(x) - l(x) \right|^2\, dx \right)^{\frac12}\right).
\end{align*}
Taking $\sigma\in (0,\frac12]$ small enough that $K\sigma^\alpha \leq \frac12$, we obtain the lemma.
\end{proof}

\section{The subadditive quantities}
\label{s.subadditive}

We introduce the subadditive quantity~$\mu_0$ and superadditive quantity~$\mu$ which play the central role in the proof of Theorem~\ref{t.subopt}. We review their basic properties and give some notation needed in the rest of the paper. 

\smallskip

Throughout this section, we let $C\geq 1$ and $c \in (0,1]$ denote positive constants which may vary in each occurrence and depend only on $(d,\Lambda)$.

\subsection{Definition of the subadditive quantities}
We denote the space of~$L^2$ solenoidal vector fields on a bounded Lipschitz domain~$U \subseteq \Rd$ by
\begin{equation*} \label{}
\Ls(U):= \left\{  \mathbf{f} \in L^2(U;\Rd) \,: \, \int_U \mathbf{f}(x)\cdot \nabla \phi(x) \,dx = 0 \ \ \mbox{for every} \ \ \phi \in H^1_0(U) \right\}.
\end{equation*}
The subspace of~$\Ls(U)$ consisting of solenoidal vector fields with zero normal component on~$\partial U$ is 
\begin{equation*} \label{}
\Lso(U):= \left\{  \mathbf{f} \in \Ls(U) \,: \, \int_U \mathbf{f}(x)\cdot \nabla \phi(x) \,dx = 0 \ \ \mbox{for every} \ \ \phi\in H^1(U) \right\}.
\end{equation*}
Note that~$\Lso(U)$ is the~$L^2$ closure of the vector fields in~$\Ls(U)$ with compact support in~$U$. 

\smallskip
For every $p,q,p^*,q^*\in\Rd$ and bounded Lipschitz domain~$U \subseteq \Rd$, we define
\begin{multline*} 
\mu(U,q^*,p^*) :=  \\
\inf_{u \in H^1(U), \, \g\in \Ls(U)} \fint_U \left( \Fz(\nabla u(x),\g(x),x) - q^* \cdot \nabla u(x) - p^* \cdot \g(x) \right) \, dx 
\end{multline*}
and
\begin{equation*}
\mu_0(U,p,q):= \inf_{v \in H^1_0(U), \, \h\in \Lso(U)} \fint_U \Fz(p+\nabla v(x),q+\h(x),x) \, dx.
\end{equation*}
Since the admissible set for $\mu_0$ has more constraints, we see that, for $p,q,p^*,q^*\in\Rd$, 
\begin{equation} \label{e.mulessnu}
\mu(U,q^*,p^*) \leq p \cdot q^* + p^* \cdot q + \mu_0(U,p,q).
\end{equation}

\smallskip

Up to normalization, $\mu(\cdot,q^*,p^*)$ is \emph{superadditive} and $\mu_0(\cdot,p,q)$ is \emph{subadditive}. Precisely, if $U, U_1,\ldots, U_k$ are bounded domains satisfying
\begin{equation} \label{e.chopping}
U_1,\ldots,U_k \quad \mbox{are pairwise disjoint and} \quad \left| U\setminus (U_1 \cup\cdots \cup U_k) \right| = 0, 
\end{equation}
then
\begin{equation} \label{e.subadd}
\mu(U,q^*,p^*) \geq \sum_{i=1}^k \frac{|U_i|}{|U|} \mu(U_i,q^*,p^*) \quad \mbox{and} \quad 
\mu_0(U,p,q) \leq \sum_{i=1}^k \frac{|U_i|}{|U|} \mu_0(U_i,p,q).
\end{equation}
This is due to the fact that an approximate minimizer for $\mu(U,q^*,p^*)$ gives, by restriction, a candidate for the minimizer of $\mu(U_i,q^*,p^*)$ for each subdomain $U_i$ and, conversely, a minimizing candidate for $\mu_0(U,p,q)$ may be constructed by assembling each of the minimizers of $\mu(U_i,p,q)$ as these agree on the boundaries of the subdomains (up to additive constants).

\smallskip

We next observe that $\mu$ and $\mu_0$ are uniformly bounded on the support of~$\P$: for every bounded open $U\subseteq \Rd$ and $p^*,q^*\in\Rd$, we have 
\begin{equation} \label{e.mubound}
-C\left( K_0^2 + |p^*|^2+|q^*|^2\right) \leq \mu(U,q^*,p^*) \leq CK_0^2, \quad \mbox{$\P$--a.s.}
\end{equation}
To see this, we first use~(P1) and the proof of Lemma~\ref{l.K0stuff} to obtain that the minimum of $F(\cdot,\cdot,x)$ occurs at a point $(p_0(x),q_0(x))$ satisfying 
\begin{equation*}
|p_0(x)|^2 + |q_0(x)|^2 \leq CK_0^2.
\end{equation*}
Moreover, by \eqref{e.infFpq} and (P1),
$$
-K_0^2 \le F(p_0(x), q_0(x), x) \le 0,
$$
so by~\eqref{e.UC} and~\eqref{e.upperUC},
\begin{equation*}
\frac{1}{2\Lambda} \left(  |p-p_0(x)|^2+|q-q_0(x)|^2 \right) - K_0^2  \leq F(p,q,x) \leq \frac{\Lambda}{2} \left(  |p-p_0(x)|^2+|q-q_0(x)|^2 \right).
\end{equation*}
Hence by Young's inequality, we have $\P$--a.s.,
\begin{equation}\label{e.FK0ineq}
\frac{1}{\Lambda} \left(  |p|^2+|q|^2 \right) - CK_0^2  \leq F(p,q,x) \leq \Lambda \left(  |p |^2+|q |^2 \right) + CK_0^2.
\end{equation}
To get the right inequality of~\eqref{e.mubound}, we test with the zero function in the definition of $\mu$ and use the previous inequality. To get the left inequality of~\eqref{e.mubound}, we use the previous inequality and Young's inequality to get, $\P$-a.s.,
\begin{equation*} \label{}
F(p,q,x) - p\cdot q^* - q\cdot p^* \geq -C\left( K_0^2 + |p^*|^2+|q^*|^2\right).
\end{equation*}
In particular, from~\eqref{e.mubound} we have the bound
\begin{equation} \label{e.mubound2}
\left| \mu(U,q^*,p^*) \right| \leq C\left( K_0^2 + |p^*|^2+|q^*|^2\right) \quad \mbox{$\P$--a.s.}
\end{equation}
Similarly, $\P$-a.s.,\ we have the following uniform estimates on $\mu_0$, for every bounded open $U\subseteq \Rd$ and $p,q\in\Rd$: 
\begin{equation} \label{e.mu0bound}
\frac1{\Lambda}\left( |p|^2+|q|^2 \right) - CK_0^2  \leq \mu_0(U,p,q) 
\leq \Lambda \left( |p|^2+|q|^2 \right) + CK_0^2.
\end{equation}
The right inequality was obtained by testing with the zero function and using~\eqref{e.FK0ineq}. The left inequality follows from~\eqref{e.FK0ineq} and Jensen's inequality. Finally, note that 
\begin{equation} \label{e.mu0pq}
\mu_0(U,p,q) \geq p\cdot q. 
\end{equation}
Indeed, for each bounded and Lipschitz~$U\subseteq\Rd$,
\begin{align*} \label{}
 \mu_0(U,p,q) & = \inf_{v \in H^1_0(U), \, \h\in \Lso(U)} \fint_U \Fz(p+\nabla v(x),q+\h(x),x) \, dx \\
 & \geq \inf_{v \in H^1_0(U), \, \h\in \Lso(U)} \fint_{U} (p+\nabla v(x))\cdot (q+\h(x)) \, dx \\
 & = p\cdot q.
\end{align*}

\subsection{Notation for cubes}
\label{ss.cubes}

Throughout the rest of the paper, we work with a triadic cube decomposition of~$\Rd$. For each $n\in\N$ and $x\in\Rd$, we set 
\begin{equation*}
\cu_n(x) := 3^n\left\lfloor 3^{-n} x \right\rfloor + \cu_n,
\end{equation*}
where we recall that $\cu_n$ is the cube $(-3^n/2,3^n/2)^d$.
For each~$n\in\N$, the family $\{ \cu_n(x) \,:\, x\in\Rd \} = \{ z+\cu_n \,:\, z\in3^n\Zd \}$ forms a disjoint partition of~$\Rd$ (up to a Lebesgue null set). Notice in particular that the cube~$\cu_n(x)$ is not centered at~$x$ (unless $x\in3^n\Zd$). Rather, it is the unique cube which contains $x$ with side length~$3^n$ and centered at an element of~$3^n\Zd$.

\smallskip

In order to work with the uniform mixing condition~(P3), it is sometimes convenient to delete a thin mesoscopic boundary strip from the cubes $\cu_n(x)$ so that the cubes are separated from each other. We denote these \emph{trimmed cubes} by
\begin{equation*}
\cut_n := \left( -\frac12 \left(3^n - 3^{n/(1+\beta)}\right) , \frac12 \left( 3^n -  3^{n/(1+\beta)} \right)\right)^d \subseteq \Rd,
\end{equation*}
where we recall that $\beta > 0$ is the exponent appearing in condition~(P3). For every $x \in \R^d$, we let $\cut_n(x) =  3^n\left\lfloor 3^{-n} x \right\rfloor + \cut_n$.
It is clear that, for every $x,y\in\Rd$,
\begin{equation} \label{e.separate}
\cut_n(x) \neq \cut_n(y) \quad \implies \quad \dist\left(\cut_n(x), \cut_n(y) \right) \geq 3^{n/(1+\beta)}.  
\end{equation}
The proportion of volume we have sculpted from $\cu_n$ to create $\cut_n$ is relatively small: 
\begin{equation} \label{e.sculpt}
\frac{\left| \cu_n\setminus \cut_n \right|}{\left| \cu_n \right|} \leq C3^{-n\beta/(1+\beta)}.
\end{equation}

\subsection{Definition of \texorpdfstring{$\overline \mu$}{mu} and \texorpdfstring{$\overline \mu_0$}{mu0}}

Since each triadic cube $\cu_{n+1}$ is the disjoint union of $3^d$ subcubes of the form $z+\cu_n$, with $z\in\Zd$, we obtain from the super- and subadditivity properties and stationarity that~$\E\left[ \mu(\cu_{n},q^*,p^*) \right]$ is nondecreasing in~$n$ and $\E\left[ \mu_0(\cu_{n},p,q) \right]$ is nonincreasing in~$n$. Therefore,
\begin{equation} \label{e.muconv}
\lim_{n\to \infty} \E \left[ \mu(\cu_n,q^*,p^*) \right] = \overline \mu(q^*,p^*):=\sup_{n\in\N}\E \left[ \mu(\cu_n,q^*,p^*)\right]
\end{equation}
and
\begin{equation} \label{e.nuconv}
\lim_{n\to \infty} \E \left[ \mu_0(\cu_n,p,q)\right] = \overline \mu_0(p,q):=\inf_{n\in\N} \E\left[ \mu_0(\cu_n,p,q)\right].
\end{equation}
Note that, by~\eqref{e.mulessnu}, we have 
\begin{equation} \label{e.musord}
\overline\mu(q^*,p^*) \leq \overline \mu_0(p,q)  + p\cdot q^* + p^*\cdot q.
\end{equation}
Next we observe that $\mu$ and $\mu_0$ are monotonic with respect to the trimmed cubes, up to a small error:
\begin{multline} \label{e.trimmono}
\mu(\cut_{n+m},q^*,p^*) \geq 3^{-dm} \sum_{\cut_{n}(x) \subseteq \cut_{n+m}} \mu(\cut_n(x),q^*,p^*)  \\
- C3^{-n\beta/(1+\beta)} \left( K_0^2 + |p^*|^2+|q^*|^2 \right)
\end{multline}
and
\begin{multline} \label{e.trimmono0}
\mu_0(\cut_{n+m},p,q) \leq 3^{-dm} \sum_{\cut_{n}(x) \subseteq \cut_{n+m}} \mu_0(\cut_n(x),p,q)  \\
+ C3^{-n\beta/(1+\beta)} \left( K_0^2 + |p|^2+|q|^2\right).
\end{multline}
To obtain these, we make the obvious partition of $\cut_{n+m}$ into the union of the trimmed cubes $\cut_n(x) \subseteq \cut_{n+m}$ and the extra trimmed region and apply~\eqref{e.mubound2},~\eqref{e.mu0bound} and~\eqref{e.sculpt}. In particular, by taking expectations and using stationarity, we get
\begin{multline} \label{e.Etrimmono}
\E \left[ \mu(\cut_{n+1},q^*,p^*)  \right] \\ \geq \E \left[ \mu(\cut_n(x),q^*,p^*)  \right] - C3^{-n\beta/(1+\beta)} \left( K_0^2 + |p^*|^2+|q^*|^2 \right)
\end{multline}
and
\begin{equation} \label{e.Etrimmono0}
\E \left[ \mu_0(\cut_{n+1},p,q)  \right] \leq \E \left[ \mu_0(\cut_n(x),p,q)  \right] + C3^{-n\beta/(1+\beta)} \left(K_0^2 + |p|^2+|q|^2\right).
\end{equation}

\subsection{The minimizing pairs for \texorpdfstring{$\mu$}{mu} and \texorpdfstring{$\mu_0$}{mu0}}

Throughout the rest of this paper, we denote the minimizing pair of $\mu(U,q^*,p^*)$  by $(u,\g)= (u(\cdot,U,q^*,p^*)$, $\g(\cdot,U,q^*,p^*))$. To be precise, $(u,\g)$ are uniquely defined for each $p^*,q^*\in\Rd$ and bounded open subject $U \subseteq \Rd$ by the conditions that $(u,\g) \in H^1(U) \times \Ls(U)$
\begin{equation*}
\left\{ 
\begin{aligned}
& \mu(U,q^*,p^*) = \fint_U \left( F\left( \nabla u(x),\g(x),x\right)  -  q^*\cdot \nabla u(x) - p^*\cdot \g(x) \right) \, dx,\\
& \fint_{V} u(x)\,dx = 0 \quad \mbox{for every connected component $V$ of $U$.}
\end{aligned}
\right.
\end{equation*}
The existence of the minimizing pair for $\mu$ is immediately obtained from the direct method since the uniform convexity of $F$ guarantees that the functional is weakly lowersemicontinuous. Uniqueness is also a consequence of uniform convexity.  

It is immediate from~\eqref{e.FK0ineq} and~\eqref{e.mubound2} that
\begin{multline} \label{e.boundminimizers}
\fint_U \left( \left| \nabla u(x,U,q^*,p^*)\right|^2 + \left| \mathbf{g}(x,U,q^*,p^*) \right|^2 \right)\,dx \\
\leq C \left( K_0^2 + |p^*|^2+|q^*|^2 \right) \quad \mbox{$\P$--a.s.}
\end{multline}
We denote the minimizing pair of $\mu_0(U,p,q)$ by $(v,\h)$, and it is uniquely defined by $(v,\h) \in H^1_0(U) \times \Lso(U)$ and 
\begin{equation}
\label{e.mu0minimizer}
 \mu_0(U,p,q) = \fint_U F\left( p+\nabla v(x),q+\h(x),x\right) \, dx.
\end{equation}
We sometimes write $v=v(\cdot,U,p,q)$ and $\h=\h(\cdot,U,p,q))$ to display the dependence on $(U,p,q)$. The existence and uniqueness of the minimizing pair for $\mu_0$ is a straightforward application of the direct method, just as for $\mu$. The analogous bound to~\eqref{e.boundminimizers} is
\begin{multline} \label{e.boundminimizers0}
\fint_U \left( \left| \nabla v(x,U,p,q)\right|^2 + \left| \mathbf{h}(x,U,p,q) \right|^2 \right)\,dx \\
\leq C \left( K_0^2 + |p|^2+|q|^2 \right) \quad \mbox{$\P$--a.s.},
\end{multline}
which follows immediately from~\eqref{e.FK0ineq} and~\eqref{e.mu0bound}.

\begin{remark}
The quantities $\mu$ and $\mu_0$ are inherently variational and may not have a simple description in terms of the PDE. We remark that neither $u$ nor $v$ is necessarily a solution of the equation~\eqref{e.pde}, nor is $\g$ or $\h$ necessarily the flux of $u$ or $v$, respectively. For $(v,\h)$, this is because the minimization problem is too restrictive, requiring $\h\in \Lso(U)$ rather than $\h\in\Ls(U)$. For $(u,\g)$, it is because there is no boundary condition, and in particular the minimizer may have boundary conditions other than an affine function with slope $p^*$.
\end{remark}

\subsection{Two convexity lemmas}
Uniform convexity enters into our arguments exclusively through the following two lemmas. We recall that $\Omega$ was defined in \eqref{e:def:Omega}.
\begin{lemma}
\label{l.unifconv}
For every $F\in \Omega$, bounded domain $U\subseteq \Rd$,  $w_1,w_2\in H^1(U)$ and $\mathbf{f}_1,\mathbf{f}_2 \in \Ls(U)$,
\begin{multline*}
\fint_{U} \left( \left| \nabla w_1(x) - \nabla w_2(x) \right|^2 + \left| \mathbf{f}_1(x) - \mathbf{f}_2(x) \right|^2 \right) \, dx \\
 \leq 4\Lambda\left( \fint_{U} F(\nabla w_1(x), \mathbf{f}_1(x),x) \,dx +  \fint_{U} F(\nabla w_2(x), \mathbf{f}_2(x),x) \,dx  - 2\mu(U,0,0) \right).
\end{multline*} 
\end{lemma}
\begin{proof}
Denote $w:=\frac12 w_1+\frac12w_2$ and $\mathbf{f}:=\frac12\mathbf{f}_1+\frac12\mathbf{f}_2$ and compute, using~\eqref{e.UC}:
\begin{align*}
\mu(U,0,0) & \leq \fint_U F(\nabla w(x),\mathbf{f}(x),x) \, dx \\
& \leq \frac 12\fint_U F(\nabla w_1(x),\mathbf{f}_1(x),x) \, dx + \frac 12\fint_U F(\nabla w_2(x),\mathbf{f}_2(x),x) \, dx \\
& \qquad - \frac1{8\Lambda} \fint_U \left( \left| \nabla w_1(x)-\nabla w_2(x) \right|^2+\left| \mathbf{f}_1(x) - \mathbf{f}_2(x) \right|^2 \right)\,dx.
\end{align*}
A rearrangement of this inequality yields the lemma. 
\end{proof}

The next lemma is a converse of the previous one and follows from~\eqref{e.upperUC}. 

\begin{lemma}
\label{l.upperunifconv}
For every $F\in \Omega$, bounded domain $U\subseteq \Rd$,  $w_1,w_2\in H^1(U)$ and $\mathbf{f}_1,\mathbf{f}_2 \in \Ls(U)$,
\begin{multline*}
\fint_{U} F(\nabla w_1(x), \mathbf{f}_1(x),x) \,dx \leq 2   \fint_{U} F(\nabla w_2(x), \mathbf{f}_2(x),x) \,dx  - \mu(U,0,0) \\
 + \frac{\Lambda}{4}\fint_{U} \left( \left| \nabla w_1(x) - \nabla w_2(x) \right|^2 + \left| \mathbf{f}_1(x) - \mathbf{f}_2(x) \right|^2 \right)\, dx .
\end{multline*}
\end{lemma}
\begin{proof}
Set $w:= 2w_2 - w_1$ and $\mathbf{f}:= 2\mathbf{f}_2 - \mathbf{f}_1$, so that $w_2=\frac12 w +\frac12 w_1$ and $\mathbf{f}_2= \frac12\mathbf{f} + \frac12 \mathbf{f}_1$. Using~\eqref{e.upperUC}, we observe that 
\begin{align*}
\lefteqn{\fint_U F(\nabla w_2(x),\mathbf{f}_2(x),x) \, dx} \qquad &  \\
& \geq \frac12\fint_U F(\nabla w(x),\mathbf{f}(x),x) \, dx + \frac12 \fint_U F(\nabla w_1(x),\mathbf{f}_1(x),x) \, dx \\
 & \qquad - \frac{\Lambda}{8} \fint_U \left( \left| \nabla w(x)-\nabla w_1(x) \right|^2+\left| \mathbf{f}(x) - \mathbf{f}_1(x) \right|^2 \right)\,dx. 
\end{align*}
Using 
\begin{equation*}
\mu(U,0,0)\leq \fint_U F(\nabla w(x),\mathbf{f}(x),x) \, dx
\end{equation*}
and rearranging, we get the lemma.
\end{proof}

\begin{remark}\label{r.mu0lemma}
We also obtain versions of the previous two lemmas for functions with affine boundary data. Under the additional assumption that $w_1+w_2 \in H^1_0(U)$ and $\mathbf{f}_1(x)+\mathbf{f}_2(x) \in \Lso(U)$, the conclusion of Lemma~\ref{l.unifconv} may be improved to
\begin{multline}
\label{e.unifconvmu0}
\fint_{U} \left( \left| \nabla w_1(x) - \nabla w_2(x) \right|^2 + \left| \mathbf{f}_1(x) - \mathbf{f}_2(x) \right|^2 \right) \, dx \\
 \leq 4\Lambda\bigg( \fint_{U} F(\nabla w_1(x), \mathbf{f}_1(x),x) \,dx \\
  +  \fint_{U} F(\nabla w_2(x), \mathbf{f}_2(x),x) \,dx  - 2\mu_0(U,0,0) \bigg).
\end{multline} 
Likewise, under the additional assumptions that $2w_1-w_2\in H^1_0(U)$ and $2\mathbf{f}_1-\mathbf{f}_2\in \Lso(U)$, the conclusion of Lemma~\ref{l.upperunifconv} may be strengthened to 
\begin{multline}
\label{e.upperunifconvmu0}
\fint_{U} F(\nabla w_1(x), \mathbf{f}_1(x),x) \,dx \leq 2   \fint_{U} F(\nabla w_2(x), \mathbf{f}_2(x),x) \,dx  - \mu_0(U,0,0) \\
 + \frac{\Lambda}{4}\fint_{U} \left( \left| \nabla w_1(x) - \nabla w_2(x) \right|^2 + \left| \mathbf{f}_1(x) - \mathbf{f}_2(x) \right|^2 \right)\, dx .
\end{multline}
\end{remark}

\begin{remark}
\label{r.dontbotherus}
Appropriate versions of Lemmas~\ref{l.unifconv} and~\ref{l.upperunifconv} as well as~\eqref{e.unifconvmu0} and~\eqref{e.upperunifconvmu0} for arbitrary $(q^*,p^*), (p,q)\in\Rd\times\Rd$ rather than $(0,0)$, are immediate consequences of the former by applying the results to the operators 
\begin{equation*}
(p,q) \mapsto F(\hat p + p, \hat q +q,x) \quad \mbox{and} \quad (p,q) \mapsto F(p, q,x) - q^*\cdot p - p^* \cdot q.
\end{equation*}
Indeed, while these operators do not in general belong to $\Omega$ because they do not satisfy~\eqref{e.infFpq}, they do satisfy the other conditions~\eqref{e.lebemeas},~\eqref{e.UC} and~\eqref{e.upperUC}, and meanwhile~\eqref{e.infFpq} was not used in the proofs of Lemmas~\ref{l.unifconv} and~\ref{l.upperunifconv}. We make use of this fact without further comment.
\end{remark}

\subsection{Some basic estimates on the triadic cubes}
We next present some inequalities which allow us to compare $\mu$ and $\mu_0$ in the trimmed versus the untrimmed cubes. We have, for every $p^*,q^*\in\Rd$ and $n\in\N$,
\begin{equation} \label{e.cutup}
\mu(\cut_n,q^*,p^*) \leq \mu(\cu_n,q^*,p^*) + C3^{-n\beta/(1+\beta)}\left(K_0^2 + |p^*|^2+|q^*|^2 \right).
\end{equation}
Indeed, it follows from superadditivity that
$$
\mu(\cu_n,q^*,p^*)  \geq \frac{|\cut_n|}{|\cu_n|} \mu(\cut_n,q^*,p^*) + \frac{|\cu_n\setminus \cut_n|}{|\cu_n|} \mu(\cu_n\setminus \cut_n,q^*,p^*),
$$
so that
$$
\mu(\cu_n,q^*,p^*) - \mu(\cut_n,q^*,p^*) \ge  \frac{|\cu_n\setminus \cut_n|}{|\cu_n|} \Ll( \mu(\cut_n,q^*,p^*) + \mu(\cu_n\setminus \cut_n,q^*,p^*) \Rr),
$$
and \eqref{e.cutup} follows using~\eqref{e.mubound2} and~\eqref{e.sculpt}. Similarly, for $\mu_0$ we have
\begin{equation} \label{e.cutup0}
\mu_0(\cut_n,p,q) \geq \mu_0(\cu_n,p,q) - C3^{-n\beta/(1+\beta)}\left(K_0^2 +|p|^2+|q|^2 \right).
\end{equation}

We do not have an almost sure inequality which bounds $\mu(\cu_n,q^*,p^*)$ by $\mu(\cut_{n},q^*,p^*)$, but using stationarity, we can show such an inequality in expectation: 
\begin{multline} \label{e.cutdownpre}
\E \left[ \mu(\cu_n,q^*,p^*) \right]  \leq \E \left[ \mu(\cut_{n+1},q^*,p^*) \right] \\ + C\left( \E \left[ \mu(\cut_{n+1},q^*,p^*) \right] - \E \left[ \mu(\cut_n,q^*,p^*) \right] \right)  
+  C3^{-\frac{n\beta}{1+\beta}} \left(K_0^2 + |p^*|^2+|q^*|^2 \right).
\end{multline}
To see this, notice that $\cut_{n+1}$ can be partitioned into $3^d+1$ disjoint connected sets, which consist of the untrimmed cube $\cu_n$, the $3^d-1$ trimmed cubes of the form $\cut_n(x)\subseteq \cut_{n+1}$ which do not intersect $\cu_n$ and a trimmed region of measure at most $C3^{-n\beta/(1+\beta)}|\cut_{n+1}|$. Using the superadditivity of $\mu$ with respect to this partition, taking expectations, using stationary and~\eqref{e.mubound2}, we arrive at
\begin{multline}
3^d \E\Ll[ \mu(\cut_{n+1},q^*,p^*) \Rr] \ge \E\Ll[ \mu(\cu_n,q^*,p^*) \Rr] + (3^d - 1) \E\Ll[ \mu(\cut_n,q^*,p^*) \Rr] \\ - C 3^{-\frac{n\beta}{1+\beta}} \left(K_0^2 + |p^*|^2+|q^*|^2 \right),
\end{multline}
from which \eqref{e.cutdownpre} follows. Using~\eqref{e.cutup}, we may slightly improve~\eqref{e.cutdownpre} to
\begin{multline} \label{e.cutdown}
\E \left[ \mu(\cu_n,q^*,p^*) \right]  \leq \E \left[ \mu(\cut_{n},q^*,p^*) \right] \\ + C\left( \E \left[ \mu(\cut_{n+1},q^*,p^*) \right] - \E \left[ \mu(\cu_n,q^*,p^*) \right] \right)  
+  C3^{-\frac{n\beta}{1+\beta}} \left( K_0^2+|p^*|^2+|q^*|^2\right).
\end{multline}
This also permits us to compare the minimizers of $\mu$ in the trimmed and untrimmed cubes. Writing $U:= \cut_{n} \cup \left( \cu_n\setminus \cut_n\right)$ and applying Lemma~\ref{l.unifconv}, we have
\begin{align*}
\lefteqn{ \fint_{U} \left| (\nabla u,\mathbf{g})(x,U,p^*,q^*) - (\nabla u,\mathbf{g})(x,\cu_n,p^*,q^*) \right|^2\,dx } \qquad & \\
& \leq C \left( \mu(\cu_n,q^*,p^*) -  \mu(U,q^*,p^*)  \right) \\
& \leq C  \left( \mu(\cu_n,q^*,p^*) - \frac{|\cut_n|}{|\cu_n|} \mu(\cut_n,q^*,p^*) - \frac{|\cu_n\setminus \cut_n|}{|\cu_n|} \mu( \cu_n\setminus \cut_n,q^*,p^*) \right) \\
& \leq C \left( \mu(\cu_n,q^*,p^*) -  \mu(\cut_n,q^*,p^*)  \right) + C \left( K_0^2+|p^*|^2+|q^*|^2\right) 3^{-\frac{n\beta}{1+\beta}} .
\end{align*}
Taking expectations and using~\eqref{e.cutdown} and the fact that $$(\nabla u,\mathbf{g})(x,\cut_n,p^*,q^*) = (\nabla u,\mathbf{g})(x,U,p^*,q^*)\vert_{\cut_n},$$ we obtain that
\begin{multline}\label{e.cutlocalize}
\E \left[ \fint_{\cut_n} \left| (\nabla u,\mathbf{g})(x,\cut_n,p^*,q^*) - (\nabla u,\mathbf{g})(x,\cu_n,p^*,q^*) \right|^2\,dx  \right] \\ \leq  C\left( \E \left[ \mu(\cut_{n+1},q^*,p^*) \right] - \E \left[ \mu(\cu_n,q^*,p^*) \right] \right)  
+  C3^{-\frac{n\beta}{1+\beta}} \left( K_0^2+|p^*|^2+|q^*|^2\right).
\end{multline}

\subsection{Some further properties of \texorpdfstring{$\mu$}{mu} and \texorpdfstring{$\mu_0$}{mu0}}

For every $U\subseteq \Rd$, we have that 
\begin{equation} \label{e.mu0UC}
(p,q) \mapsto \mu_0(U,p,q) - \frac1{2\Lambda}\left( |p|^2+|q|^2 \right) \quad \mbox{is convex}
\end{equation}
and
\begin{equation} \label{e.mu0upperUC}
(p,q) \mapsto \mu_0(U,p,q) - \frac\Lambda{2}\left( |p|^2+|q|^2 \right) \quad \mbox{is concave.}
\end{equation}
The first claim follows from~\eqref{e.unifconvmu0}, Remark~\ref{r.dontbotherus} and Jensen's inequality. The second claim follows from~\eqref{e.upperunifconvmu0} and Remark~\ref{r.dontbotherus}. Combined with~\eqref{e.mu0bound}, these imply that $\mu_0(U,\cdot,\cdot)$ satisfies the continuity estimate
\begin{multline} \label{e.mu0cont}
\left| \mu_0(U,p_1,q_1) - \mu_0(U,p_2,q_2) \right| \\ \leq C\left(K_0+|p_1|+|p_2|+ |q_1| +|q_2|\right) \left( |p_1-p_2|+|q_1-q_2| \right).
\end{multline}
The analogous inequality also holds for $\mu$:
\begin{multline} \label{e.mucont}
\left| \mu(U,q_1^*,p_1^*) - \mu(U,q_2^*,p_2^*) \right| \\ \leq C\left(K_0+|p_1^*|+|p_2^*|+ |q_1^*| +|q_2^*|\right) \left( |p_1^*-p_2^*|+|q_1^*-q_2^*| \right).
\end{multline}
To see this, we use~$(u,\g)(\cdot,U,q_1^*,p_1^*)$ as a candidate for attaining the infimum in the definition of~$\mu(U,q_2^*,p_2^*)$ and use~\eqref{e.boundminimizers}. 

\subsection{The homogenized coefficients \texorpdfstring{$\overline F$}{} and functional \texorpdfstring{$\overline\J$}{}}
\label{s.homcoeff}

We define the homogenized variational coefficients $\overline F$ by
\begin{equation*}
\overline F(p,q):= \overline \mu_0(p,q).
\end{equation*}
It is then immediate from~\eqref{e.mu0bound},~\eqref{e.nuconv},~\eqref{e.mu0UC} and~\eqref{e.mu0upperUC} that $\overline F$ grows quadratically and is uniformly convex and $C^{1,1}$ in both variables, that is,  
\begin{equation} \label{e.Fbarbound}
\frac1{\Lambda}\left( |p|^2+|q|^2 \right) - CK_0^2   \leq \overline F(p,q) 
\leq  \Lambda \left( |p|^2+|q|^2 \right) + CK_0^2,
\end{equation}
\begin{equation}
\label{e.FbarUC}
(p,q) \mapsto \overline F(p,q) - \frac1{2\Lambda}\left( |p|^2+|q|^2 \right) \quad \mbox{is convex}
\end{equation}
and
\begin{equation}
\label{e.FbarupperUC}
(p,q) \mapsto \overline F(p,q) - \frac\Lambda{2}\left( |p|^2+|q|^2 \right) \quad \mbox{is concave.}
\end{equation}
Likewise,~\eqref{e.mu0cont} implies that 
\begin{multline}
\label{e.Fbarcont}
\left| \overline F(p_1,q_1) - \overline F(p_2,q_2) \right| \\ \leq C\left(K_0+|p_1|+|p_2|+ |q_1| +|q_2|\right) \left( |p_1-p_2|+|q_1-q_2| \right)
\end{multline}
and hence 
\begin{equation} \label{e.Fbarnablabnd}
\left|  \nabla\overline F(p,q) \right| \leq C(K_0+|p|+|q|).
\end{equation}
From~\eqref{e.FbarUC} and~\eqref{e.FbarupperUC} we also deduce that 
\begin{equation}
\label{e.nablaFbarcont}
\left| \nabla\overline F(p_1,q_1) - \nabla \overline F(p_2,q_2) \right| \leq  C\left( |p_1-p_2|+|q_1-q_2| \right).
\end{equation}
By~\eqref{e.mu0pq}, for every $p,q\in\Rd$,
\begin{equation} \label{e.Fuppq}
\overline F(p,q) \geq p\cdot q.
\end{equation}
In order to see that $\overline F$ variationally represents a uniformly monotone and Lipschitz vector field (which we would then denote by~$\overline\a$), it suffices by~Lemma~\ref{l.werepresent} to check that, for each $p\in\Rd$,
\begin{equation} \label{e.Fminpq}
\inf_{q\in\Rd} \left( \overline F(p,q)-p\cdot q \right) = 0. 
\end{equation}
This is a consequence of~\eqref{e.Fuppq} and the duality between~$\overline \mu$ and~$\overline \mu_0$, for which we require the results in Section~\ref{s.mu}. 
We thus postpone the demonstration of~\eqref{e.Fminpq} and the construction of the homogenized coefficients $\overline \a$ to Section~\ref{s.variationalbar} (see Proposition~\ref{p.variationalbar}).

\smallskip

We define, in each bounded Lipschitz domain $U\subseteq \Rd$, the functional 
\begin{multline} \label{e.barJ}
\overline \J \left[ u,u^* \right]:= \inf\bigg\{ \int_U \left( \overline F(\nabla u(x),\g(x)) - \nabla u(x) \cdot \g(x) \right)\,dx \\
:\, \g\in L^2(U;\Rd), \, -\nabla \cdot \g=u^* \bigg\} .
\end{multline}
Modulo the proof of~\eqref{e.Fminpq}, we have shown that the theory described in Section~\ref{s.varchar} applies to $\overline\J$.

\section{Structure of the proof of Theorem~\ref{t.subopt}}
\label{s.proof}

In this section, we reduce the proof of Theorem~\ref{t.subopt} to two ingredients. The first is the core issue, namely the convergence of the subadditive quantities $\mu$ and $\mu_0$ defined in the previous section. The second is a general, deterministic fact that follows from an oscillating test function argument and essentially allows to recover Theorem~\ref{t.subopt} from the convergence of the subadditive quantities.

\subsection{Two ingredients in the proof of Theorem~\ref{t.subopt}}

We first present the results concerning the convergence of the quantities~$\mu$ and~$\mu_0$, which are proved in the next section. As we will see, the quantities~$\overline \mu$ and~$\overline F$ are dual convex functions and thus each pair~$(p,q)$ is dual to~$\nabla \overline F(p,q)$. For every~$(p,q)\in\Rd\times\Rd$ and bounded domain~$U\subseteq \Rd$, we denote 
\begin{equation*}
\mathcal E(U,p,q) := \left| \mu_0(U,p,q) - \overline F(p,q) \right| + \left| \mu(U, \nabla \overline F(p,q)) - \nabla \overline F(p,q) \cdot(p,q) + \overline F(p,q) \right|.
\end{equation*}
The goal is to show that, for some exponent $\alpha > 0$, 
\begin{equation*}
\P \left[ \mathcal E(\cu_n,p,q) \gtrsim 3^{-n\alpha} \right] \ll 1.
\end{equation*}
We are interested in bounding the probability on the left side as strongly as possible -- as opposed to finding the optimal exponent~$\alpha$, which is much less important for our purposes. In addition, we need an estimate which also possesses some uniformity in $(p,q)$ and allows for translations of the cubes. 

\smallskip

The following theorem possesses each of these desired properties. It can be compared to~\cite[Theorem 3.1]{AS}. Its proof is the focus of Section~\ref{s.mu}.

\begin{theorem}[Convergence of the subadditive quantities]
\label{t.muconv}
Fix $M,R\geq 1$, $\theta \in (0,\beta)$ and $\tau \geq 1$. There exist~$s_0(d,\Lambda,\beta,\theta,\tau)\geq1$ and $C(d,\Lambda,\beta,C_3,\theta,\tau) \geq 1$ such that, for every $n\in\N$, $s \in [s_0,\infty)$ and $t\in [1,\infty)$,
\begin{equation}\label{e.muconv0}
\P \left[  \sup_{p,q\in B_{M3^{n\tau/s}}} \ \sup_{y\in B_{R}}   \frac{ \mathcal E(y+\cu_n,p,q) }{ (K_0^2+|p|^2+|q|^2) 3^{-n\theta/s}} \geq  t \right]  \leq CR^dM^{2d}t^{-s}.
\end{equation}
Moreover, if~(P4) holds, then for some $\alpha(d,\Lambda,\beta)>0$ and with $C$ depending additionally on $(\gamma,C_4)$, we have the following stronger estimate: for every exponent $s\in (0,d\gamma/(d+\gamma))$, $n\in\N$ and $t\geq 1$,
\begin{multline}
\label{e.muconvboom0}
\P \left[  \sup_{p,q\in B_{M3^{n}}} \ \sup_{y\in B_{R3^{n}}} \frac{\mathcal E(y+\cu_n,p,q) }{ (K_0^2+|p|^2+|q|^2) 3^{-n\alpha (d\gamma/(d+\gamma)-s) }} \geq C t \right] \\
\leq CR^dM^{2d}\exp\left( -3^{ns} t \right).
\end{multline}
\end{theorem}

The second ingredient in the proof of Theorem~\ref{t.subopt} is the following proposition (which can be compared to~\cite[Proposition 4.1]{AS}) the proof of which is the focus of Section~\ref{s.blackbox}. It is a purely deterministic statement which asserts that estimates on $\mathcal E$ on mesoscopic cubes (of exactly the sort appearing in Theorem~\ref{t.muconv}) imply quenched estimates for the error in homogenization for general macroscopic Dirichlet problems. Thus it essentially reduces Theorem~\ref{t.subopt} to Theorem~\ref{t.muconv}. 

\begin{proposition}[Deterministic bounds for Dirichlet problems]
For every bounded Lipschitz domain $U_0\subseteq \Rd$ and $\delta > 0$, there exist $C(d,\Lambda,U_0,\beta,C_3,\delta)$ and exponents $\alpha(d,\Lambda,\beta,\delta) >0$ and $\rho(d,\Lambda,\delta)>1$ such that the following holds. For every $F\in\Omega$, $m,n\in\N$, $U:= 3^{n+m}U_0$ and $f\in W^{1,2+\delta}(U)$, if $u,\uhom \in f+H^1_0(U)$ are such that
\begin{equation} \label{e.Dirmins}
 \J \left[u, 0 \right] = \overline\J\left[ \uhom,0 \right] =  0,
\end{equation}
then for every integer $l$ satisfying $n \le l \le m+n$,
\begin{equation} \label{e.exhale} 
3^{-2(n+m)} \fint_U \left| u(x) - \uhom(x) \right|^2\, dx \\
 \leq C M^2 \left(  \mathcal E'_{n,m,M} + 3^{-2m} +  3^{-\alpha(n+m-l)} \right),
\end{equation} 
where we denote
\begin{equation} \label{e.M}
M:=  K_0 + \left( \fint_U \left| \nabla f(x) \right|^{2+\delta}\, dx \right)^{\frac1{2+\delta}}
\end{equation}
and 
\begin{equation} \label{e.Eprimestat}
\mathcal E'_{n,m,M} = \mathcal E'_{n,m,M}(U_0):= \left( \fint_{U} \, \left( \sup_{p,q \in B_{CM3^{dm/2}}} \frac{\mathcal E(x+\cu_n,p,q) }{K_0^2+|p|^2+|q|^2}    \, \right)^{\rho}\,dx  \right)^{\frac1\rho}.
\end{equation}
\label{p.blackbox}
\end{proposition}

\subsection{The proof of Theorem~\ref{t.subopt}}

We next show that Theorem~\ref{t.subopt} is a consequence of the previous two statements. 

\begin{proof}[{Proof of Theorem~\ref{t.subopt}}]

\smallskip

\emph{Step 1.} We verify the weaker estimate, i.e.~\eqref{e.Xsst} and~\eqref{e.wkee}. We allow~$C(d,\Lambda,\beta,C_3,\delta,U_0,\theta)\geq1$ to vary in each occurrence. 
 
\smallskip

Fix $R\geq 1$. Also set $\theta' := \theta + \frac12 (\beta - \theta)$ and fix $\tau \geq 1$ and $s\geq s_0$ to be selected below. We may suppose without loss of generality that $R=3^k$ for some $k\in\N$. Write $k=n+m$ with $n,m\in\N$ and $m$ chosen as large as possible subject to the constraint 
\begin{equation*} \label{}
\frac{md}{2} \leq \frac{n\tau}s,
\end{equation*}
and pick~$l:= n + \lfloor m/2 \rfloor$. Observe that, with these choices of $n,m$ and $l$, the last factor in the second term on the right side of~\eqref{e.exhale} is estimated by
\begin{equation*} \label{}
3^{-2m} + 3^{-\alpha(n+m-l)} \leq CR^{-\theta'/s},
\end{equation*}
provided that $\tau$ and $s_0$ are chosen sufficiently large. Indeed, since 
$$
\alpha (n+m - l) \sim \frac{\alpha \tau}{2sd} n , \qquad \frac{\theta'}{s} (n+m) \sim \Ll( 1+ \frac{\tau}{sd}   \Rr)  \frac{\theta'}{s} n \qquad (n \to \infty), 
$$
it suffices to check that
$$
\frac{\alpha \tau}{2d} > \Ll( 1+\frac{\tau}{sd}   \Rr) \theta' ,
$$
and this is satisfied provided that $\tau > 2d\theta'/\alpha$ and $s_0\geq C$ is sufficiently large.

\smallskip

Therefore, Proposition~\ref{p.blackbox} gives 
\begin{equation*} \label{}
R^{-2} \fint_U  \left| u(x) - \uhom(x) \right|^2\, dx  \leq C M^2 \mathcal E'_{n,m,M} + C M^2 R^{-\theta'/s},
\end{equation*}
where~$\mathcal{E}'_{n,m,M}$ is defined in~\eqref{e.Eprimestat}.
By Jensen's inequality and stationarity, we find that, for every $s\geq \rho$, where $\rho(d,\Lambda,\delta)>1$ is the exponent in the statement of~Proposition~\ref{p.blackbox}, we have 
\begin{align*}
\E \left[ \Ll( \mathcal E'_{n,m,M_0} \Rr)^s\right] & = \E \left[ \left( \fint_U \left( \sup_{p,q\in B_{CM_03^{md/2}}} \,\frac{\mathcal E(x+\cu_n,p,q)}{K_0^2+|p|^2+|q|^2}\right)^{\rho}\,dx \right)^{\frac{s}\rho} \right]  \\
& \leq  \E \left[ \fint_U \left( \sup_{p,q\in B_{CM_03^{md/2}}}\, \frac{\mathcal E(x+\cu_n,p,q)}{K_0^2+|p|^2+|q|^2}\right)^{s}\,dx  \right] \\
& \leq \E \left[ \left( \sup_{x\in B_{\sqrt{d}}}\, \sup_{p,q\in B_{CM_03^{md/2}}} \, \frac{\mathcal E(x+\cu_n,p,q)}{K_0^2+|p|^2+|q|^2}\right)^{s} \right]. 
\end{align*}
We next use the fact that, for a nonnegative random variable $X$,  $$\E[X^s] = \int_0^\infty t^{s-1} \P[X > t] \, dt,$$ and apply Theorem~\ref{t.muconv} (replacing $s$ by $s+1$ and $\theta$ by $\theta'$ there, and increasing $s_0$ as necessary) to obtain, for every $s\geq s_0(d,\Lambda,\beta,\theta,\delta)$,
\begin{equation*}
\E \left[ \left( \sup_{x\in B_{\sqrt{d}}}\,  \sup_{p,q\in B_{CM_03^{md/2}}} \,\frac{\mathcal E(\cu_n,p,q)}{K_0^2+|p|^2+|q|^2}\right)^{s} \right] \leq CM_0^{2d} 3^{-n\theta'}.
\end{equation*} 
Setting
\begin{equation*} \label{}
\X_{s,n,M_0}:= C\left(1+ 3^{n\theta'/s} \mathcal{E}'_{n,m,M_0}\right),
\end{equation*}
we therefore obtain, under the assumption that $M\leq M_0$, the estimate
\begin{equation*} \label{}
R^{-2} \fint_U  \left| u(x) - \uhom(x) \right|^2\, dx  \leq  \X_{s,n,M_0} M^{2} 3^{-n\theta'/s},
\end{equation*}
with $\X_{s,n,M_0}$ satisfying $\E[\X_{s,n,M_0}^s] \leq CM_0^{2d}$. We now define 
\begin{equation*} \label{}
{\X}_{s,M_0} := \sup_{j\in\N} 3^{-j(\beta-\theta)/4s} \X_{s,j,M_0}.
\end{equation*}
Clearly, $\E[{\X}_{s,M_0}^s] \leq CM_0^{2d}$. Note also that the definition of~${\X}_{s,M_0}$ does not depend on~$R$, and
\begin{equation} \label{e.almostthere}
R^{-2} \fint_U  \left| u(x) - \uhom(x) \right|^2\, dx  \leq  {\X}_{s,M_0} M^{2} 3^{-n\theta''/s},
\end{equation}
where $\theta'':= \theta+\frac14(\beta-\theta)$. If $s_0$ is large enough (depending on the appropriate quantities), then by our choice of $n$ and $m$, we have
$3^{-n\theta''} \leq R^{-\theta}$. To remove the dependence of ${\X}_{s,M_0}$ on $M_0$, we define
\begin{equation*}
{\X}_s:= \sup_{k\in\N} 2^{-2dk/s}\X_{s+1,2^k}.
\end{equation*}
It is clear that~\eqref{e.Xsst} holds and, translating~\eqref{e.almostthere} in terms of~$\X_s$ yields~\eqref{e.wkee}. The proof is now complete.

\smallskip

\emph{Step 2.} Under assumption~(P4), we verify the stronger estimates~\eqref{e.Xsexp},~\eqref{e.stree}. Here we fix an exponent~$s\in (0,d\gamma/(d+\gamma))$ and allow $C$ and $c$ to depend additionally on $(\gamma,C_4,s)$. We pick $s_1$ and $s_2$ such that $s<s_1<s_2<d\gamma/(d+\gamma)$ with the gaps between these bounded by $c$. As above we may suppose that $R=3^k$ for some $k\in\N$. We write $k=n+m$ where $n,m\in\N$ are chosen so that $m$ is as large as possible such that $n\geq m$, and 
\begin{equation}\label{e.mnconditions}
(m+n)s_1 < ns_2 \quad \mbox{and} \quad md\gamma /(d+\gamma) \leq n\alpha \left(d\gamma/(d+\gamma)-s_2\right).
\end{equation}
Note that, in addition to $R$, the integers $m$ and $n$ depend only on $s$. Choosing $l = n + \lfloor m/2 \rfloor$ in~\eqref{e.exhale} yields
\begin{multline} \label{e.almostthereY}
R^{-2} \fint_U  \left| u(x) - \uhom(x) \right|^2\, dx  \leq C M^2 \mathcal E'_{n,m,M} + C M^2 R^{-\alpha/s} \\
\leq CM^2\left( 1+ \Y_{s,M_0} R^{-s} \right)R^{-\alpha/s},
\end{multline}
provided that $M\leq M_0$, where we have defined 
\begin{equation*}
\Y_{s,M_0}:= \sup_{n\in\N} \left( R^{s} \left( 3^{md\gamma/(d+\gamma)} \mathcal E'_{n,m,M_0} -C  \right)_+ \right).
\end{equation*}
(Here and in what follows we think of $R$ and $m$ as functions of $n$.) We next study the integrability of~$\Y_{s,M_0}$. According to~\eqref{e.muconvboom0}, for every $t \ge 1$,
\begin{align*}
\P \left[   \mathcal E'_{n,m, M_0} \geq  C3^{-(n+m)\alpha (d\gamma/(d+\gamma)-s_2)} t \right] \leq CM_0^{2d} \exp\left( -3^{(n+m)s_2}t \right).
\end{align*}
Using~\eqref{e.mnconditions} and rearranging this leads to the bound
\begin{equation*}
\P \left[ R^{s}\left( 3^{md\gamma/(d+\gamma)} \mathcal E'_{n,m,M_0} -C  \right)_+ \geq  Ct \right] \leq CM_0^{2d} \exp\left( -c3^{-(s_1-s)} t\right).
\end{equation*}
Taking a union bound and summing this over $n\in\N$ yields
\begin{equation*}
\P \left[ \Y_{s,M_0} \geq Ct \right] \leq CM_0^{2d} \exp\left( -ct \right).
\end{equation*}
Defining $\Y_s:=c\sup_{k} k \Y_{s,2^k}$ and integrating the previous inequality, we obtain~\eqref{e.Xsexp}. The inequality~\eqref{e.stree} is obtained by expressing~\eqref{e.almostthereY} in terms of $\Y_s$. This completes the proof of the theorem. 
\end{proof}

\section{Convergence of the subadditive quantities}
\label{s.mu}

In this section, we prove Theorem~\ref{t.muconv}. We focus the majority of our effort to obtain the following slightly weaker statement. In the final subsection, we derive Theorem~\ref{t.muconv} from it. 

\begin{proposition}
\label{p.muconv}
Fix $(q^*,p^*)\in \Rd\times \Rd$. There exist a unique pair $$(\overline P(q^*,p^*),\overline Q(q^*,p^*)) \in \Rd \times \Rd$$ such that
\begin{equation} \label{e.dualatlimit}
\overline \mu(q^*,p^*) + p^*\cdot \overline Q + q^*\cdot \overline P = \overline \mu_0(\overline P,\overline Q)
\end{equation}
and, for every $\theta \in (0,\beta)$, an exponent $s_0(d,\Lambda,\beta,\theta)>0$ and~$C(d,\Lambda,\beta,C_3,\theta)>0$ 
such that, for all $s\in [s_0,\infty)$, $R\geq 1$, $n\in\N$ and $t\geq 1$,
\begin{equation}
\label{e.muconverge}
\P \left[ \sup_{y\in B_R} \mathcal E_*(y+\cu_n,q^*,p^*)\geq C(K_0^2+|p^*|^2+|q^*|^2) 3^{-n\theta / s}t \right] \leq C R^d t^{-s},
\end{equation}
where we denote
\begin{equation} \label{e:def:Estar}
\mathcal E_*(U,q^*,p^*) := \left| \mu(U,q^*,p^*) - \overline \mu(q^*,p^*) \right| + \left| \mu_0(U,\overline P,\overline Q) - \overline \mu_0(\overline P,\overline Q)\right|.
\end{equation}
Moreover, if we assume~(P4), then, for every exponent $s\in (0,d\gamma/(d+\gamma))$ and with $C$ depending additionally on $(s,\gamma,C_4)$, we have the following stronger estimate: for all $R\geq 1$, $n\in\N$ and $t\geq 1$,
\begin{multline}
\label{e.muconvboom}
\P \left[ \sup_{y\in B_R} \mathcal E_*(y+\cu_n,q^*,p^*)\geq C(K_0^2+|p^*|^2+|q^*|^2) 3^{-n\min\{ \alpha, d\gamma/(d+\gamma)-s\}}t \right] \\
\leq CR^d\exp\left( - 3^{ns} t \right).
\end{multline}
\end{proposition}

\subsection{Reduction to the case \texorpdfstring{$p^*=q^*=0$}{p*=q*=0}}
\label{ss.reducpush}

It suffices to prove Proposition~\ref{p.muconv} for $p^* = q^* = 0$. Indeed, given $p^*,q^* \in \R^d$, let
$$
F_{p^*,q^*}(p,q,x) := F(p+p^*,q+q^*,x) - p \cdot q^* - q\cdot p^*  -   p^*\cdot q^*.
$$
Note that~$F_{p^*,q^*}$ belongs to $\Omega$ and satisfies the same assumptions as $F$, except that the constant $K_0$ in (P1) must be replaced by $K_0 + |p^*|+|q^*|$. Applying Proposition~\ref{p.muconv} to $F_{p^*,q^*}$ (to be more precise, to the pushforward of $\P$ under the map $F\mapsto F_{p^*,q^*}$) at $(0,0)$ then yields the general result.

\subsection{Flatness of minimizers}

The purpose of this subsection is to prove Lemma~\ref{l.iterable} (stated below), which is the key step in the proof of Proposition~\ref{p.muconv}.  The main idea is to show that if the difference in the expectation of $\mu$ between two successive triadic scales is small, then the minimizers of $\mu$ must be very flat. This invites a comparison to corresponding minimizers of $\mu_0$, allowing us to show that the expectation of $\mu$ is close to its limit. 

%

\smallskip

In order to lighten notation, we simply write $\mu(U)$ for $\mu(U,0,0)$, $\overline \mu$ for $\overline \mu(0,0)$, and $(u,\g)(\cdot,U)$ for $(u,\g)(\cdot,U,0,0)$. We denote the spatial averages of the minimizing pair $(u,\g)(\cdot,U)$ of $\mu(U)$ by
\begin{equation*} \label{}
P(U):= \fint_{U} \nabla u(x,U)\, dx \quad \mbox{and} \quad Q(U):= \fint_U \g(x,U)\, dx,
\end{equation*}
and we use the shorthand notation
\begin{equation*}
\overline P_n:= \E \left[ P(\cut_n) \right] \quad \mbox{and} \quad \overline Q_n:= \E \left[ Q(\cut_n) \right].
\end{equation*}
Throughout this section, we denote the difference between the expected values of~$\mu$ at the two successive triadic scales $n+1$ and $n$ (with the smaller cube trimmed) by 
\begin{equation} \label{e.taun}
\tau_n := \E \left[ \mu(\cu_{n+1}) \right] - \E \left[ \mu(\cut_n) \right]. 
\end{equation}
and we denote errors which will accumulate due to mixing and trimming at scale~$n$ by
\begin{equation*}
\kappa_n:= K_0^2 3^{-n\beta/8(1+\beta)}+ K_0^23^{-n/3}.
\end{equation*}

\smallskip

The following adaptation of~\cite[Lemma 3.2]{AS} is the first step in the argument for Proposition~\ref{p.muconv}. The lemma states that the variances of $P(\cut_n)$ and $Q(\cut_n)$ are controlled by a multiple of  $\tau_n+\kappa_n$. The proof uses the weak mixing condition~(P3) for the first time in the paper. In preparation for its application, we rephrase~(P3) in terms of the following covariance estimate (this is proved in Appendix~\ref{s.mixing}, see~\eqref{e.covalphamix}): for every $U,V\subseteq\Rd$, $\F_U$--measurable random variable $X$ and $\F_V$--measurable random variable $Y$, we have
\begin{equation}\label{e.howweuseP3}
\cov\left[ X,Y \right]  \leq 4 C_3 \| X \|_\infty \|Y\|_{\infty} (1+\dist(U,V))^{-\beta},
\end{equation}
where $\| X \|_{\infty}$ denotes the $\P$--essential supremum of $|X|$: 
\begin{equation*}
\| X \|_\infty := \inf\left\{ \lambda > 0\,:\, \P \left[ |X| > \lambda \right] = 0 \right\}.
\end{equation*}

\begin{lemma}
\label{l.flat}
There exist $C(d,\Lambda,\beta) \geq 1$ and $C'(d,\Lambda,\beta,C_3)\geq 1$ such that, for every $n\in\N$,
\begin{equation*}
\E  \left[ \left| P(\cut_n) - \overline P_n \right|^2 \right] + \E\left[ \left| Q(\cut_n) - \overline Q_n \right|^2 \right] \leq C \tau_n+C'\kappa_n.
\end{equation*}
\end{lemma}

\begin{proof}
In this argument, $C$ denotes a constant depending only on $(d,\Lambda,\beta)$ and $C'$ denotes a constant depending only on~$(d,\Lambda,\beta,C_3)$; these may vary in each occurrence. Fix $n\in\N$,  a unit direction $e\in \partial B_1$, a smooth vector field $\mathbf{f}:\Rd \to \Rd$ and a smooth function $\varphi:\Rd \to \R$ satisfying the following:
\begin{equation*}
\left\{ 
\begin{aligned}
& \mathbf{f} \ \ \mbox{and} \ \ \varphi \quad\mbox{have compact support in} \ \cut_{n+1},\\
& \mathbf{f} = \nabla \varphi = e \quad \mbox{in} \ \cu_n, \\
& \div \mathbf{f} = 0, \\
& \left|  \mathbf{f}  \right| + \left| \nabla \varphi \right| \leq C \quad \mbox{in} \ \Rd.
\end{aligned}
\right.
\end{equation*}
Since $\mathbf{f}$ and $\g(\cdot,\cu_{n+1})$ are divergence-free, we have
\begin{equation}
\label{e.divfree}
\fint_{\cu_{n+1}} \mathbf{f}(x) \cdot \nabla u(x,\cu_{n+1})\,dx  =   \fint_{\cu_{n+1}} \nabla \varphi(x) \cdot \g(x,\cu_{n+1})\, dx  = 0.
\end{equation}

\emph{Step 1.} We localize $\nabla u(\cdot,\cu_{n+1})$ and $\g(\cdot,\cu_{n+1})$ to the trimmed subcubes. The precise claim is that 
\begin{multline} \label{e.gradloc}
\E \left[ \fint_{U\cup V} \left(  \left| \nabla u(x,\cu_{n+1}) - \nabla u(x,U\cup V)\right|^2 +  \left| \g(x,\cu_{n+1}) - \g(x,U\cup V)\right|^2 \right)\,dx \right] \\
\leq C\tau_n+C\kappa_n,
\end{multline}
where
\begin{equation*} \label{}
U:= \bigcup_{z\in\{ -3^n,0,3^n\}^d} \left( z+\cut_n \right) \qquad \mbox{and} \qquad V:= \cu_{n+1} \setminus \overline U,
\end{equation*}
and $\overline U$ denotes the closure of $U$. 
There are $3^d+1$ connected components of $U\cup V$, which are $V$ and the cubes of the form $\cut_n(x) \subseteq \cu_{n+1}$. Thus for every $x\in \cu_{n+1}$, we have
\begin{equation}\label{e.restrict}
\left\{ \begin{aligned}
& (u,\g)(\cdot,U\cup V) \vert_{\cut_n(x) }  = (u,\g)(\cdot,U)\vert_{\cut_n(x) } = (u,\g)(\cdot,\cut_n(x)) , \\
& (u,\g) (\cdot,U\cup V)\vert_{V} =  (u,\g)(\cdot,V).
\end{aligned} \right.
\end{equation}
In particular, 
\begin{equation*}
 \mu(U)  = 3^{-d} \sum_{\cut_{n}(x)\subseteq \cu_{n+1}} \mu(\cut_{n}(x))
\end{equation*}
and
\begin{equation}
\label{e:addit}
 \mu(U \cup V)  = \frac{|V|}{|U\cup V|} \mu(V) + \frac{|U|}{|U\cup V|} \mu(U).
\end{equation}
By stationarity,
\begin{equation*}
\E\left[ \mu(U)\right]  =  \E\left[ \mu(\cut_{n}) \right].
\end{equation*}
Notice that, by~\eqref{e.sculpt}, the volume of~$V$ is small in proportion to that of~$\cu_{n+1}$:
\begin{equation} \label{e.Vsculpt}
\left| V \right| \leq C 3^{-n\beta/(1+\beta)} \left| \cu_{n+1} \right|.
\end{equation}
It follows from~\eqref{e:addit} and \eqref{e.mubound2} that
\begin{equation*}
\mu(U \cup V)  = \mu(U) + \frac{|V|}{|U\cup V|}\left(  \mu(V) - \mu(U)\right)   \geq \mu(U) - CK_0^2 3^{-n\beta/(1+\beta)}.
\end{equation*}
By Lemma~\ref{l.unifconv} and the previous inequality,
\begin{multline*}
\lefteqn{ \E \left[ \fint_{U\cup V} \left(  \left| \nabla u(x,\cu_{n+1}) - \nabla u(x,U\cup V)\right|^2 +  \left| \g(x,\cu_{n+1}) - \g(x,U\cup V)\right|^2 \right)\,dx \right]  } \\
\begin{aligned}
& \leq C \left(  \E \left[  \mu(\cu_{n+1}) \right] - \E \left[ \mu(U \cup V) \right] \right)  \\
& \leq C \left(  \E \left[  \mu(\cu_{n+1}) \right] - \E \left[ \mu(U) \right] + CK_0^2  3^{-n\beta/(1+\beta)} \right)  \\
& = C( \tau_n + C\kappa_n),
\end{aligned}
\end{multline*}
and this is~\eqref{e.gradloc}.

\smallskip

\emph{Step 2.} Using~\eqref{e.divfree},~\eqref{e.gradloc} and \eqref{e.restrict}, we compute
\begin{multline*}
\var\left[ \int_{V} \mathbf{f}(x) \cdot \nabla u(x,V)\,dx + \sum_{\cut_n(y) \subseteq \cu_{n+1}}  \int_{\cut_n(y)} \mathbf{f}(x) \cdot \nabla u(x,\cut_n(y))\,dx \right]  \\
\begin{aligned}
 & = \var\left[ \int_{U\cup V} \mathbf{f}(x) \cdot \nabla u(x,U\cup V)\,dx \right] \\
 &  \leq 2\var\left[ \int_{\cu_{n+1}} \mathbf{f}(x) \cdot \nabla u(x,\cu_{n+1})\,dx \right]  + C(\tau_n +C\kappa_n) \left|\cu_{n+1}\right|^2  \\
 & = C( \tau_n +C\kappa_n) \left|\cu_{n+1}\right|^2.
\end{aligned}
\end{multline*}
By~\eqref{e.boundminimizers} and~\eqref{e.Vsculpt}, we have 
\begin{equation*} \label{}
\int_{V} \left| \nabla u(x,V)\right|^2\,dx\leq C|V|K_0^2 \leq CK_0^2 3^{-n\beta/(1+\beta)} \left|\cu_{n+1}\right| = C \kappa_n  \left|\cu_{n+1}\right|.
\end{equation*}
The previous two inequalities yield
\begin{equation*} \label{}
\var\left[ \sum_{\cut_n(y) \subseteq \cu_{n+1}}  \fint_{\cut_n(y)} \mathbf{f}(x) \cdot \nabla u(x,\cut_n(y))\,dx  \right] \leq C ( \tau_n+C\kappa_n). 
\end{equation*}
Now we expand the variance using the identity
\begin{multline*}
\var\left[ \sum_{\cut_n(y) \subseteq \cu_{n+1}}  \fint_{\cut_n(y)} \mathbf{f}(x) \cdot \nabla u(x,\cut_n(y))\,dx  \right] \\ = \sum_{\cut_{n}(y),\cut_n(z) \subseteq \,\cu_{n+1} } \cov\bigg[ \fint_{\cut_n(y)} \mathbf{f}(x) \cdot \nabla u(x,\cut_n(y))\,dx \,, \\ \, \fint_{\cut_n(z)} \mathbf{f}(x) \cdot \nabla u(x,\cut_n(z))\,dx \bigg].
\end{multline*}
Recall that by \eqref{e.boundminimizers} and the definition of $\mathbf f$, the random variables appearing in the covariances above are $\P$-a.s.~bounded by $C K_0$. In view of~\eqref{e.separate}, we can apply the mixing condition~(P3) in the form of~\eqref{e.howweuseP3} to obtain, for every~$\cut_{n}(y) \neq \cut_n(z)$,
\begin{multline*}
\left| \cov\left[ \fint_{\cut_n(y)} \mathbf{f}(x) \cdot \nabla u(x,\cut_n(y))\,dx \ , \ \fint_{\cut_n(z)} \mathbf{f}(x) \cdot \nabla u(x,\cut_n(z))\,dx \right] \right| \\
 \leq C'K_0^2 3^{-n\beta/(1+\beta)} = C'K_0^2\kappa_n.
\end{multline*}
To sum up, we have proved the estimate
\begin{equation*} \label{}
\sum_{\cut_n(y) \subseteq \cu_{n+1}}  \var\left[  \fint_{\cut_n(y)} \mathbf{f}(x) \cdot \nabla u(x,\cut_n(y))\,dx  \right] \leq C(\tau_n+C'\kappa_n).
\end{equation*}
The previous inequality and the fact that $\mathbf{f}\equiv e$ in $\cut_n$ imply
\begin{align*}
\var\left[ e \cdot P(\cut_n) \right] & =  \var\left[ \fint_{\cut_n} \mathbf{f}(x) \cdot \nabla u(x,\cut_n)\,dx \right]  \\
& \leq \sum_{\cut_n(y) \subseteq \cu_{n+1}}  \var\left[  \fint_{\cut_n(y)} \mathbf{f}(x) \cdot \nabla u(x,\cut_n(y))\,dx  \right]  \leq C(\tau_n+C'\kappa_n).
\end{align*}
An almost identical argument, starting from the second equality of~\eqref{e.divfree} rather than the first and replacing each occurrence of $\mathbf{f}(x) \cdot \nabla u$ by $\nabla \varphi(x) \cdot \g$, gives the estimate 
\begin{equation*}
 \var\left[ e \cdot Q(\cut_n) \right] \leq C(\tau_n+C'\kappa_n).
\end{equation*}
Summing the previous two inequalities over $e\in \{ e_1,\ldots, e_d\}$ yields the lemma. 
\end{proof}

Note that~\eqref{e.cutlocalize} permits us to obtain the conclusion of Lemma~\ref{l.flat} for the untrimmed cubes: there exists $C(d,\Lambda,\beta) \geq 1$ and $C'(d,\Lambda,\beta,C_3)\geq 1$ such that, for every $n\in\N$,
\begin{equation} \label{e.flatuntrimmed}
\E  \left[ \left| P(\cu_n) - \overline P_n \right|^2 \right] + \E\left[ \left| Q(\cu_n) - \overline Q_n \right|^2 \right] \leq C \tau_n+C'\kappa_n.
\end{equation}

The next lemma is the key step in the proof of Proposition~\ref{p.muconv}. It allows us to estimate the expected difference between $\mu$ and its limit $\overline\mu$ by the expected difference between $\mu$ at two successive scales. 

\begin{lemma}
\label{l.iterable}
There exist $C(d,\Lambda,\beta) \geq 1$ and $C'(d,\Lambda,\beta,C_3)\geq 1$ such that, for every $n\in\N$,
\begin{equation}
\label{e.iterable}
\E \left[ \mu_0(\cut_{2n},\overline P_n,\overline Q_n ) \right] \leq \E \left[ \mu(\cu_n) \right]  + C(\tau_n+C'\kappa_n).
\end{equation}
\end{lemma}
\begin{proof}
The convention for the constants $C$ and $C'$ is the same here as in the previous lemma. 


\smallskip

Fix $n\in\N$. In order to estimate the quantity $\mu_0(\cut_{2n},\overline P_n,\overline Q_n )$ from above, we construct a candidate for minimizing the energy which satisfies the appropriate affine boundary conditions. Precisely, it suffices to exhibit $(v,\h)\in H^1_0(\cut_{2n}) \times \Lso(\cut_{2n})$ satisfying 
\begin{equation} 
\label{e.exhibition}
\E \left[ \fint_{\cut_{2n}} F \left( \overline P_n+ \nabla v(x), \overline Q_n + \h(x),x\right) \, dx \right] \\
\leq \E \left[ \mu(\cu_n) \right]  
+C(\tau_n+C'\kappa_n).
\end{equation}

\smallskip

\emph{Step 1.} The construction of the candidate $(v,\h)\in H^1_0(\cut_{2n}) \times \Lso(\cut_{2n})$, which we build by patching the minimizers for $\mu$ on the overlapping family of cubes $\{ z+\cu_{n+1}\,:\, z\in 3^n\Zd\}.$  We first build a partition of unity subordinate to these cubes by setting
\begin{equation*}
\psi(y):= \int_{\cu_{n}} \psi_0(y-x)\,dx,
\end{equation*}
where $\psi_0\in C^\infty_c(\Rd)$ is a smooth, even function satisfying
\begin{equation*}
0\leq \psi_0\leq C3^{-dn}, \quad \int_{\Rd} \psi_0(x)\,dx=1, \quad |\nabla \psi_0|\leq C3^{-(d+1)n}, \quad \supp \psi_0 \subseteq \cu_n.
\end{equation*}
We note that $\psi$ is smooth, even and supported in $\cu_{n+1}$, satisfies $0\leq \psi\leq 1$, and the translates of it form a partition of unity: for each $x\in\Rd$,
\begin{equation}\label{e.partunity}
\sum_{z\in3^n\Zd} \psi(x-z) = 1.
\end{equation}
Moreover, we have 
\begin{equation} \label{e.psiderv}
\sup_{x\in\Rd} \left| \nabla \psi(x) \right| \leq C3^{-n}.
\end{equation}
We also introduce two smooth cutoff functions $\xi,\zeta \in C^\infty_c(\cu_{2n})$ satisfying
\begin{multline} \label{e.xicutoff}
0\leq \xi\leq 1, \quad \xi\equiv 1 \ \mbox{on} \ \{ x\in\cu_{2n}\,:\, \dist(x,\partial \cu_{2n}) > C 3^{2n/(1+\delta)} \}, \\
 \xi \equiv 0 \ \mbox{on} \  \left\{x \in z+\cu_{n} \,:\, z\in 3^n\Zd, \ z+\cu_{n+1} \not\subseteq  \cut_{2n} \right\}  , \quad |\nabla\xi|\leq C3^{-2n/(1+\delta)}
\end{multline}
where $\delta \in (0,\beta]$ will be selected below in Step~7, and
\begin{multline} \label{e.zetacutoff}
 0\leq \zeta \leq 1, \quad \zeta \equiv 1 \ \mbox{on} \  \left\{x \in  z+\cu_{n} \,:\, z\in 3^n\Zd, \ z+\cu_{n+3} \subseteq \cu_{2n}  \right\}, \\
 \zeta \equiv 0 \ \mbox{on} \  \left\{ x \in z+\cu_{n+1} \,:\, z\in 3^n\Zd, \ z+\cu_{n+1} \not\subseteq \cu_{2n} \right\}, \quad |\nabla\zeta|\leq C3^{-n}.
\end{multline}

\smallskip

To construct $v$, we first define a vector field $\mathbf{f} \in L^2 (\Rd;\Rd)$ by 
\begin{equation} \label{e.def.f}
\mathbf{f} (x):= \zeta(x) \sum_{z\in3^n\Zd} \psi(x-z) \left( \nabla u(x,z+\cu_{n+1}) - \overline P_n  \right).
\end{equation}
Since~$\mathbf{f}$ is not necessarily the gradient of an $H^1$ function, due to the errors made by introducing the partition of unity and the cutoff function, we need to take its Helmholtz-Hodge projection. We may write 
\begin{equation} \label{e.helmholtz}
\mathbf{f} = \overline{\mathbf{f}}+ \nabla w - \nabla \cdot \mathbf{S} \quad \mbox{in} \ \cu_{2n},
\end{equation}
where 
\begin{equation*}
\overline{\mathbf{f}}:= \fint_{\cu_{2n}} \mathbf{f}(x)\,dx,
\end{equation*}
$w\in H^1_{\mathrm{loc}}(\Rd)$ is defined as the unique solution of 
\begin{equation*} \label{}
\left\{ 
\begin{aligned}
& -\Delta w = -\nabla \cdot {\mathbf{f}} \ \  \mbox{in}  \ \cu_{2n}, \\
& \fint_{\cu_{2n}} w(x)\,dx =0, \\
& w \ \   \mbox{is \ $\cu_{2n}$--periodic},
\end{aligned}
\right.
\end{equation*}
and $\mathbf{S}$ is valued in the skew-symmetric matrices and has entries $S_{ij} \in H^1_{\mathrm{loc}}(\Rd)$ uniquely determined (up to an additive constant) by 
\begin{equation*} \label{}
\left\{ 
\begin{aligned}
& -\Delta S_{ij} = \partial_j f_i - \partial_i f_j \quad \mbox{in} \ \cu_{2n}, \\
& S_{ij}\ \   \mbox{is \ $\cu_{2n}$--periodic}.
\end{aligned}
\right.
\end{equation*}
Here $f_i$ denotes the $i$th entry of~$\mathbf{f}$ and $\nabla \cdot \mathbf{S}$ is the vector field with entries $\sum_{j=1}^d \partial_j S_{ij}$. Indeed,
one may check via a straightforward computation that each component of the vector field $\mathbf{f} - \nabla w + \nabla \cdot \mathbf{S}$ is harmonic and therefore constant by periodicity. This constant must be $\overline {\mathbf{f}}$ since $\nabla w$ and $\nabla \cdot \mathbf{S}$ have zero mean in~$\cu_{2n}$. This confirms~\eqref{e.helmholtz}. 

\smallskip

We finally define $v\in H^1_0(\cut_{2n})$ by setting
\begin{equation*} \label{}
v(x):= \xi(x) w(x), \quad x\in \cu_{2n}.
\end{equation*}
Note that the cutoff function~$\xi$ is supported in $\cut_{2n}$ and thus we indeed have $v\in H^1_0(\cut_{2n})$. Below we will argue that $\nabla v$ is expected to be close to $\mathbf{f}$ in $L^2(\cut_{2n})$ due to the fact that, as we will show,~$w$,~$\overline{\mathbf{f}}$ and~$\nabla \cdot \mathbf{S}$ each have a small expected~$L^2$ norm.

\smallskip

We proceed similarly to construct $\mathbf{h} \in \Lso(\cut_{2n})$. Define~$\mathbf{k}\in L^2 (\Rd;\Rd)$ by
 \begin{equation} \label{e.def.k}
 \mathbf{k} (x) := \zeta(x) \sum_{z\in3^n\Zd} \psi(x-z)\left( \g(x,z+\cu_{n+1} ) - \overline Q_n \right),
\end{equation}
Since~${\mathbf{k}}$ is not necessarily solenoidal, we remove its divergence part via the Helmholtz-Hodge projection. As above, we may write 
\begin{equation*} \label{}
 \mathbf{k} =  \overline{\mathbf{k}} + \nabla h - \nabla \cdot \mathbf{T}
\end{equation*}
where
\begin{equation*} \label{}
 \overline{\mathbf{k}}:= \fint_{\cu_{2n}} \mathbf{k}(x)\,dx,
\end{equation*}
$h\in H^1_{\mathrm{per}}(\cu_{2n})$ is the unique solution of
\begin{equation*} \label{}
\left\{ 
\begin{aligned}
& -\Delta h = -\nabla \cdot {\mathbf{k}} \ \  \mbox{in}  \ \cu_{2n}, \\
& \fint_{\cu_{2n}} h(x)\,dx =0, \\
& h \ \   \mbox{is \ $\cu_{2n}$--periodic},
\end{aligned}
\right.
\end{equation*}
and $\mathbf{T}$ is a skew-symmetric matrix-valued field with entries in $H^1_{\mathrm{per}}(\cu_{2n})$. We finally define $\mathbf{h} \in \Lso(\cut_{2n})$ by setting
\begin{equation*} \label{}
\mathbf{h}(x):= \xi(x)\left( \nabla \cdot \mathbf{T}\right)(x) - \nabla \tilde h,
\end{equation*}
where $\tilde h \in H^1(\cut_{2n})$ is defined (also uniquely up to a constant) as the solution of
\begin{equation*} \label{}
\left\{ 
\begin{aligned}
& -\Delta \tilde h = -\nabla \xi(x) \cdot (\nabla \cdot \mathbf{T}) &  \mbox{in} & \  \cut_{2n} , \\
& \partial_\nu \tilde h  = 0 &  \mbox{on} & \  \partial\cut_{2n}.
\end{aligned}
\right.
\end{equation*}
It is clear that $\mathbf{h} \in \Lso(\cut_{2n})$. Below we will argue that $\left| \mathbf{k} - \mathbf{h}\right|$ has small expected~$L^2(\cut_{2n})$ norm.

\smallskip

This completes the construction of $(v,\mathbf{h}) \in H^1_0(\cut_{2n}) \times \Lso(\cut_{2n})$. The rest of the argument is focused on the proof of~\eqref{e.exhibition}.

\smallskip

\emph{Step 2.} We show that, for every $z\in3^n\Zd \cap \cu_{2n}$,
\begin{multline}\label{e.gradlocalize}
\E \bigg[ \fint_{z+\cu_n} \Big(  \left| \mathbf{f}(x) - \zeta(x) \left( \nabla u(x,z+\cu_{n+1}) - \overline P_n \right) \right|^2  \\ + \left| \mathbf{k}(x) - \zeta(x) \left( \g(x,z+\cu_{n+1}) - \overline Q_n \right) \right|^2 \Big)\,dx  \bigg] \leq  C \tau_n.
\end{multline}
By Lemma~\ref{l.unifconv} we have, for every $z\in3^n\Zd$,
\begin{multline*}
\sum_{y\in  \{ -3^n,0,3^n\}^d} \fint_{z+y+\cu_n} \left| (\nabla u,\g)(x,z+\cu_{n+1}) - (\nabla u,\g)(x,z+y+ \cu_{n}) \right|^2 \, dx \\
\begin{aligned}
& \leq C \sum_{y\in  \{ -3^n,0,3^n\}^d} \left( \fint_{z+y+\cu_n} F\left( (\nabla u,\g)(x,z+\cu_{n+1}),x\right)\, dx - \mu(y+z+\cu_n) \right) \\
& = C \left( \mu(z+\cu_{n+1}) - \sum_{y\in  \{ -3^n,0,3^n\}^d} \mu(z+y+\cu_n) \right).
\end{aligned}
\end{multline*}
Taking expectations and using the triangle inequality, we find that, for every $z\in 3^n\Zd$ and $y\in \{ -3^n,0,3^n\}^d$,
\begin{equation}\label{e.gridtrap}
\E \left[ \fint_{z+\cu_n} \left| (\nabla u,\g)(x,z+y+\cu_{n+1}) - (\nabla u,\g)(x,z+\cu_{n+1}) \right|^2 \, dx \right] 
\leq C \tau_n.
\end{equation}
The bound on the first term on the left of~\eqref{e.gradlocalize} is obtained from the previous inequality and the following identity, which holds for every $z\in 3^n\Zd$ and $x\in z+\cu_{n}$ by the definition of $\mathbf{f}$:
\begin{multline*}
\mathbf{f}(x) - \zeta(x) \left( \nabla u(x,z+\cu_{n+1}) - \overline P_n \right) \\ 
= \zeta(x) \sum_{y\in \{ -3^n,0,3^n\}^d} \psi(x-y) \left( \nabla u(x,z+y+\cu_{n+1}) - \nabla u(x,z+\cu_{n+1}) \right).
\end{multline*}
The bound on the second term on the left of~\eqref{e.gradlocalize} follows from~\eqref{e.gridtrap} and a similar identity. 

\smallskip

\emph{Step 3.} We claim that 
\begin{equation}\label{e.whipfbar}
 \E \left[  \left|\overline{\mathbf{f}}\right|^2  \right] + \E \left[  \left|\overline{\mathbf{k}}\right|^2  \right] \leq C\tau_n + C'\kappa_n.
\end{equation}
In view of~\eqref{e.zetacutoff}, it is convenient to denote
\begin{equation} \label{e.Zn}
\mathcal Z_n:= \left\{ z\in 3^n\Zd\,:\, z+\cu_{n+3} \subseteq \cu_{2n} \right\}. 
\end{equation}
Observe that $( 3^n\Zd \cap \cu_{2n} ) \setminus \mathcal Z_n$ has $C 3^{n(d-1)}$ elements. 
By Lemma~\ref{l.flat} (or more precisely~\eqref{e.flatuntrimmed}),~\eqref{e.gradlocalize} and~\eqref{e.boundminimizers}, we have
\begin{align*}
\E \left[ \left|\overline{\mathbf{f}} \right|^2 \right] & \leq \E \left[ 3^{-nd} \sum_{z\in 3^n\Zd \cap \cu_{2n}} \left| \fint_{z+\cu_n} \mathbf{f}(x)\,dx  \right|^2 \right]  \\
& \leq 2 \E \left[ 3^{-nd} \sum_{z\in 3^n\Zd \cap \cu_{2n}} \left| \fint_{z+\cu_n} \zeta(x)\left( \nabla u(x,z+\cu_{n+1})\,dx-\overline P_n \right)  \right|^2 \right]  + C\tau_n  \\
& \leq 2\E \left[ 3^{-nd} \sum_{z\in \mathcal{Z}_n} \left| \fint_{z+\cu_n}  \nabla u(x,z+\cu_{n+1})\,dx-\overline P_n  \right|^2 \right] + CK_0^23^{-n} + C\tau_n \\
& \leq C \tau_n + C'\kappa_n.
\end{align*}
An analogous calculation gives the estimate for $ \E \left[  \left|\overline{\mathbf{k}}\right|^2  \right]$.

\smallskip

\emph{Step 4.} We show that 
\begin{equation} \label{e.boundonw2}
\E \left[ 3^{-4n} \fint_{\cu_{2n}} \left| w(x) \right|^2\,dx \right] \leq C'K_0^23^{-n\beta / (1+\beta)}.
\end{equation}
Let $\phi \in H^2_{\mathrm{loc}}(\Rd)$ denote the unique solution of
\begin{equation*} \label{}
\left\{ 
\begin{aligned}
& -\Delta \phi = w \ \  \mbox{in}  \ \Rd, \\
& \fint_{\cu_{2n}} \phi(x)\,dx =0, \\
& \phi \ \  \mbox{is $\cu_{2n}$--periodic}.
\end{aligned}
\right.
\end{equation*}
By integration by parts, we have the identities
\begin{multline}\label{e.softH2}
\int_{\cu_{2n}} \left| \nabla \nabla \phi(x) \right|^2\,dx = \int_{\cu_{2n}} \left| w(x) \right|^2\,dx = \int_{\cu_{2n}} \nabla \phi(x) \cdot \mathbf f(x)  \, dx \\ 
= \int_{\cu_{2n}} \nabla \phi(x) \cdot \left( \mathbf f(x) - \E \left[ \int_{\cu_n} \mathbf{f}(x)\,dx \right]  \right)  \, dx .
\end{multline}
We need a second mesoscale, given by an integer $k \in (n,2n)$ to be selected below. In what follows, we denote $ (\nabla \phi)_z:= \fint_{z+\cu_k} \nabla \phi(x)\,dx$ and $\sum_z = \sum_{z \in3^k\Zd\cap \cu_{2n}}$ as well as $\tilde{ \mathbf{f}}:=\mathbf{f}- \E \left[ \int_{\cu_n} \mathbf{f}(x)\,dx \right] $. Now we estimate, by the Poincar\'e inequality:
\begin{align*}
\int_{\cu_{2n}} \left| w(x) \right|^2\,dx & = \int_{\cu_{2n}}  \nabla \phi(x) \cdot \tilde{\mathbf {f}}(x) \, dx \\
& = \sum_z \left( \int_{z+\cu_k} \left( \nabla \phi(x)    - (\nabla \phi)_z  \right) \cdot  \tilde{\mathbf {f}}(x) \,dx +  (\nabla \phi)_z \cdot \int_{z+\cu_k} \tilde{\mathbf {f}}(x)\,dx \right) \\
& \begin{multlined}[.77\textwidth]
\leq C \sum_z 3^k \left( \int_{z+\cu_k} \left| \nabla\nabla \phi(x) \right|^2\,dx \right)^{\frac12} \left( \int_{z+\cu_k}  \left| \mathbf{f}(x) \right|^2\,dx \right)^{\frac12} \\
 +  C\sum_z(\nabla \phi)_z \cdot \int_{z+\cu_k}  \tilde{\mathbf {f}}(x)\,dx.\end{multlined}
\end{align*}
To estimate the first sum on the right side of the previous inequality, we use the discrete H\"older inequality,~\eqref{e.softH2} and~\eqref{e.boundminimizers}, and then Young's inequality:
\begin{align*}
\lefteqn{ \sum_z 3^k \left( \int_{z+\cu_n} \left| \nabla\nabla \phi(x) \right|^2\,dx \right)^{\frac12} \left( \int_{z+\cu_n}  \left|  \tilde{\mathbf {f}}(x) \right|^2\,dx \right)^{\frac12} } \qquad & \\
& \leq 3^k \left(\sum_z \int_{z+\cu_n} \left| \nabla\nabla \phi(x) \right|^2\,dx \right)^{\frac12} \left(\sum_z \int_{z+\cu_n}  \left|  \tilde{\mathbf {f}}(x) \right|^2\,dx\right)^{\frac12} \\
& = 3^k \left( \int_{\cu_{2n}} \left| \nabla\nabla \phi(x) \right|^2 \,dx \right)^{\frac12} \left( \int_{\cu_{2n}} \left|  \tilde{\mathbf {f}}(x)\right|^2 \,dx \right)^{\frac12} \\
& \leq 3^k \left( \int_{\cu_{2n}} \left| w(x) \right|^2\,dx \right)^{\frac12} \left( CK_0^2\left| \cu_{2n} \right|\right)^{\frac12} \\
& \leq \frac14 \int_{\cu_{2n}} \left| w(x) \right|^2\,dx  + CK_0^2\left| \cu_{2n} \right| 3^{2k}.
\end{align*}
We next estimate the expectation of the second sum. Using H\"older's inequality in two different forms, we get
\begin{align}\label{e.holderapps}
\lefteqn{\E \left[ \sum_z (\nabla \phi)_z \cdot \int_{z+\cu_k} \tilde{\mathbf{f}}(x)\,dx \right] } \qquad & \\
& = \sum_z \E \left[ (\nabla \phi)_z \cdot \int_{z+\cu_k} \tilde{\mathbf{f}}(x)\,dx \right] \notag \\
& \leq \sum_z \E \left[ \left| \left( \nabla \phi \right)_z \right|^2  \right]^{\frac12} \E \left[ \left( \int_{z+\cu_k} \tilde{\mathbf{f}}(x)\,dx \right)^2 \right]^{\frac12}\notag  \\
& \leq \left( \sum_z \E \left[ \left| \left( \nabla \phi \right)_z \right|^2  \right]  \right)^{\frac12} \left(\sum_z \E \left[ \left( \int_{z+\cu_k} \tilde{\mathbf{f}}(x)\,dx \right)^2 \right]  \right)^{\frac12}.  \notag
\end{align}
For the first factor on the right side of the previous inequality, we have, by the Poincar\'e inequality and~\eqref{e.softH2}, 
\begin{align}\label{e.ingred1}
 \sum_z \E \left[ \left| \left( \nabla \phi \right)_z \right|^2  \right]  = \E \left[ \sum_z \left| \left( \nabla \phi \right)_z \right|^2 \right] & \leq  C3^{-kd}\,\E \left[ \int_{\cu_{2n}} \left| \nabla \phi(x) \right|^2\,dx \right]  \\
& \leq C3^{-kd+4n}\,\E \left[ \int_{\cu_{2n}} \left| \nabla \nabla \phi(x) \right|^2\,dx \right] \notag\\
& =  C3^{-kd+4n}\,\E \left[ \int_{\cu_{2n}} \left| w(x) \right|^2\,dx \right].\notag
\end{align}
In preparation to estimate the second factor, we use the mixing condition in the form of~\eqref{e.howweuseP3} to get
\begin{align*}
\E \left[ \left( \int_{\cu_k} \tilde{\mathbf{f}}(x)\,dx \right)^2 \right] & = \E\left[ \sum_{y,y'\in 3^n\Zd\cap \cu_k}  \int_{y+\cu_n} \tilde{\mathbf{f}}(x)\,dx \int_{y'+\cu_n} \tilde{\mathbf{f}}(x)\,dx\right] \\
& \leq  C' \left( CK_0 \left| \cu_n \right| \right)^2  \sum_{y,y'\in 3^n\Zd\cap \cu_k} \left( 1 + |y-y'|\right)^{-\beta} \\
& = C' K_0^2 3^{2dk - \beta (k-n)}
\end{align*}
By stationarity, the same estimate holds with $z+\cu_k$ in place of $\cu_k$ provided that the cube $z+\cu_k$ does not touch $\partial \cu_{2n}$. For the cubes which do touch the  boundary of the macroscopic cube (and thus intersect the support of~$\zeta$), we use the following cruder, deterministic bound given by~\eqref{e.boundminimizers}:
\begin{equation*}
\left( \int_{\cu_k}  \tilde{\mathbf{f}}(x) \,dx \right)^2 \leq  C3^{dk} \int_{\cu_k}  \left| \tilde{\mathbf{f}}(x) \right|^2 \,dx \leq C K_0^2 3^{2dk} 
\end{equation*}
Combining these, using that there are at most $C3^{(2n-k)(d-1)}$ cubes of the form $z+\cu_k$ which touch the boundary of $\cu_n$, we get 
\begin{equation}\label{e.ingred2}
\sum_z \E \left[ \left( \int_{z+\cu_k} \tilde{\mathbf{f}}(x)\,dx \right)^2 \right] \leq C'K_0^2 3^{2dn + dk - \beta (k-n)} + CK_0^2 3^{2dn + dk - (2n-k)}.
\end{equation}
We may now estimate the right side of~\eqref{e.holderapps} using applying~\eqref{e.ingred1},~\eqref{e.ingred2} and Young's inequality. The result is:
\begin{multline}
\E \left[ \sum_z (\nabla \phi)_z \cdot \int_{z+\cu_k} \tilde{\mathbf{f}}(x)\,dx \right] \\ \leq \frac14 \E \left[ \int_{\cu_{2n}} \left| w(x) \right|^2\,dx \right] + K_0^2 \left| \cu_{2n}\right| 3^{4n}\left( C' 3^{-\beta (k-n)} + C3^{-(2n-k)} \right) 
\end{multline}
Combining the above inequalities now yields
\begin{equation*}
\E \left[ \int_{\cu_{2n}} \left| w(x) \right|^2\,dx \right]  \leq K_0^2 3^{4n} \left| \cu_{2n} \right| \left( C'3^{-\beta (k-n)} + C3^{-(2n-k)} \right).
\end{equation*}
Taking finally $k$ to be the nearest integer to $(2n+n\beta)/(1+\beta)$, we obtain~\eqref{e.boundonw2}.

\smallskip

\emph{Step 5.}
We estimate the expected contribution of $|\nabla h|^2$ by a computation which bears a resemblance to the one in the previous step. The claim is that
\begin{equation} \label{e.bangbang}
\E \left[ \fint_{\cu_{2n}} \left| \nabla h(x) \right|^2 \, dx \right] \leq  C(\tau_n + \kappa_n).
\end{equation}
Here we use the abbreviations $\mathbf{g}_z:= \mathbf{g}(\cdot,z+\cu_{n+1})$, $\psi_z:= \psi(\cdot-z)$, $\sum_z:= \sum_{z\in 3^n\Zd \cap \cu_{2n} }$ and $\int_z:= \int_{z+\cu_{n+1}}$. Observe that, in the sense of distributions, we have
\begin{equation}
\label{e.divkL2}
\nabla \cdot \mathbf{k} = \sum_{z}\zeta  \nabla \psi_z \cdot\left( \mathbf{g}_z - \overline Q_n \right) 
+ \sum_{z} \psi_z \nabla \zeta\cdot \left( \mathbf{g}_z - \overline Q_n\right) \quad \mbox{in} \ \cu_{2n}.
\end{equation}
In particular, the right-hand side belongs to $L^2(\cu_{2n})$ and thus $h \in H^2_{\mathrm{per}}(\cu_{2n})$. 
Using the identity
\begin{equation} \label{e.sillyidentity}
 \sum_{z} \zeta(x)\nabla \psi_z(x)  = 0,\quad x\in\Rd,
\end{equation}
which follows from~\eqref{e.partunity} and~\eqref{e.zetacutoff}, we may re-express the previous identity, for $x\in \cu_{2n}$, as
\begin{multline}
\label{e.divkL22}
(\nabla \cdot \mathbf{k}) (x)= \sum_{z} \zeta(x) \nabla \psi_z(x) \cdot\left( \mathbf{g}_z(x) - \overline Q_n - \mathbf{k}(x) \right) \\
+ \sum_{z} \psi_z(x) \nabla \zeta(x) \cdot \left( \mathbf{g}_z(x) - \overline Q_n\right).
\end{multline}
For each $i\in\{ 1,\ldots,d\}$, let $\phi_i\in H^3_{\mathrm{loc}}(\Rd)$ denote the unique solution of
\begin{equation*} \label{}
\left\{ 
\begin{aligned}
& -\Delta \phi_i = \partial_i h \ \  \mbox{in}  \ \cu_{2n}, \\
& \fint_{\cu_{2n}} \phi_i(x)\,dx =0, \\
&  \phi_i \ \  \mbox{is $\cu_{2n}$--periodic}.
\end{aligned}
\right.
\end{equation*}
Then integrating by parts, we find
\begin{equation}\label{e.softH22}
\int_{\cu_{2n}} \left| \nabla \nabla \phi_i(x) \right|^2\,dx  = \int_{\cu_{2n}} \left| \partial_i h (x) \right|^2\,dx =  \int_{\cu_{2n}} \partial_i \phi_i(x) \left(\nabla \cdot \mathbf{k}\right)(x) \, dx.
\end{equation}
We continue the computation by substituting~\eqref{e.divkL22} and using the notation $(\partial_i \phi_i)_z:=\fint_{z+\cu_{n+1}} \partial_i \phi_i(x)\,dx$:
\begin{align*}
\lefteqn{ \int_{\cu_{2n}} \partial_i \phi_i(x) \left(\nabla \cdot \mathbf{k}\right)(x) \, dx } \qquad &  \\
& = \sum_{z} \int_{z} \left( \partial_i \phi_i(x) - \left( \partial_i \phi_i\right)_z \right) \zeta(x) \nabla \psi_z(x) \cdot \left( \mathbf{g}_z(x) - \overline Q_n - \mathbf{k}(x) \right)   \,dx \\
& \qquad + \sum_{z}  \left( \partial_i \phi_i\right)_z \int_{z} \zeta(x) \nabla \psi_z(x) \cdot \left( \mathbf{g}_z(x) - \overline Q_n - \mathbf{k}(x) \right)   \,dx \\
& \qquad + \sum_{z}  \int_{z} \partial_i\phi_i(x)\psi_z(x) \nabla \zeta(x) \cdot \left( \mathbf{g}_z(x) - \overline Q_n \right) \,dx. 
\end{align*}
We put the second sum on the right side of the previous expression into a more convenient form using~\eqref{e.def.k}, \eqref{e.partunity}, integration by parts and \eqref{e.sillyidentity} twice:
\begin{align*}
\lefteqn{ \sum_{z}  \left( \partial_i \phi_i\right)_z \int_{z} \zeta(x) \nabla \psi_z(x) \cdot \left( \mathbf{g}_z(x) - \overline Q_n - \mathbf{k}(x) \right)   \,dx } \qquad & \\
& = \sum_{y,z}  \left( \partial_i \phi_i\right)_z \int_{z} \zeta(x)\psi_y(x) \nabla \psi_z(x) \cdot \left( \mathbf{g}_z(x) - \overline Q_n - \zeta(x) \left( \mathbf{g}_y(x) -\overline Q_n\right) \right)   \,dx \\
& = \sum_{y,z} \left(  \left( \partial_i \phi_i\right)_z - \left( \partial_i \phi_i\right)_y \right) \int_{z} - (\zeta(x))^2 \psi_y(x) \nabla \psi_z(x) \cdot  \left( \mathbf{g}_y(x) -\overline Q_n\right)   \,dx  \\
& \qquad + \sum_z \left( \partial_i \phi_i\right)_z \int_{z} -\psi_z(x) \nabla \zeta(x) \cdot \left( \mathbf{g}_z(x) - \overline Q_n\right)\,dx.
\end{align*}
Combining the previous two identities yields
\begin{align} \label{e.splittingk3}
\lefteqn{ \int_{\cu_{2n}} \partial_i \phi_i(x) \left(\nabla \cdot \mathbf{k}\right)(x) \, dx } \qquad &  \\
&  = \sum_{z} \int_{z} \left( \partial_i \phi_i(x) - \left( \partial_i \phi_i\right)_z \right) \zeta(x) \nabla \psi_z(x) \cdot \left( \mathbf{g}_z(x) - \overline Q_n - \mathbf{k}(x) \right)   \,dx \notag\\
& \qquad + \sum_{y,z} \left(  \left( \partial_i \phi_i\right)_y - \left( \partial_i \phi_i\right)_z \right) \int_{z} (\zeta(x))^2 \psi_y(x) \nabla \psi_z(x) \cdot  \left( \mathbf{g}_y(x) -\overline Q_n\right)   \,dx  \notag\\
& \qquad + \sum_z \int_{z} \left( \partial_i\phi_i(x)- \left( \partial_i \phi_i\right)_z \right)\psi_z(x) \nabla \zeta(x) \cdot \left( \mathbf{g}_z(x) - \overline Q_n\right)\,dx.\notag
\end{align} 
We estimate the three sums on the right side of~\eqref{e.splittingk3} in a similar fashion, each in turn. For the first sum, we use~\eqref{e.psiderv}, the H\"older, discrete H\"older and Poincar\'e inequalities,~\eqref{e.softH22} and Young's inequality to deduce that
\begin{align*}
\lefteqn{\sum_{z} \int_{z} \left( \partial_i \phi_i(x)    - (\partial_i \phi_i)_z  \right)\zeta(x) \left( \nabla \psi_z(x) \cdot \left( \mathbf{g}_z(x) - \overline Q_n- \mathbf{k}(x) \right) \right)  \,dx } \qquad & \\
& \leq C \left(  \sum_z  \int_z \left| \nabla \nabla \phi_i(x) \right|^2\,dx  \right)^{\frac12}  \left(  \sum_z  \int_z \left| \mathbf{g}_z(x) - \overline Q_n- \mathbf{k}(x) \right|^2\,dx  \right)^{\frac12} \\
& \leq \frac14 \int_{\cu_{2n}} \left| \partial_i h(x) \right|^2\,dx + C\sum_z \int_z \left| \mathbf{g}_z(x) - \overline Q_n- \mathbf{k}(x) \right|^2\,dx.
\end{align*}
Taking expectations and using~\eqref{e.gradlocalize}, we get
\begin{multline} \label{e.firstsumnailedk}
\E \left[ \sum_{z} \int_{z} \left( \partial_i \phi_i(x)    - (\partial_i \phi_i)_z  \right) \left( \nabla \psi_z(x) \cdot \left( \mathbf{g}_z(x) - \overline Q_n- \mathbf{k}(x) \right) \right)  \,dx \right] \\
 \leq \frac14 \E \left[ \int_{2n} \left| \partial_ih(x) \right|^2\,dx \right] + C\left| \cu_{2n} \right| \tau_n.
\end{multline}
For the second sum, we notice that each entry vanishes unless $y \in z+\cu_{n+2}$, there are at most $C$ such entries $y$ in the sum for any given entry $z$, and for such $y$ and $z$, the Poincar\'e inequality gives
\begin{equation*} \label{}
 \left|  \left( \partial_i \phi_i\right)_y - \left( \partial_i \phi_i\right)_z \right|^2 \leq C3^{2n}  |\cu_n|^{-1}  \int_{z+\cu_{n+3}} \left|\nabla\nabla \phi_i(x) \right|^2\,dx.
\end{equation*}
Using this,~\eqref{e.psiderv},~\eqref{e.softH22} and~\eqref{e.boundminimizers}, the H\"older and Young inequalities and the fact that $3^n\Zd\cap \cu_{2n}$ has $C3^{nd}$ elements, we get 
\begin{align*}
\lefteqn{ \sum_{y,z} \left(  \left( \partial_i \phi_i\right)_y - \left( \partial_i \phi_i\right)_z \right) \int_{z} (\zeta(x))^2 \psi_y(x) \nabla \psi_z(x) \cdot  \left( \mathbf{g}_y(x) -\overline Q_n\right)   \,dx  } \qquad & \\
& \leq C \left(  3^{2n}|\cu_{n}|^{-1} \sum_{z} \int_{z+\cu_{n+3}} \left| \nabla\nabla \phi_i(x) \right|^2\,dx \right)^{\frac12}\left(  \sum_{z} 3^{-2n} CK_0^2|\cu_n| \right)^{\frac12} \\
& \leq \frac14 \int_{\cu_{2n}} \left| \partial_i h(x) \right|^2\,dx + C K_0^23^{nd}.
\end{align*}
For the third sum on the right side of~\eqref{e.splittingk3}, we proceed in almost the same way as for the first two, except that rather than use~\eqref{e.psiderv} we use the estimate for~$\nabla \zeta$ in~\eqref{e.zetacutoff} and the fact that $\nabla \zeta$ vanishes except if $z\not\in \mathcal Z_n$ (recall that $\mathcal Z_n$ is defined in~\eqref{e.Zn}) and there are at most $C3^{n(d-1)}$ such elements in the sum. We obtain:
\begin{align*}
\lefteqn{ \sum_z \int_{z} \left( \partial_i\phi_i(x)- \left( \partial_i \phi_i\right)_z \right)\psi_z(x) \nabla \zeta(x) \cdot \left( \mathbf{g}_z(x) - \overline Q_n\right)\,dx } \qquad & \\
& \leq C \left(  3^{2n} \sum_{z} \int_{z} \left| \nabla\nabla \phi_i(x) \right|^2\,dx \right)^{\frac12}\left(  \sum_{z\notin\mathcal Z_n} 3^{-2n} CK_0^2|\cu_n| \right)^{\frac12} \\
& \leq \frac14 \int_{\cu_{2n}} \left| \partial_i h(x) \right|^2\,dx + C K_0^2 3^{n(2d - 1)}.
\end{align*}
Combining the previous two inequalities with~\eqref{e.softH22}, ~\eqref{e.splittingk3}  and~\eqref{e.firstsumnailedk} yields
\begin{equation*}
\E \left[  \int_{\cu_{2n}} \left| \partial_i h (x) \right|^2\,dx \right] \leq CK_0^2 3^{n(2d-1)} + C\left| \cu_{2n} \right| \tau_n.
\end{equation*}
Dividing by~$\left| \cu_{2n} \right|$ gives~\eqref{e.bangbang}. 

\smallskip

\emph{Step 6.} We estimate the expected size of $\left| \nabla \cdot \mathbf{S} \right|^2$, using a computation which is closely analogous to the one in Step~5. Indeed, the main difference from Step~5 is essentially notational. The estimate we will show is  
\begin{equation}\label{e.whipsolenoidal}
\E \left[ \fint_{\cu_{2n}} \left|\nabla \cdot \mathbf{S}(x) \right|^2\,dx \right] \leq C(\tau_n+ \kappa_n).
\end{equation}
As in the previous step, we use the abbreviations $u_z:= u(\cdot,z+\cu_{n+1})$, $\psi_z:= \psi(\cdot-z)$, $\sum_z:= \sum_{z\in 3^n\Zd \cap \cu_{2n} }$ and $\int_z:= \int_{z+\cu_{n+1}}$
Observe that, in the sense of distributions, for every $i,j\in\{1,\ldots,d\}$, we have
\begin{multline*}
\partial_j f_i - \partial_i f_j 
= \sum_{z} \zeta \left(\partial_j\psi_z ( \partial_iu_z- \overline P_{n,i})  - \partial_i\psi_z (\partial_ju_z -\overline P_{n,j})\right) \\
+ \sum_z \psi_z \left(\partial_j\zeta (\partial_iu_z -\overline P_{n,i}))  - \partial_i\zeta (\partial_ju_z -\overline P_{n,j})\right) \quad \mbox{in} \ \cu_{2n}.
\end{multline*}
The right side belongs to $L^2(\cu_{2n})$, thus  $\partial_j f_i - \partial_i f_j\in L^2(\cu_{2n})$ and $S_{ij}\in H^2_{\mathrm{loc}} (\Rd)$. Using~\eqref{e.sillyidentity}, we may also express the previous identity slightly differently as
\begin{multline}
\label{e.curlfL2}
\left( \partial_j f_i - \partial_i f_j \right)(x)
 = \sum_{z} \zeta(x) \left[\nabla \psi_z(x), \nabla u_z(x) - \overline P_n - \mathbf{f}(x) \right]_{ij}  \\
+ \sum_z \psi_z \left[ \nabla \zeta, \nabla u_z  -\overline P_{n} \right]_{ij} \quad \mbox{in} \ \cu_{2n},
\end{multline}
where we henceforth use the notation 
\begin{equation*}
\left[ \mathbf{v},\mathbf{w} \right]_{ij} := v_j w_i - v_iw_j
\end{equation*}
for indices $i,j\in\{ 1,\ldots,d\}$ and vectors $\mathbf{v},\mathbf{w}\in\Rd$ with entries $(v_i)$ and $(w_i)$, respectively. 
Next we define, for each~$i \in \{ 1,\ldots,d\}$,
\begin{equation*}
\sigma_i := -\left( \nabla \cdot \mathbf{S} \right)_i = -\sum_{j=1}^d \partial_jS_{ij}.
\end{equation*}
It is evident that $\sigma_i \in H^1_{\mathrm{per}}(\cu_{2n})$ and $\sigma_i$ is a solution of the equation
\begin{equation*}
-\Delta \sigma_i = -\sum_{j=1}^d \partial_j \left( \partial_j f_i - \partial_i f_j \right) \quad \mbox{in} \ \cu_{2n}.
\end{equation*}
Since $\sigma_i$ has zero mean in $\cu_{2n}$, there exists~$\rho_i \in H^3_{\mathrm{loc}} (\Rd)$, which is unique up to an additive constant, satisfying 
\begin{equation*} \label{}
\left\{ 
\begin{aligned}
& -\Delta \rho_i = \sigma_i \quad \mbox{in} \ \Rd, \\
& \rho_i \ \   \mbox{is \ $\cu_{2n}$--periodic}.
\end{aligned}
\right.
\end{equation*}
We have the identities 
\begin{equation} \label{e.H2equal}
\int_{\cu_{2n}} \left| \nabla \nabla \rho_i(x) \right|^2\,dx =  \int_{\cu_{2n}} \left| \sigma_i(x)\right|^2\,dx =  \int_{\cu_{2n}} \nabla \rho_i(x) \cdot \nabla \sigma_i(x)\,dx.
\end{equation}
Integrating by parts and using the equation for $\sigma_i$, we obtain
\begin{equation*}
\int_{\cu_{2n}} \left| \sigma_i(x)\right|^2\,dx = \int_{\cu_{2n}} \nabla \rho_i(x) \cdot \nabla \sigma_i(x)\,dx = \sum_{j = 1}^d \int_{\cu_{2n}} \partial_j \rho_i(x)\left( \partial_jf_i - \partial_i f_j \right)\,dx.
\end{equation*}
To further shorten the notation, in each of the following expressions we keep the sum over~$j$ implicit (note that $i$ is not summed over) and set $\left( \partial_j \rho_i \right)_z:= \fint_{z+\cu_{n+1}} \partial_j\rho_i(x) \,dx$. Continuing then the computation by substituting~\eqref{e.curlfL2}, we obtain
\begin{align*}
\lefteqn{ \int_{\cu_{2n}} \partial_j \rho_i(x)\left( \partial_jf_i - \partial_i f_j \right)\,dx } \qquad & \\
& =  \sum_z \int_{z} \left( \partial_j \rho_i(x)- \left( \partial_j \rho_i \right)_z \right) \zeta(x) \left[\nabla \psi_z(x), \nabla u_z(x) - \overline P_n - \mathbf{f}(x) \right]_{ij} \,dx \\
& \quad +  \sum_z \left( \partial_j \rho_i \right)_z  \int_{z} \zeta(x)\left[\nabla \psi_z(x), \nabla u_z(x) - \overline P_n - \mathbf{f}(x) \right]_{ij}\,dx \\
& \quad +  \sum_z\int_{z} \partial_j \rho_i(x)  \psi_z(x) \left[\nabla \zeta(x), \nabla u_z(x) - \overline P_n  \right]_{ij} \,dx.
\end{align*}
We put the second sum on the right side into a more convenient form via \eqref{e.def.f}, \eqref{e.partunity}, integration by parts and~\eqref{e.sillyidentity}: 
\begin{align*}
\lefteqn{ \sum_z \left( \partial_j \rho_i \right)_z  \int_{z} \zeta(x)\left[\nabla \psi_z(x), \nabla u_z(x) - \overline P_n - \mathbf{f}(x) \right]_{ij}\,dx} \quad & \\
& = \sum_{y,z}  \left( \partial_j \rho_i \right)_z \int_z \zeta(x) \psi_y(x) \left[\nabla \psi_z(x), \nabla u_z(x) - \overline P_n - \zeta(x)( \nabla u_y(x) - \overline P_{n} ) \right]_{ij}\,dx \\
& = \sum_{y,z}  \left( \left( \partial_j \rho_i \right)_z - \left( \partial_j \rho_i \right)_y \right) \int_z -(\zeta(x))^2 \psi_y(x) \left[\nabla \psi_z(x), \nabla u_y(x) - \overline P_n  \right]_{ij}\,dx \\
& \qquad + \sum_{z}  \left( \partial_j \rho_i \right)_z \int_z - \psi_z(x) \left[\nabla \zeta(x), \nabla u_z(x) - \overline P_n  \right]_{ij}\,dx.
\end{align*}
Combining this with the previous identity, we get 
\begin{align*}
\lefteqn{\int_{\cu_{2n}} \left| \sigma_i(x)\right|^2\,dx } \quad & \\
& =  \sum_z \int_{z} \left( \partial_j \rho_i(x)- \left( \partial_j \rho_i \right)_z \right) \zeta(x) \left[\nabla \psi_z(x), \nabla u_z(x) - \overline P_n - \mathbf{f}(x) \right]_{ij} \,dx \\
& \quad + \sum_{y,z}  \left( \left( \partial_j \rho_i \right)_y - \left( \partial_j \rho_i \right)_z \right) \int_z (\zeta(x))^2 \psi_y(x) \left[\nabla \psi_z(x), \nabla u_y(x) - \overline P_n  \right]_{ij}\,dx \\
&  \quad + \sum_{z}  \int_z \left( \partial_j \rho_i (x)- \left( \partial_j \rho_i \right)_z \right)  \psi_z(x) \left[\nabla \zeta(x), \nabla u_z(x) - \overline P_n  \right]_{ij}\,dx.
\end{align*}
We may now compare this identity with~\eqref{e.splittingk3} and observe that the three sums on the right side are similar to those on the right side of~\eqref{e.splittingk3}. In fact, following the arguments in Step~5 (with obvious substitutions, changing for instance $\phi_i$ to $\rho_i$, $\mathbf{g}_z - \overline Q_n$ to $\nabla u_z - \overline P_n$ and $\mathbf{k}$ to $\mathbf{f}$), we may bound these three sums in the same way. This completes the argument for~\eqref{e.whipsolenoidal}. 

\smallskip

\emph{Step 7.} We show that the effect of the cutoff $\xi$ in the definitions of $v$ and $\mathbf{h}$ is expected to be small: precisely,
\begin{multline}\label{e.closenvnw}
\E \left[ \fint_{\cu_{2n}} \left| \nabla v(x) - \nabla w(x) \right|^2\,dx \right] +  \E \left[ \fint_{\cu_{2n}} \left| \mathbf{h}(x) - \nabla\cdot \mathbf{T}(x) \right|^2\,dx \right] \\
\leq  C\tau_n+C'\kappa_n. 
\end{multline}
We use the identity
\begin{equation*}
\nabla v(x) - \nabla w(x) = w(x) \nabla \xi(x) + (\xi(x)-1) \nabla w(x)
\end{equation*}
and~\eqref{e.xicutoff} to obtain
\begin{multline} \label{e.nvnwsplit}
\fint_{\cu_{2n}} \left| \nabla v(x) - \nabla w(x) \right|^2\,dx \\
 \leq C3^{-4n/(1+\delta)} \fint_{\cu_{2n}} \left| w(x) \right|^2\,dx + C\fint_{\cu_{2n}} \left| \xi(x) - 1 \right|^2 \left| \nabla w(x) \right|^2\,dx.
\end{multline}
The expectation of the first integral on the right side is controlled by~\eqref{e.boundonw2}:
\begin{equation*}
\E \left[3^{-4n/(1+\delta)} \fint_{\cu_{2n}} \left| w(x) \right|^2\,dx \right] \leq CK_0^2 3^{-n\beta/(1+\beta) + 4n\delta/(1+\delta)} \leq C\kappa_n, 
\end{equation*}
provided we select
$$\delta := \frac1{16}\beta.$$ For the expectation of the second integral on the right side of~\eqref{e.nvnwsplit}, we recall from~\eqref{e.xicutoff} that $\xi \equiv 1$ except in $$D:= \left\{ x\in \cu_{2n} \,:\, \dist(x,\partial \cu_{2n}) > C3^{2n/(1+\delta)} \right\}.$$ Therefore, using that $D$ intersects at most $C3^{n(d-2\delta/(1+\delta))}$ subcubes of the form $z+\cu_{n+1}$, with $z\in 3^n\Zd$, and applying~\eqref{e.boundminimizers},~\eqref{e.whipfbar} and~\eqref{e.whipsolenoidal}, we obtain
\begin{align} \label{e.cutboundarysbcs}
\lefteqn{ \E \left[\fint_{\cu_{2n}} \left| \xi(x) - 1 \right|^2 \left| \nabla w(x) \right|^2\,dx \right]  \leq \frac1{\left| \cu_{2n} \right|} \E\left[ \int_{D} \left| \nabla w(x) \right|^2\,dx\right] } \qquad & \\
& \leq \frac{C}{\left| \cu_{2n} \right|} \E \left[ \int_{D} \left| \mathbf{f}(x) \right|^2\,dx + \int_{\cu_{2n}} \left| \mathbf{f}(x) - \nabla w(x) \right|^2\,dx  \right] \notag \\
& \leq \frac{C}{\left| \cu_{2n} \right|} \left( CK_0^2 3^{n(d-2\delta/(1+\delta))}\left| \cu_{n} \right| + C| \cu_{2n} |( \tau_n+C'\kappa_n)  \right) \notag \\
&  \leq C  \tau_n+C'\kappa_n. \notag
\end{align}
Combining the previous inequality with~\eqref{e.boundonw2} and~\eqref{e.nvnwsplit}, we obtain the desired estimate for the first term on the left of~\eqref{e.closenvnw}. 

\smallskip

The argument for estimating the second term on the left side of~\eqref{e.closenvnw} is similar. We start from the identity 
\begin{equation*} \label{}
\mathbf{h}(x) - \nabla\cdot \mathbf{T}(x)= \left( \xi(x) - 1 \right) \nabla \cdot \mathbf{T}(x) - \nabla \tilde h(x)
\end{equation*}
and observe that the equation for $\tilde h$ can be written as
\begin{equation*} \label{}
- \Delta \tilde h = -\nabla \cdot \left( (\xi-1)\nabla\cdot\mathbf{T} \right) \quad \mbox{in} \ \cut_{2n},
\end{equation*}
and therefore we have 
\begin{equation*} \label{}
\int_{\cut_{2n}} \left| \nabla \tilde h(x) \right|^2\,dx \leq C \int_{\cut_{2n} } \left| \xi(x) - 1 \right|^2 \left| \nabla \cdot \mathbf{T}(x) \right|^2\,dx. 
\end{equation*}
It thus suffices to show that 
\begin{equation*} \label{}
\E \left[ \int_{\cu_{2n} } \left| \xi(x) - 1 \right|^2 \left| \nabla \cdot \mathbf{T}(x) \right|^2\,dx \right] \leq C\tau_n + C'\kappa_n. 
\end{equation*}
The proof of this estimate is very similar to~\eqref{e.cutboundarysbcs}, and so we omit the details. 

\smallskip

\emph{Step 8.}
We estimate the expected difference in $L^2(z+\cu_{n})$ between $(\nabla v,\h)$ and $(\nabla u,\g)(\cdot,z+\cu_{n+1})$ for each $z \in 3^n\Zd \cap \cut_{2n} $. The claim is that 
\begin{multline} \label{e.patchcomp}
\E \left[ 3^{-dn} \sum_{z \in 3^n\Zd \cap \cut_{2n} }  \fint_{z+\cu_n}  \left| \nabla v(x) - \nabla u(x,z+\cu_{n+1}) +\overline P_n \right|^2\, dx\right] \\
+ \E \left[ 3^{-dn} \sum_{z \in 3^n\Zd \cap \cut_{2n}}  \fint_{z+\cu_n}  \left| \h(x) - \g(x,z+\cu_{n+1}) +\overline Q_n \right|^2   \,dx  \right] \leq C(\tau_n+C'\kappa_n).
\end{multline}
Indeed, for each $z \in  3^n\Zd\cap \cut_{2n}$ and $x\in z+\cu_n$, we have 
\begin{multline*} \label{}
\nabla v(x) - \nabla u(x,z+\cu_{n+1}) +\overline P_n \\
= \left( \nabla v(x) - \nabla w(x) \right) + \left( \mathbf{f}(x) -\nabla u(x,z+\cu_{n+1}) + \overline P_n \right) - \left( \overline{\mathbf{f}} - \nabla \cdot \mathbf{S} \right).
\end{multline*}
Thus the estimate for the first term on the left of~\eqref{e.patchcomp} is a consequence of~\eqref{e.gradlocalize} (note that $\zeta \equiv 1$ on $z+\cu_{n+1}$ for every $z\in 3^n\Zd \cap \cut_{2n}$),~\eqref{e.whipfbar},~\eqref{e.whipsolenoidal} and~\eqref{e.closenvnw}. The estimate for the other term follows immediately from~\eqref{e.gradlocalize},~\eqref{e.whipfbar},~\eqref{e.bangbang} and~\eqref{e.closenvnw}.

\smallskip

\emph{Step 9.}
We complete the argument by deriving~\eqref{e.exhibition}. By~Lemma~\ref{l.upperunifconv}, we have, for each $z \in 3^n\Zd\cap \cut_{2n}$,
\begin{align*} \label{}
\lefteqn{ \fint_{z+\cu_n} F(\overline P_n+\nabla v(x),\overline Q_n+\h(x),x) \, dx } \qquad &  \\
&   \leq 2 \fint_{z+\cu_n} F(\nabla u(x,z+\cu_{n+1}), \g(x,z+\cu_{n+1}),x) \, dx - \mu(z+\cu_{n}) \\
&  \qquad 
\begin{multlined}[.8\textwidth]
+ C \fint_{z+\cu_n} \Big( \left| \nabla v(x) - \nabla u(x,z+\cu_{n+1})+\overline P_n \right|^2  \\  
+ \left| \h(x) - \g(x,z+\cu_{n+1})+\overline Q_n \right|^2 \Big) \, dx.\end{multlined}
\end{align*}
In view of~\eqref{e.xicutoff}, it is convenient to denote $\mathcal Z_n':=  \left\{ z\in 3^n\Zd \,:\, \ z+\cu_{n+1} \not\subseteq  \cut_{2n} \right\}$ and  $U:= \cup_{z\in\mathcal Z_n'} (z+\cu_n)$. Note that $\xi$ vanishes on $\cut_{2n} \setminus U$ and thus $v$ and $\h$ do as well. Observe also that 
\begin{equation*} \label{}
\left| \cut_{2n} \setminus U \right| \leq C3^{-n} \left| \cut_{2n} \right| \quad \mbox{and} \quad \left| \left| \mathcal Z_n' \right| \left| \cu_n \right| - \left|\cut_{2n}\right| \right| \leq C3^{-n} \left| \cut_{2n} \right|.
\end{equation*}
Now we take the expectation of the previous inequality, sum over $z\in \mathcal Z_n'$, using~\eqref{e.FK0ineq},~\eqref{e.mubound2},~\eqref{e.boundminimizers},~\eqref{e.gridtrap}, Lemma~\ref{l.upperunifconv},~\eqref{e.patchcomp}, stationarity and the above observations to obtain
\begin{align*}
\lefteqn{ \E\Ll[\int_{\cut_{2n}} F(\overline P_n+\nabla v(x),\overline Q_n+\h(x),x) \, dx \Rr] } \qquad & \\ 
& \begin{multlined}[.8\textwidth]
\leq \sum_{z \in\mathcal Z_n'} \E\Ll[\int_{z+\cu_n} F(\overline P_n+\nabla v(x),\overline Q_n+\h(x),x) \, dx \Rr] \\
+ \E \left[ \int_{\cut_{2n} \setminus U} F(\overline P_n,\overline Q_n,x) \, dx \right] \end{multlined} \\
& \leq \left| \cut_{2n} \right|\left(  \E \left[ \mu(\cu_n) \right]  +C(\tau_n+C'\kappa_n) \right).
\end{align*}
Dividing by~$|\cut_{2n}|$ yields~\eqref{e.exhibition} and completes the proof of the lemma.
\end{proof}

\subsection{Proof of Proposition~\ref{p.muconv}}

We use the flatness theory developed in the previous subsection to prove Proposition~\ref{p.muconv}.
We begin by iterating Lemma~\ref{l.iterable} to obtain a rate of convergence for $\E\left[ \mu(\cut_n) \right]$ and $\E \left[ \mu_0(\cut_n, \overline P, \overline Q) \right]$ as $n\to \infty$ to their common limit $\overline \mu$. 

\begin{lemma}
\label{l.iteration}
There exist $\alpha(d,\Lambda,\beta)>0$, $C(d,\Lambda,\beta,C_3) \geq 1$ and $\overline P,\overline Q\in\Rd$ such that 
\begin{equation}\label{e.primitivedual}
\overline \mu = \overline \mu_0(\overline P,\overline Q)
\end{equation}
and, for every $n\in\N$,
\begin{equation} \label{e.muexpconv3}
\left| \E \left[ \mu(\cut_n) \right] - \overline\mu \right| + \left| \E \left[ \mu_0(\cut_n,\overline P,\overline Q) \right] - \overline\mu \right| \leq CK_0^2 3^{-n\alpha}. 
\end{equation}
\end{lemma}
\begin{proof}

We split the proof into three steps. The convention for constants is different from the previous subsection. Here $C\geq 1$ denotes a constant depending only on $(d,\Lambda,\beta,C_3)$, while $\tilde C\geq 1$ and $0< \tilde c <1$, which are used only in~Step~1, are allowed to depend on only $(d,\Lambda,\beta)$. These may vary in each occurrence. 

\smallskip

\emph{Step 1.} We show that there exists $\alpha(d,\Lambda,\beta)>0$ such that, for every $n\in\N$,
\begin{equation}
\label{e.muexpconv}
\overline \mu - \E \left[ \mu(\cut_n) \right] \leq C K_0^2 3^{-n\alpha}.
\end{equation}
By~\eqref{e.nuconv}, \eqref{e.musord} and Lemma~\ref{l.iterable}, we have
\begin{equation}\label{e.muexpconvwith0}
\overline \mu \leq \E \left[ \mu_0(\cut_{2n},\overline P_n,\overline Q_n) \right] \leq \E\Ll[\mu(\cu_n) \Rr] + \left( \tilde C \tau_n + C\kappa_n\right),
\end{equation}
which by \eqref{e.cutdown} can be upgraded to
$$
\overline \mu \leq \E \left[ \mu(\cut_n) \right] +  \left( \tilde C \tau_n + C\kappa_n\right).
$$
Let $\mu_n:= \E \left[ \mu(\cut_n) \right]$. Recalling the definition of $\tau_n$ in \eqref{e.taun} and using \eqref{e.cutup} yields
$$
\overline{\mu} \le \E \left[ \mu(\cut_n) \right] + \tilde C\Ll(\E \left[ \mu(\cut_{n+1}) \right] - \E \left[ \mu(\cut_n) \right]\Rr) + C \kappa_n.
$$
Denoting $\tilde\mu_n: = \overline \mu - \E\left[ \mu(\cut_n) \right]$, we can rewrite this as
$$
\tilde\mu_{n+1} - \tilde\mu_n \le- \tilde C^{-1} \, \tilde\mu_n + CK_0^2 3^{-n\beta/(1+\beta)}.
$$
Letting $\tilde c: = 1- \tilde C^{-1} < 1$, we arrive at
$$
\frac{\tilde\mu_{n+1}}{\tilde c^{n+1}} - \frac{\tilde\mu_n}{\tilde c^n} \le \frac C{\tilde c^{n+1}}K_0^2 3^{-n\beta/(1+\beta)}.
$$
Summing the telescopic series, we get the discrete Duhamel formula
$$
\tilde\mu_{n} \le \tilde c^{n}\,  \tilde\mu_0 +  CK_0^2 \sum_{k = 0}^{n-1} \tilde c^{n-k-1} 3^{-k\beta/(1+\beta)},
$$
and \eqref{e.muexpconv} follows since $\tilde \mu_0 \leq CK_0^2$.

\smallskip

Notice that by~\eqref{e.Etrimmono} and~\eqref{e.cutup}, \eqref{e.muexpconv} implies that for every $n\in\N$,
\begin{equation} \label{e.muexpconv2}
\left| \E \left[ \mu(\cut_n) \right] - \overline\mu \right| + \left| \E \left[ \mu(\cu_n) \right] - \overline\mu \right| \leq CK_0^2 3^{-n\alpha}. 
\end{equation}
In particular,~$\tau_n \leq CK_0^2 3^{-n\alpha}$ by the triangle inequality.

\smallskip

\emph{Step 2.} We argue that there exist $\overline P,\overline Q\in\Rd$ such that 
\begin{equation}\label{e.PQbar}
\left| \overline P -  \overline P_n \right| + \left| \overline Q -\overline Q_n  \right| \leq CK_0^2 3^{-n\alpha}.
\end{equation}
By Lemma~\ref{l.unifconv},
\begin{multline*}
\int_{\cu_n} |\nabla u(x,\cu_{n+1}) - \nabla u(x,\cu_{n})|^2 \, dx \\
\le C \Ll(\int_{\cu_n} F(\nabla u(x,\cu_{n+1}), \g(x,\cu_{n+1}),x) \, dx - \mu(\cu_n)\Rr).
\end{multline*}
By \eqref{e.gridtrap}, the latter is bounded by a multiple of $\tau_n$. By \eqref{e.gridtrap} and \eqref{e.muexpconv2}, we thus have
\begin{equation*} \label{}
\left| \E \left[ P(\cu_{n+1}) \right] - \E \left[ P(\cu_{n})\right] \right| + \left| \E \left[ Q(\cu_{n+1}) \right] - \E \left[ Q(\cu_{n})\right] \right| \leq CK_0^2 3^{-n\alpha}. 
\end{equation*}
This implies that $\{ \E \left[ P(\cu_{n})\right] \}_{n\in\N}$ and $\{ \E \left[ Q(\cu_{n})\right] \}_{n\in\N}$ are Cauchy sequences in~$\Rd$, and thus there exist $\overline P,\overline Q\in\Rd$ such that 
\begin{equation} \label{e.salims}
\E \left[ P(\cu_{n})\right] \rightarrow \overline P \quad \mbox{and} \quad \E \left[ Q(\cu_{n})\right] \rightarrow \overline Q \quad \mbox{as} \ n\to \infty.
\end{equation}
Moreover, we have
\begin{equation*} \label{}
\left| \overline P  - \E \left[ P(\cu_{n})\right] \right| \leq \sum_{k = n}^{+\infty} \left| \E \left[ P(\cu_{k+1}) \right] - \E \left[ P(\cu_{k})\right] \right| \leq CK_0^2 3^{-n\alpha},
\end{equation*}
and similarly
\begin{equation*} \label{}
\left| \overline Q  - \E \left[ Q(\cu_{n})\right] \right| \leq CK_0^2 3^{-n\alpha}.
\end{equation*}
To complete the proof of~\eqref{e.PQbar}, it remains to show that 
\begin{equation*} \label{}
\left| \overline P_n -  \E \left[ P(\cu_{n})\right] \right| + \left| \overline Q_n  - \E \left[ Q(\cu_{n})\right] \right|  \leq CK_0^2 3^{-n\alpha}.
\end{equation*}
This follows from~\eqref{e.flatuntrimmed} and~\eqref{e.muexpconv2}.

\smallskip

For future reference, we notice that~\eqref{e.boundminimizers} also implies that
\begin{equation}\label{e.boundsPnQn}
\left| \overline P \right| + \left| \overline Q \right| \leq CK_0^2.
\end{equation}

\smallskip

\emph{Step 3.} By redefining $\alpha(d,\Lambda,\beta)>0$, if necessary, we argue that, for every $n\in\N$, 
\begin{equation} \label{e.convexpboth}
\left| \E \left[ \mu_0(\cut_n,\overline P,\overline Q) \right] - \E\left[ \mu(\cut_n) \right] \right| \leq CK_0^2 3^{-n\alpha}.
\end{equation}
We have
\begin{align*}
\lefteqn{ \E \left[ \mu_0(\cut_{2n},\overline P,\overline Q) \right]  - \E \left[ \mu(\cut_{2n}) \right] } \qquad & \\
& \leq \left| \E \left[ \mu_0(\cut_{2n},\overline P,\overline Q) \right] - \E \left[ \mu_0(\cut_{2n},\overline P_n,\overline Q_n) \right]\right| \\
& \qquad + \left( \E \left[ \mu_0(\cut_{2n},\overline P_n,\overline Q_n) \right] - \E \left[ \mu(\cut_{n}) \right] \right) 
 + \left( \E \left[ \mu(\cut_n) \right] -  \E \left[ \mu(\cut_{2n}) \right] \right).
\end{align*}
The first term is estimated by~\eqref{e.mu0cont},~\eqref{e.PQbar} and~\eqref{e.boundsPnQn}:
\begin{equation*} \label{}
\Ll|\E \left[ \mu_0(\cut_{2n},\overline P,\overline Q) \right] - \E \left[ \mu_0(\cut_{2n},\overline P_n,\overline Q_n) \right]\right| \leq C K_0^2 3^{-n\alpha}.
\end{equation*}
The second term is estimated by~\eqref{e.muexpconvwith0} and~\eqref{e.muexpconv2} and the third term by~\eqref{e.muexpconv2}. Combining these, we deduce
\begin{align*} \label{}
\left| \E \left[ \mu_0(\cut_{2n},\overline P,\overline Q) \right]  - \E \left[ \mu(\cut_{2n}) \right] \right|  = \E \left[ \mu_0(\cut_{2n},\overline P,\overline Q) \right]  - \E \left[ \mu(\cut_{2n}) \right]  \leq CK_0^2 3^{-n\alpha}. 
\end{align*}
Replacing $\alpha$ with $\alpha/2$ yields~\eqref{e.convexpboth}.  

\smallskip

Notice that~\eqref{e.convexpboth} implies $\overline \mu = \overline \mu_0(\overline P,\overline Q)$, that is, we have proved~\eqref{e.primitivedual}. It also implies, by~\eqref{e.convexpboth} and the triangle inequality, that
\begin{equation} \label{e.convexpmu0}
\left| \E \left[ \mu_0(\cut_n,\overline P,\overline Q)\right] - \overline \mu \right| \leq CK_0^2 3^{-n\alpha}.
\end{equation}
This completes the proof of the lemma.
\end{proof}

We next upgrade the convergence from the previous lemma, using the mixing conditions, to obtain estimates on stochastic moments of $\left|\mu-\overline \mu\right|$ and $\left| \mu_0-\overline \mu \right|$ in the triadic cubes. Recall that $\mathcal E_*$ is defined in~\eqref{e:def:Estar}. 

\begin{lemma}
\label{l.stochasticupgrade}
For every $\theta \in (0,\beta)$, there exist an exponent~$s_0(d,\Lambda,\beta,\alpha)\geq 1$ and a constant~$C(d,\Lambda,\beta,C_3,\alpha)\geq 1$ such that, for every $n\in\N$ and $s \geq s_0$,
\begin{equation} 
\label{e.Lsbound}
\E \left[ \left| \mathcal E_*(\cut_n) \right|^s \right] \leq \left(CK_0^2\right)^s 3^{-\theta n}.
\end{equation}
Under the additional assumption that~(P4) holds, there exist~$\alpha(d,\Lambda,\beta,\gamma)>0$ and~$C(d,\Lambda,\beta,C_3,\gamma,C_4)\geq1$ such that, for every $n\in\N$,
\begin{equation} \label{e:expon-mom}
 \E \left[ \exp\left( 3^{nd\gamma /(d+\gamma)} \frac{\mathcal E_*(\cut_{n})}{CK_0^2} \right) \right] \leq \exp\left( C 3^{n \left( d\gamma /(d+\gamma)-\alpha\right)} \right).
\end{equation}
\end{lemma}

\begin{proof}
In this argument, we drop the dependence of $\mu_0$ on $(\overline P,\overline Q)$ for convenience and denote
\begin{equation}\label{e.sigman}
\sigma_n:= K_0^23^{-n\alpha},
\end{equation}
where the exponent $\alpha > 0$ implicit in $\sigma_n$ depends only on $(d,\Lambda,\beta)$ and may change in each occurrence. Here $C$ depends only on $(d,\Lambda,\beta,C_3)$ and may vary in each occurrence. 

\smallskip

The first four steps are concerned with the proof of the first statement. Note that~\eqref{e.Lsbound} is much easier to prove under slightly stronger (although still relatively weak) mixing conditions than~(P3), such as for example ``$\psi$-mixing." We have to work a bit because ``$\alpha$--mixing" is a very weak condition. 

\smallskip

\emph{Step 1.} We claim that, for every $n\in\N$, at least one of the following two estimates must hold:
\begin{equation}\label{e.alter1}
\E \left[ \left( \overline \mu - \mu(\cut_n) \right)_+ \right]^2 \leq  \frac{5}{8} \E \left[ \left( \overline \mu - \mu(\cut_n) \right)_+^2 \right] + C\sigma_n^2,
\end{equation}
or
\begin{equation} \label{e.alter2}
\E \left[ \left(  \mu_0(\cut_n) -\overline \mu \right)_+ \right]^2 \leq  \frac{5}{8} \E \left[ \left(  \mu_0(\cut_n) -\overline \mu \right)_+^2 \right] + C\sigma_n^2.
\end{equation}
(In fact, one can replace $5/8$ by any number larger than $1/2$.)
We proceed with the proof of this alternative by noting that, since $\mu_0(\cut_n) -\mu(\cut_n) \geq 0$, Chebyshev's inequality yields, for every $t>0$, 
\begin{equation*}
\P \left[ \mu_0(\cut_n) -\mu(\cut_n) \geq t\sigma_n \right] \leq (t\sigma_n)^{-1} \E \left[ \mu_0(\cut_n) -\mu(\cut_n) \right] .
\end{equation*}
By Lemma~\ref{l.iteration}, taking $C$ sufficiently large, we obtain that 
\begin{equation*}
\P \left[ \mu_0(\cut_n) -\mu(\cut_n) \geq C\sigma_n \right] \leq \frac{1}{8}. 
\end{equation*}
It follows that 
\begin{equation*}
\P\left[ \mu(\cut_n) \leq \overline \mu - C\sigma_n \ \ \mbox{and} \ \  \mu_0(\cut_n) \geq \overline \mu + C\sigma_n \right] \leq \frac{1}{8}. 
\end{equation*}
This implies
\begin{equation} \label{e.oneduh}
\min\left\{ \P\left[ \mu(\cut_n) \leq \overline \mu - C\sigma_n \right] ,\, \P\left[ \mu_0(\cut_n) \geq \overline \mu + C\sigma_n \right]  \right\} \leq \frac{9}{16}.
\end{equation}
From~\eqref{e.oneduh}, we obtain the desired claim observing that for every nonnegative random variable $X$ and $s> 0$, 
\begin{equation*}
\E\left[ X \right] \leq \E\Ll[X \mathbf{1}_{X > s}\Rr] + s \leq \Ll(\P \left[ X>s \right] \E\left[ X^2 \right]\Rr)^{1/2} + s
\end{equation*}
and applying Young's inequality.

\smallskip

\emph{Step 2.}
We show using the mixing condition~(P3) that
\begin{multline}
\label{e.vardecay}
\E\Ll[(\mu-\mu(\cut_{n+1}))_+^2\Rr] \\
\le \frac65 \Ll(3^{-d} \var\left[ \left( \overline \mu - \mu(\cut_{n})\right)_+  \right]  + \E\Ll[(\mu - \mu(\cut_{n}))_+\Rr]^2   \Rr) + C \,  \sigma_n^2,
\end{multline}
and
\begin{multline}
\label{e.vardecay0}
\E \left[  \left( \mu_0(\cut_{n+1}) - \overline \mu \right)_+^2  \right] \\ 
\leq \frac65 \Ll(3^{-d} \,\var\left[\Ll(\mu_0(\cut_{n}) - \overline \mu \right)_+  \right] + \E\left[ \left( \mu_0(\cut_{n}) - \overline \mu \right)_+  \right]^2 \Rr) + C \, \sigma_n^2.
\end{multline}
(In fact, one can replace $6/5$ by any number larger than $1$.)
The arguments for~\eqref{e.vardecay} and~\eqref{e.vardecay0} are almost identical, so we only prove~\eqref{e.vardecay}. 
Recall from \eqref{e.trimmono} that
$$
\mu(\cut_{n+1}) \ge 3^{-d} \sum_{\cut_n(x) \subset \cut_{n}} \mu(\cut_n(x)) - C \sigma_n,
$$
hence
$$
(\mu - \mu(\cut_{n+1}))_+ \le  \underbrace{3^{-d} \sum_{\cut_n(x) \subset \cut_{n+1}} (\mu - \mu(\cut_{n}(x)))_+}_{=:\mathcal S_n} + C \sigma_n.
$$
We have
$$
\E\Ll[ \mathcal S_n^2 \Rr] = \var\Ll[ \mathcal S_n\Rr] + \E\Ll[\mathcal S_n\Rr]^2,
$$
and by stationarity, $\E[\mathcal S_n] = \E\Ll[(\mu - \mu(\cut_{n}))_+\Rr]$. In order to estimate the variance, we use~(P3) in the form given by \eqref{e.howweuseP3}, \eqref{e.mubound2} and~\eqref{e.separate} to find that, for every $\cut_n(x) \neq\cut_n(y)$,
\begin{equation*}
\left| \cov\left[ \left( \overline \mu - \mu(\cut_{n}(x))\right)_+ ,\, \left( \overline \mu - \mu(\cut_{n}(y))\right)_+  \right]  \right| \leq CK_0^4 3^{-n\beta/(1+\beta)} \leq C \sigma_n^2,	
\end{equation*}
and therefore
\begin{multline*}
\var[\mathcal S_n] = 3^{-2d} \sum_{\cut_n(x),\cut_n(y) \subseteq \cut_{n+1}}  \cov\left[ \left( \overline \mu - \mu(\cut_{n}(x))\right)_+ ,\, \left( \overline \mu - \mu(\cut_{n}(y))\right)_+  \right] \\
 \leq 3^{-d} \var\left[ \left( \overline \mu - \mu(\cut_{n})\right)_+  \right] + C\sigma_n^2.
\end{multline*}
Summarizing and using Young's inequality, we obtain \eqref{e.vardecay}.

\smallskip

\emph{Step 3.} 
We show that for every $n\in\N$, at least one of the following two inequalities holds:
\begin{equation}\label{e.altdecay}
\E \left[  \left( \overline \mu - \mu(\cut_{n+1}) \right)_+^2  \right]  \leq \frac9{10} \, \E \left[  \left( \overline \mu - \mu(\cut_{n}) \right)_+^2  \right]   + C\sigma_n^2
\end{equation}
or 
\begin{equation}\label{e.altdecay0}
\E \left[  \left( \mu_0(\cut_{n+1}) - \overline \mu  \right)_+^2  \right]  \leq \frac9{10} \, \E \left[  \left( \mu_0(\cut_{n+1}) - \overline \mu  \right)_+^2  \right]   + C\sigma_n^2.
\end{equation}
In fact, we claim that~\eqref{e.alter1} implies~\eqref{e.altdecay} and~\eqref{e.alter2} implies~\eqref{e.altdecay0}, and thus the alternative follows from the one in Step~1. Indeed, by \eqref{e.vardecay}, we have
\begin{align*}
& \E\Ll[(\mu-\mu(\cut_{n+1}))_+^2\Rr] \\
& \quad \le \frac65 \Ll(3^{-d} \var\left[ \left( \overline \mu - \mu(\cut_{n})\right)_+  \right]  + \E\Ll[(\mu - \mu(\cut_{n}))_+\Rr]^2   \Rr) + C \,  \sigma_n^2 \\
& \quad = \frac65 \Ll(3^{-d} \E\left[ \left( \overline \mu - \mu(\cut_{n})\right)_+^2  \right]  + (1-3^{-d})\E\Ll[(\mu - \mu(\cut_{n}))_+\Rr]^2   \Rr) + C \,  \sigma_n^2 \\
& \quad \le \frac65 \Ll(3^{-d} + \frac{5}{8}(1-3^{-d}) \Rr) \E\left[ \left( \overline \mu - \mu(\cut_{n})\right)_+^2  \right]  + C \,  \sigma_n^2,
\end{align*}
where we assumed \eqref{e.alter1} to hold in the last step. This gives \eqref{e.altdecay}. We omit the proof of the second implication, which is very similar.

\smallskip

\emph{Step 4.} We show that, for each fixed $n\in\N$, 
\begin{equation}\label{e.altdecay2}
\min\left\{ \E \left[  \left( \overline \mu - \mu(\cut_{n}) \right)_+^2  \right],\, \E \left[  \left( \mu_0(\cut_{n}) - \overline \mu  \right)_+^2  \right]\right\}  \leq C\sigma_n^2.
\end{equation}
Let $A_{1}$ and $A_{2}$, respectively, be the subset of $\N$ consisting of~$n$ for which~\eqref{e.altdecay} and~\eqref{e.altdecay0} hold. We have $\N = A_1 \cup A_2$ and thus, for each $n\in\N$, at least one of the set $A_1\cap \{ 1,\ldots,n\}$ and $A_2\cap \{ 1,\ldots,n\}$ has at least $n/2$ elements. If the former, then by~\eqref{e.Etrimmono},~\eqref{e.altdecay},~\eqref{e.mubound2} and a simple computation, we have 
\begin{align*}
\E \left[  \left( \overline \mu - \mu(\cut_{n}) \right)_+^2  \right]  & \leq \Ll(\frac9{10}\Rr)^{n/2} \E \left[  \left( \overline \mu - \mu(\cut_{0}) \right)_+^2  \right] + C\sigma_n^2 \\
& \leq CK_0^4\Ll(\frac9{10}\Rr)^{n/2}   + C\sigma_n^2.
\end{align*}
If it is rather~$A_2\cap \{1,\ldots,n\}$, then we have a similar bound for~$\E \left[  \left( \mu_0(\cut_{n}) - \overline \mu  \right)_+^2  \right] $. Since the first term on the right side is bounded by~$C\sigma_n^2$ after a redefinition of~$\alpha$, we obtain~\eqref{e.altdecay2}.

\smallskip

\emph{Step 5.} We show that 
\begin{equation} \label{e.Lsbounds1}
\E \left[  \mathcal E_* (\cut_n) \right] \leq CK_0^2 3^{-n\alpha}. 
\end{equation}
Observe, using $\mu_0(U) \leq \mu(U)$ again, that we have, for any bounded domain $U\subseteq \Rd$, 
\begin{equation*} \label{}
\left| \mu_0(U) - \overline \mu \right|  \leq \mu_0(U) - \mu(U) + \left|\overline \mu- \mu(U)  \right|
\end{equation*}
and
\begin{equation*}
\left| \overline \mu - \mu(U) \right| \leq \mu_0(U) - \mu(U) + \left| \mu_0(U) - \overline \mu \right|.
\end{equation*}
Combining the previous two inequalities for $U=\cut_n$, taking their expectation and applying Lemma~\ref{l.iteration} (in particular~\eqref{e.convexpboth}), we get
\begin{multline*}
\max\left\{\E\left[ \left| \mu_0(\cut_n) - \overline \mu \right|\right] ,\, \E\left[ \left| \overline \mu - \mu(\cut_n) \right| \right]   \right\} \\ \leq  \min\left\{ \E\left[ \left| \mu(\cut_n) - \overline \mu \right| \right],\, \E\left[  \left| \mu_0(\cut_n) - \overline \mu \right| \right] \right\}  
+ \E\left[ \mu_0(\cut_n) - \mu(\cut_n) \right].
\end{multline*}
Note that a centered random variable $X$ satisfies $\E[X_+] = \E[X_-]$, so that $\E[|X|] = 2 \E[X_+]$. We use this observation and apply Lemma~\ref{l.iteration} twice to obtain
\begin{align*}
\E \left[ \left| \overline \mu - \mu(\cut_n) \right| \right] & \leq \E \left[ \left| \E \left[ \mu(\cut_n) \right]  - \mu(\cut_n) \right|\right] + C\sigma_n \\
& = 2 \E \left[ \left( \E \left[ \mu(\cut_n) \right]  - \mu(\cut_n) \right)_+\right] + C\sigma_n \\
& \leq 2 \E \left[ \left( \overline\mu  - \mu(\cut_n) \right)_+\right] + C\sigma_n.
\end{align*}
Similarly, we also have 
\begin{equation*}
\E \left[ \left| \mu_0(\cut_n) - \overline \mu \right| \right] \leq 2 \E \left[ \left(  \mu_0(\cut_n) - \overline \mu \right)_+\right] + C\sigma_n.
\end{equation*}
The previous three inequalities and~\eqref{e.altdecay2} yield 
\begin{align*}
\E\left[ \mathcal E_*(\cut_n) \right] & \leq 2\max\left\{\E\left[ \left| \mu_0(\cut_n) - \overline \mu \right|\right] ,\, \E\left[ \left| \overline \mu - \mu(\cut_n) \right| \right]   \right\} \\
& \leq 4\min\left\{ \E \left[  \left( \overline \mu - \mu(\cut_{n}) \right)_+  \right],\, \E \left[  \left( \mu_0(\cut_{n}) - \overline \mu  \right)_+  \right]\right\} + C\sigma_n  \\
& \leq C\sigma_n.
\end{align*}
This completes the proof of~\eqref{e.Lsbounds1}.

\smallskip

\emph{Step 6.} We upgrade the convergence rate in~\eqref{e.Lsbounds1} to the optimal one for large moments, using Lemma~\ref{l.finitemoments}, thereby completing the proof of the first statement of the lemma. The claim is that, for every $\theta \in (0,\beta)$, there exist an exponent $s_0(d,\Lambda,\beta,\theta)\geq 1$ and $C(d,\Lambda,\beta,\theta,C_3)\geq 1$ such that, for every $s\geq s_0$,
\begin{equation} \label{e.finitemomupgrade}
\E \left[\left| \mathcal E_*(\cut_n) \right|^s  \right]   \leq  \left( CK_0^2\right)^s 3^{-\theta n}.
\end{equation}
In this step, we allow the constant $C$ to depend additionally on~$\theta$. We also allow $\alpha(d,\Lambda,\beta)>0$ to vary in each occurrence, as usual. 

\smallskip

Fix $n,m\in\N$. Also fix $k\in\N$, which will be selected below and only depends on $(d,\Lambda,\beta,\theta)$. In particular, $k\leq C$. Using the subadditivity of $(\overline \mu - \mu(\cdot))_+$ in the form of~\eqref{e.trimmono} and then applying Lemma~\ref{l.finitemoments} and~(P3), we obtain
\begin{align*}
\lefteqn{ \E \left[ \left( \overline \mu - \mu(\cut_{n+m}) \right)_+^{2k}  \right]  } \qquad & \\
& \leq \E \left[ \left( 3^{-dm} \sum_{\cut_n(x) \subseteq \cut_{n+m}}  \left( \overline \mu - \mu(\cut_{n}(x))   \right)_+  + CK_0^2 3^{-n\beta/(1+\beta)}  \right)^{2k} \right]  \\
& \leq2 \E \left[ \left( 3^{-dm} \sum_{\cut_n(x) \subseteq \cut_{n+m}}  \left( \overline \mu - \mu(\cut_{n}(x))   \right)_+   \right)^{2k} \right] + (CK_0^2)^{2k} 3^{-2kn\beta/(1+\beta)} \\
& \leq CK_0^{4k} \left(  \max\left\{  \E\left[ \frac{\left( \overline \mu - \mu(\cut_{n}(x))   \right)_+}{CK_0^2} \right], \, 3^{-dm}  \right\}^{2k} + 3^{-\beta n}  + 3^{-kn\alpha}\right) .
\end{align*}
Inserting~\eqref{e.Lsbounds1}, we obtain
\begin{equation*} \label{}
\E \left[ \left( \overline \mu - \mu(\cut_{n+m}) \right)_+^{2k}  \right]   \leq  CK_0^{4k} \left(   \left( 3^{-n\alpha} + 3^{-dm}  \right)^{2k} +   3^{-\beta n} \right).
\end{equation*}
A slight reformulation of the previous inequality (replace~$n+m$ by~$n$ and~$n$ by~$n-m$) yields, for every~$n,m\in\N$ with~$m \leq n$,
\begin{equation*} \label{}
\E \left[ \left( \overline \mu - \mu(\cut_{n}) \right)_+^{2k}  \right]   \leq  CK_0^{4k} \left(   \left( 3^{-(n-m)\alpha} + 3^{-dm} \right)^{2k} +   3^{-\beta (n-m)} \right).
\end{equation*}
We may extend this inequality to $m\in\R_+$ with $m\leq n$ by adjusting the constant~$C$. We now select~$m$ as a function of~$n$ so that~$\beta(n-m) = \theta n$, that is,~$m(n):=(\beta-\theta)n/\beta$ and we thereby obtain, for every $n\in\N$, 
\begin{equation*} \label{}
\E \left[ \left( \overline \mu - \mu(\cut_{n}) \right)_+^{2k}  \right]   \leq  CK_0^{4k} \left(   \left( 3^{-n\theta\alpha/\beta} + 3^{-nd(\beta-\theta)/\beta}  \right)^{2k} +   3^{-\theta n} \right).
\end{equation*}
We now select $k\in\N$ to be the smallest positive integer satisfying
\begin{equation*} \label{}
2k \max\left\{ \frac{\theta \alpha}{\beta}, \, \frac{d(\beta-\theta)}{\beta} \right\} \geq \theta.
\end{equation*}
Thus, as we had promised, $k \leq C$. We deduce that, for every $n\in\N$,
\begin{equation*} \label{}
\E \left[ \left( \overline \mu - \mu(\cut_{n}) \right)_+^{2k}  \right] \leq  CK_0^{4k}Ck 3^{-\theta n} \leq CK_0^{4k} 3^{-\theta n}.
\end{equation*}
Fix $s_0:= 2k$ and observe that we may use the previous inequality and~\eqref{e.mubound2} to obtain, for every $s\geq s_0$, 
\begin{equation*} \label{}
\E \left[ \left( \overline \mu - \mu(\cut_{n}) \right)_+^{s}  \right] \leq (CK_0^2)^{s-2k} \, \E \left[ \left( \overline \mu - \mu(\cut_{n}) \right)_+^{2k}  \right] \leq (CK_0^2)^s 3^{-\theta n}.
\end{equation*}
By a very similar argument, we also obtain the bound 
\begin{equation*} \label{}
\E \left[  \left(  \mu_0(\cut_{n+m},\overline P,\overline Q) - \overline \mu \right)_+^{s}  \right] \leq (CK_0^2)^s 3^{-\theta n}.
\end{equation*}
We next use the fact that, due to $\mu(U) \leq \mu_0(U,\overline P,\overline Q)$, we have
\begin{equation} \label{e.trickwithplus}
\mathcal E_*(U)= \left| \overline \mu - \mu(U) \right| + \left| \mu_0(U,\overline P,\overline Q) - \overline \mu \right|  \leq 2\left( \overline \mu - \mu(U) \right)_+ + 2\left( \mu_0(U,\overline P,\overline Q) - \overline \mu \right)_+
\end{equation}
and combining this with the previous inequalities to obtain~\eqref{e.finitemomupgrade}.

\smallskip

\emph{Step 7.} In this last step, we upgrade the convergence to exponential moments under the additional assumption that~(P4) holds. This is the first and only place in the paper that we use~(P4), and is analogous to the previous step except that we use the much stronger Lemma~\ref{l.concentration} in place of~Lemma~\ref{l.finitemoments}.

\smallskip

The convention for the constants in this step is different from the rest of the proof: here we allow~$C$ to depend on the parameters~$(\gamma,C_4)$ in~(P4) in addition to~$(d,\Lambda,\beta,C_3)$, and~$\alpha$ may depend also on~$\gamma$ in addition to~$(d,\Lambda,\beta)$. 

\smallskip

Fix $n,m\in\N$ and $0<t \leq (C_1 K_0^2)^{-1} 3^{dm}$, where the previous~$C_1:=C\geq1$ is fixed large enough that, for all bounded Lipschitz domains $U\subseteq\Rd$,
\begin{equation} \label{e.yesbounded}
\P \left[ (\overline \mu - \mu(U))_+ \leq  C_1 K_0^2 \right]=1.
\end{equation}
Now we compute, using~\eqref{e.trimmono} and Lemma~\ref{l.concentration}:
\begin{multline*}
\log \E \left[  \exp\left( t \left( \overline \mu - \mu(\cut_{n+m}) \right)_+ \right)   \right]  \\
\begin{aligned}
& \leq \log \E \left[ \exp\left( t3^{-dm}  \sum_{\cut_n(x) \subseteq \cut_{n+m}}  \left( \overline \mu - \mu(\cut_n(x))  \right)_+ \right)\right] + CtK_0^23^{-n\beta/(1+\beta)} \\
& = \log \E \left[ \prod_{\cut_n(x) \subseteq \cut_{n+m}} \exp\left( t 3^{-dm} \left( \overline \mu - \mu(\cut_n(x)) \right)_+ \right)  \right] + CK_0^2t3^{-n\alpha}   \\
&\leq C\left( t\E\left[  \left( \overline \mu - \mu(\cut_n) \right)_+ \right] +  K_0^23^{dm} \exp\left( -3^{n\gamma} \right) \exp\left( CK_0^2t \right) \right)+ CK_0^2t3^{-n\alpha}.
\end{aligned}
\end{multline*}
We now apply~\eqref{e.Lsbounds1} to obtain, for every $n,m\in\N$ and $0<t \leq (C K_0^2)^{-1} 3^{dm}$,
\begin{align*} \label{}
\frac1t \log \E \left[  \exp\left( t \left( \overline \mu - \mu(\cut_{n+m}) \right)_+ \right)   \right] 
 & \leq CK_0^2 3^{-n\alpha} + CK_0^2t^{-1}3^{dm} \exp\left( CK_0^2t -3^{n\gamma}\right).
\end{align*}
A very similar computation yields
\begin{multline*} \label{}
\frac1t \log \E \left[  \exp\left( t \left(  \mu_0(\cut_{n+m},\overline P,\overline Q) - \overline \mu \right)_+ \right) \right] \\
  \leq CK_0^2 3^{-n\alpha} + CK_0^2t^{-1}3^{dm} \exp\left( CK_0^2t -3^{n\gamma}\right).
\end{multline*}
Combining these inequalities and using~\eqref{e.trickwithplus}, we obtain
\begin{equation*} \label{}
\frac1t \log \E \left[ \exp\left( t\mathcal E_*(\cut_{n+m}) \right) \right] \leq CK_0^2 3^{-n\alpha} + CK_0^2t^{-1}3^{dm} \exp\left( CK_0^2 t -3^{n\gamma}\right).
\end{equation*}
A reformulation of this inequality (as in Step~6, we replace~$n+m$ with~$n$ and~$n$ with~$n-m$) yields, for every~$n,m\in\N$ with~$m\leq n$ and~$t>0$ as above,
\begin{equation*} \label{}
\frac1t \log \E \left[ \exp\left( t\mathcal E_*(\cut_{n}) \right) \right] \leq CK_0^2 3^{-(n-m)\alpha} + CK_0^2t^{-1}3^{dm} \exp\left( CK_0^2t -3^{(n-m)\gamma}\right).
\end{equation*}
We may extend the previous inequality to~$m\in\R_+$ with~$m\leq n$ by adjusting the constants~$C$. We now select $m:= n \gamma /(d+\gamma)$ and $t:= (CK_0^2)^{-1} 3^{dm}$, with $C \geq C_1$ chosen sufficiently large that the expression inside the exponential on the right side of the previous inequality above can be estimated by
\begin{equation*}
CK_0^2t - 3^{(n-m)\gamma} \leq \frac12 3^{dm} - 3^{(n-m)\gamma} = -\frac12 3^{(n-m)\gamma}.
\end{equation*}
Substituting, we get
\begin{align*} \label{}
 \log \E \left[ \exp\left( 3^{dm} \frac{\mathcal E_*(\cut_{n})}{CK_0^2} \right) \right] & \leq  C3^{dm -(n-m)\alpha} + C3^{dm}\exp\left( -\frac12 \cdot 3^{nd\gamma / (d+\gamma)} \right) \\
 & \leq C 3^{dm-(n-m)\alpha}.
\end{align*}
This is~\eqref{e:expon-mom}. 
\end{proof}

We next complete the proof of Proposition~\ref{p.muconv}. Most of the heavy lifting has already been accomplished, and the statement of Lemma~\ref{l.stochasticupgrade} is already quite close to that of Proposition~\ref{p.muconv}. What is left is to allow for arbitrary translations of the cubes.

\begin{proof}[{Proof of Proposition~\ref{p.muconv}}]
Here $C$ is a positive constant depending only on $(d,\Lambda,\beta,C_3)$ while $\alpha$ depends only on $(d,\Lambda,\beta)$, and these may vary in each occurrence. As above, we drop the dependence of $\mu_0$ on $(\overline P,\overline Q)$ and let $\sigma_n$ be denoted by~\eqref{e.sigman}, where the exponent $\alpha$ implicit in $\sigma_n$ may change in each occurrence.

\smallskip

For $y\in\Rd$, let $[y]$ be the nearest element of $\Zd$ to $y$ (where we resort to the lexicographical ordering by the indices in case of nonuniqueness). Observe that
\begin{equation*}
[y]+\cut_n \subseteq y+\cu_n \quad \mbox{and} \quad \left| (y+\cu_n) \setminus ([y]+\cut_n) \right| \leq C3^{-n\beta/(1+\beta)} |\cu_n|.
\end{equation*}
Similar to the proof of~\eqref{e.cutup}, it follows from~\eqref{e.subadd} and~\eqref{e.mubound2} that, for each $y\in \Rd$ and $n\in\N$, 
\begin{equation*}
\mu(y+\cu_n) \geq \mu([y]+\cut_n) - CK_0^23^{-n\beta/(1+\beta)} \geq \mu([y]+\cut_n) - C\sigma_n  \quad \mbox{$\P$--a.s.}
\end{equation*}
In particular, for every $R\geq 1$,
\begin{equation*}
\sup_{y\in B_R} \left( \overline \mu - \mu(y+\cu_n) \right)_+ \leq \sup_{z\in B_{CR}} \left( \overline \mu - \mu(z+\cut_n) \right)_+ + C\sigma_n  \quad \mbox{$\P$--a.s.}
\end{equation*}
Similarly,
\begin{equation*}
\sup_{y\in B_R} \left( \mu_0(y+\cu_n) -\overline \mu \right)_+ \leq \sup_{z\in B_{CR}} \left( \mu_0(z+\cut_n) - \overline \mu\right)_+  + C\sigma_n
\quad \mbox{$\P$--a.s.}
\end{equation*}
The previous two inequalities and~\eqref{e.trickwithplus} imply
\begin{equation*}
\sup_{y\in B_R} \mathcal E_*(y+\cu_n) \leq C\sup_{z\in B_{CR}}\mathcal E_*(z+\cut_n) + C\sigma_n \quad  \mbox{$\P$--a.s.}
\end{equation*}
Hence by a union bound and stationarity, we get, for any $\zeta\geq C\sigma_n$,
\begin{equation*}
\P\left[ \sup_{y\in B_R} \mathcal E_*(y+\cu_n) \geq \zeta  \right]  \leq \left| B_{CR} \right|\, \P\left[ \mathcal E_*(\cut_n) \geq C\zeta \right] \leq CR^d\, \P\left[ \mathcal E_*(\cut_n) \geq C\zeta \right].
\end{equation*}

\smallskip

To prove the first statement of the proposition, it remains to show that, for any $\theta \in (0,\beta)$, there exists $s_0(d,\Lambda,\beta,\theta)\geq 1$ such that, for every $s\in[s_0,\infty)$ and $t\geq 1$,
\begin{equation}\label{e.almdoneP3}
\P \left[ \mathcal E_*(\cut_n) \geq C K_0^23^{-n\theta / s} t\right] \leq Ct^{-s}.
\end{equation}
This follows at once from Chebyshev's inequality and~\eqref{e.Lsbound}.

\smallskip

To prove the second assertion of the proposition, we have left to show that, under assumption~(P4), that for any $t\geq 1$ and $s\in (0,d\gamma/(d+\gamma))$, we have
\begin{equation}\label{e.almdoneP4}
\P \left[ \mathcal E_*(\cut_n) \geq C K_0^23^{-n\min\{ \alpha, d\gamma/(d+\gamma)-s\}} t\right] \leq C\exp\left( -\frac12 3^{ns} t \right),
\end{equation}
where the constants $\alpha$ and $C$ have appropriate dependence as stated in the proposition. Fix $s\in (0,d\gamma/(d+\gamma))$, $n\in\N$ and $t\geq 1$ and compute, using Chebyshev's inequality and~\eqref{e:expon-mom}:
\begin{align*}
\P \left[ 3^{nd\gamma /(d+\gamma)} \frac{\mathcal E_*(\cut_{n})}{CK_0^2}  \geq 3^{ns} t \right] 
& \leq \exp\left( -3^{ns} t \right) \E \left[ \exp\left( 3^{nd\gamma /(d+\gamma)} \frac{\mathcal E_*(\cut_{n})}{CK_0^2} \right) \right] \\
& \leq \exp\left( -3^{ns} t + C3^{n \left( d\gamma /(d+\gamma)-\alpha\right)}  \right).
\end{align*}
Suppose that $s \geq (d\gamma /(d+\gamma) - \alpha) + c$.  Then the expression in the exponential in the last expression is at most $-\frac12 3^{ns} t$ for $n\geq C$, and thus, for such~$s$, we obtain, for every $n\in\N$ and $t\geq 1$,
\begin{equation*}
\P \left[ \frac{\mathcal E_*(\cut_{n})}{CK_0^2}  \geq 3^{-n(d\gamma /(d+\gamma) - s)} t \right] 
 \leq C \exp\left( -\frac12 3^{ns} t \right). 
\end{equation*}
By applying this inequality for $s':=(d\gamma /(d+\gamma) - \alpha) + c$ and adjusting the constant~$C$, we obtain, for every $0< s \leq s'$,
\begin{equation*}
\P \left[ \frac{\mathcal E_*(\cut_{n})}{CK_0^2}  \geq 3^{-nc} t \right] 
 \leq C \exp\left( -\frac12 3^{ns'} t \right) \leq C \exp\left( -\frac12 3^{ns} t \right). 
\end{equation*}
After combining the previous two inequalities and possibly redefining~$\alpha$, we obtain~\eqref{e.almdoneP4}. 
\end{proof}

\subsection{The proof of Theorem~\ref{t.muconv}}

We next give the proof of Theorem~\ref{t.muconv}, which follows from Proposition~\ref{p.muconv} and some union bounds. 

\begin{proof}[{Proof of Theorem~\ref{t.muconv}}]

According to~\eqref{e.mulessnu} and~\eqref{e.dualatlimit}, we have 
\begin{equation*}
\overline F(p,q) = \overline \mu_0(p,q) = \inf_{p^*,q^*\in\Rd} \left( \overline \mu(q^*,p^*) - p\cdot q^*-p^*\cdot q \right).
\end{equation*}
That is, $\overline F$ is the Legendre-Fenchel transform of $-\overline{\mu}$. It follows that, if $(q^*,p^*) = \nabla \overline F(p,q)$, then
\begin{equation*}
\mathcal E_*(U,q^*,p^*) = \mathcal E(U,p,q) = \left| \mu(U,q^*,p^*) - \overline \mu(q^*,p^*) \right| + \left| \mu_0(U,p,q) - \overline \mu_0(U,p,q)\right|.
\end{equation*}
Thus we see that the error estimate~\eqref{e.muconverge} is already close to the desired conclusion~\eqref{e.muconv0}, we just need to obtain some uniformity in $(p,q)$ and in large scales. 

\smallskip

\emph{Step 1.} We prove~\eqref{e.muconv0}. Fix $\theta \in (0,\beta)$ and $\tau \geq 1$. According to~\eqref{e.mu0cont}, \eqref{e.mucont}, \eqref{e.Fbarcont} and~\eqref{e.nablaFbarcont}, the quantity~$\mathcal E$ is continuous in $(p,q)$: that is, for every $p_1,p_2,q_1,q_2\in\Rd$ and bounded domain $U\subseteq \Rd$,
\begin{multline} 
\label{e.calEcont}
\left| \mathcal E(U,p_1,q_1) - \mathcal E(U,p_2,q_2) \right| \\
\leq C\left( K_0+|p_1|+|q_1|+|p_2|+|q_2| \right) \left( \left| p_1-p_2 \right| + \left| q_1-q_2 \right| \right) \quad \mbox{$\P$--a.s.}
\end{multline}
This yields that 
\begin{multline}
\sup_{p,q\in B_{M3^{n\tau/s}}} \,\frac{\mathcal E(\cu_n,p,q)}{(K_0^2 + |p|^2+|q|^2)3^{-n\theta/s}} \\
\leq C+\, \sup_{p,q\in B_{M3^{n\tau/s}}\cap G_{n,s}}\,\frac{\mathcal E(\cu_n,p,q)}{(K_0^2 + |p|^2+|q|^2)3^{-n\theta/s}} \quad \mbox{$\P$--a.s.,}
\end{multline}
where we have set
\begin{equation*}
G_{n,s}:= 3^{-n\theta/s}\Zd \subseteq \Rd.
\end{equation*}
Observe that the number of elements in the set $B_{M3^{n\tau/s}}\cap G_{n,s}$ is 
\begin{equation*}
\left| B_{M3^{n\theta/s}}\cap G_{n,s} \right| = C M^{d} 3^{nd(\theta+\tau)/s}.
\end{equation*}
By a union bound and Proposition~\ref{p.muconv} applied with exponent $\theta + \frac12 (\beta-\theta)$, we have, for every $n\in\N$,  $t\geq 1$ and $s \geq s_0(d,\Lambda,\beta,\theta) \geq1$,
\begin{align*}
\lefteqn{\P \left[ \sup_{p,q\in B_{M3^{n\tau/s}}} \ \sup_{y\in B_R}\frac{\mathcal E(y+\cu_n,p,q)}{(K_0^2 + |p|^2+|q|^2)3^{-n\theta/s}} \geq Ct \right]} \qquad & \\
& \leq \left| B_{M3^{n\theta/s}} \cap G_{n,s} \right|^2 \cdot \sup_{p,q\in \Rd} \P\left[ \sup_{y\in B_R} \frac{\mathcal E(y+\cu_n,p,q)}{(K_0^2 + |p|^2+|q|^2)3^{-n\theta/s}} \geq t \right]  \\
& \leq C R^d  M^{2d} 3^{2nd(\theta+\tau)/s} \left( 3^{n(\beta-\theta)/2s} t \right)^{-s} \\
& = CR^dM^{2d} 3^{2nd(\theta+\tau)/s - n(\beta-\theta)/2} t^{-s}. 
\end{align*}
By making $s_0(d,\Lambda,\beta,\theta,\tau)$ larger,  we may assume that $2d(\theta+\tau)/s \leq \frac12 (\beta-\theta)$ and then we get, for every~$n\in\N$,~$t\geq 1$ and~$s \geq s_0(d,\Lambda,\beta,\theta,\tau)$,
\begin{equation*}
\P \left[ \sup_{p,q\in B_{M3^{n\tau/s}}} \ \sup_{y\in B_R}\frac{\mathcal E(y+\cu_n,p,q)}{(K_0^2 + |p|^2+|q|^2)3^{-n\theta/s}} \geq Ct \right] \leq CR^dM^{2d} t^{-s}.
\end{equation*}
This is~\eqref{e.muconv0}.

\smallskip

\emph{Step 2.} We prove~\eqref{e.muconvboom0} under assumption~(P4). The argument is almost the same as in Step~1, only easier. Following the argument there, the result of the first union bound can be improved by using~\eqref{e.muconvboom} rather than~\eqref{e.muconverge}: we get, for every $s\in (0,d\gamma/(d+\gamma))$, $R\geq 1$, $n\in\N$ and $t\geq 1$,
\begin{align*}
\lefteqn{\P \left[ \sup_{p,q\in B_{M3^{n}}} \,\sup_{y\in B_{R3^n}}\frac{\mathcal E(y+\cu_n,p,q)}{(K_0^2 + |p|^2+|q|^2)3^{-n\alpha}} \geq Ct \right]} \qquad & \\
& \leq \left| B_{M3^{n}} \cap G_{n} \right|^2 \cdot \sup_{p,q\in \Rd} \P\left[ \sup_{y\in B_{R3^n}} \frac{\mathcal E(y+\cu_n,p,q)}{(K_0^2 + |p|^2+|q|^2)3^{-n\alpha_s}} \geq t \right]  \\
& \leq C R^d M^{2d} 3^{3nd} \exp\left( -3^{ns} t \right),
\end{align*}
where we have set $\alpha_s:= \min\left\{ \alpha, \, d\gamma/(d+\gamma) -s \right\}$ and $G_n:= 3^{-n\alpha_s/2} (\Zd\times\Zd)$. For convenience we may assume that $\alpha_s\leq 1$  so that $\left| B_{M3^{n}} \cap G_{n} \right| \leq 3^{3nd}$. Finally, we may absorb the factor~$3^{3nd}$ into the constant $C$ at the cost of slightly shrinking the exponent~$s$. Since $\alpha_s \geq \alpha( d\gamma/(d+\gamma) -s )$ for some $\alpha(d,\Lambda,\beta)>0$, we thereby obtain the result. 
\end{proof}

\subsection{Construction of the homogenized coefficients \texorpdfstring{$\overline \a$}{}}
\label{s.variationalbar}
We conclude this section by verifying that $\overline F$ is the representative of a Lipschitz, uniformly monotone vector field~$\overline\a$.

\begin{proposition}
\label{p.variationalbar}
There exists a Lipschitz, uniformly monotone vector field $\overline \a$ which is variationally represented by $\overline F$ and satisfies, for some $C(\Lambda)\geq 1$,
\begin{equation} \label{e.abarprops}
\left\{ \begin{aligned} 
& \left| \overline\a(0) \right| \leq C K_0,  \\
& \left| \overline\a(p_1) -\overline \a(p_2) \right| \leq 4\Lambda \left| p_1-p_2 \right|, \\
& \left(\overline \a(p_1) -\overline \a(p_2) \right) \cdot(p_1-p_2) \geq \frac1{4\Lambda}\left|p_1-p_2\right|^2.
\end{aligned} \right.
\end{equation}
\end{proposition}
\begin{proof}
According to Lemma~\ref{l.werepresent}, to prove the existence of a vector field~$\overline \a$ represented by $\overline F$, it suffices to show that, for each $p\in\Rd$,
\begin{equation} \label{e.dualitywts}
\inf_{q\in\Rd} (\overline F(p,q) - p\cdot q) = 0. 
\end{equation}
Fix $p_0\in\Rd$. We first claim that there exists unique $q,q^*\in\Rd$ such that 
\begin{equation} \label{e.qqsidenty}
\nabla \overline F(p_0,q) = (q^*,p_0). 
\end{equation}
To see this, observe that $q\mapsto D_q F(p_0,q)$ is bijective on $\Rd$ by~\eqref{e.Fbarbound} and the uniform convexity of $\overline F$, where $D_q\overline F$ denotes the gradient of $\overline F$ in the variable~$q$. Thus there exists a unique $q\in\Rd$ for which $D_qF(p_0,q) = p_0$. We may then identify~$q^* = D_pF(p_0,q)$. 

\smallskip

According to Proposition~\ref{p.yesvariational}, we have
\begin{multline*}
0 = \inf \bigg\{  \fint_{\cu_n} \left( F(p_0+\nabla w(x),\mathbf{f} (x),x) - (p_0+\nabla w(x))\cdot \mathbf{f}(x) \right)\,dx \\  
:\, w\in H^1_0(\cu_n),\, \mathbf{f} \in \Ls(\cu_n) \bigg\}.
\end{multline*}
It is immediate from this expression that 
\begin{equation*}
\mu(\cu_n,q^*,p_0)+q^*\cdot p_0 \leq 0 \leq \mu_0(\cu_n,p_0,q) -p_0\cdot q.
\end{equation*}
Passing to the limit $n\to \infty$ yields
\begin{equation}\label{e.point1}
\overline \mu(q^*,p_0)+q^*\cdot p_0 \leq 0 \leq \overline \mu_0(p_0,q) -p_0\cdot q.
\end{equation}
In the notation of Proposition~\ref{p.muconv}, the identify~\eqref{e.qqsidenty} asserts that 
\begin{equation*} \label{}
(p_0,q) = \left( \overline P(p_0,q^*) ,\overline Q(p_0,q^*) \right). 
\end{equation*}
Thus according to~\eqref{e.dualatlimit},
\begin{equation} \label{e.point2}
\overline F(p_0,q) = \overline \mu_0(p_0,q) = \overline \mu(q^*,p_0) + (q+q^*)\cdot p_0.
\end{equation}
Combining~\eqref{e.point1} and~\eqref{e.point2} yields
\begin{equation*}
\overline F(p_0,q) = \overline \mu_0(p_0,q) = p_0\cdot q.
\end{equation*}
This gives~\eqref{e.dualitywts}, as the other inequality was proved already in~\eqref{e.Fuppq}.

\smallskip

To conclude, we notice that the second and third lines of~\eqref{e.abarprops} are immediate consequences of the second assertion of~Lemma~\ref{l.werepresent} and~\eqref{e.FbarUC}. The first line of~\eqref{e.abarprops} follows from~\eqref{e.Fbarbound} and Lemma~\ref{l.K0stuff}.
\end{proof}

\section{The proof of Proposition~\ref{p.blackbox}}
\label{s.blackbox}

In this section, we show that the $L^2$ difference between the solutions of the Dirichlet problem (on a given macroscopic domain) for the heterogeneous and homogeneous equations is controlled by the convergence rate of $\mu$ to its limit $\overline\mu$ on mesoscopic subdomains. The argument is a generalization of the one for~\cite[Proposition 4.1]{AS}, although here the proof is much simpler due to the fact that the variational formulation we use is very convenient to work with -- even compared to the more classical variational formulation (for equations which admit one). In particular, since the variational problem for $\overline \J$ is always null, we have only to prove the analogue of the easier of the two energy estimates from~\cite[Appendix A]{AS}.

\smallskip

For the rest of the section, we fix a bounded Lipschitz domain $U_0\subseteq \Rd$, $m,n\in\N$ and set $U:=3^{n+m}U_0$. We also fix $F\in \Omega$, $\delta > 0$, $f\in W^{1,2+\delta}(U)$ and let $M$ be defined by~\eqref{e.M}. We denote by $(u,\g)$ and $(\uhom, \ghom)$ the unique elements of $(f+H^1_0(U)) \times \Ls(U)$ such that
\begin{equation}
\label{e.def.g}
0 = \int_{U}
\left( F\left( \nabla u(x),\g(x) ,x\right) -\nabla u(x) \cdot \g(x)\right) \,dx 
\end{equation}
and
\begin{equation}
\label{e.def.ghom}
0 = \int_{U}
\left( \overline F\left( \nabla \uhom(x),\ghom(x) ,x\right) -\nabla \uhom(x) \cdot \ghom(x)\right) \,dx.
\end{equation}
In particular, $u$ and $\uhom$ satisfy~\eqref{e.Dirmins}. Note that for $u, \uhom \in f + H^1_0(U)$, we can use the substitutions
\begin{align*}
\fint_U\nabla u(x) \cdot \g(x) \,dx & = \fint_U\nabla f(x) \cdot \g(x) \,dx, \\ 
\fint_U\nabla  \uhom(x) \cdot  \ghom(x) \, dx & = \fint_U\nabla f(x) \cdot  \ghom(x) \, dx,
\end{align*}
which reveal the uniform convexity of the minimization problems.
As in the statement of Proposition~\ref{p.blackbox}, we consider~$n+m$ to be the macroscopic (triadic) scale and~$n$ to denote a mesoscopic scale.

\smallskip

The convention for constants in this section is as follows: $\alpha>0$ is an exponent depending on $(d,\Lambda,\beta,\delta)$, while the constants $C\geq 1$ and $0<c\leq 1$ depend on $(d,\Lambda,U_0,\beta,C_3,\delta)$. We allow each of these parameters to vary in each occurrence. 

\smallskip

We prove Proposition~\ref{p.blackbox} by an energy comparison argument.  The idea is to remove the microscopic fluctuations from $(u,\g)$ by a mesoscopic spatial average, thereby obtaining $(\tilde u,\tilde \g)$. The main step is to show that the $\overline F$--energy of~$(\tilde u,\tilde \g)$ is not much larger than the $F$--energy of~$(u,\g)$. This is stated precisely below in~\eqref{e.easyway}. Since by construction, we will also have that $\g - \tilde\g \in \Lso(U)$,
we will be able to deduce that
\begin{multline*} \label{}
0 = \fint_U \left( F(\nabla u(x),\g(x),x)-\nabla f(x) \cdot \g(x) \right)\,dx \\
 = \fint_U \left( \overline F(\nabla  \tilde u(x), \tilde \g(x))-\nabla  f(x) \cdot  \tilde \g(x) \right)\,dx \ - \ \mbox{a small error.}
\end{multline*}
This allows us to conclude that $\nabla \tilde u$ and $\nabla \uhom$ are close in $L^2(U)$, as are $\tilde \g$ and $\ghom$. By construction, $u$ and $\tilde u$ are also close in $L^2(U)$ since the latter is the spatial average of the former, which belongs to $W^{1,2+\delta}$ for some $\delta > 0$ by the Meyers estimate. Hence, by the triangle inequality, we deduce that~$u$ and~$\uhom$ are close in $L^2(U)$. 

\subsection{Construction of the approximating pair}
To prove Proposition~\ref{p.blackbox}, it suffices to consider the case in which $m \ge 5$ and $n+3 \leq l < m+n$, since we can then recover the general case by adjusting the constant $C$ in \eqref{e.exhale}. The integer~$l$ is used to measure an additional mesoscopic scale that gives the thickness of a boundary strip to be removed from $U$ in the construction. 
Note that 
\begin{equation} \label{e.sizeU}
c 3^{d(n+m)} \leq |U| \leq C3^{d(n+m)}.
\end{equation}
We define Lipschitz subdomains $V^\circ \subseteq V \subseteq W \subseteq U$ by 
\begin{equation*} \label{}
 \begin{aligned}
& V^\circ : = \bigcup \left\{ z+ \cu_n \,:\, z\in 3^n\Zd,\ z+\cu_{l+4} \subseteq U \right\},\\ &V:= \bigcup \left\{ z+ \cu_n\,:\, z\in 3^n\Zd,\ z+\cu_{l+2} \subseteq U \right\},  \\&  W:=  \bigcup \left\{ z+ \cu_n\,:\, z\in 3^n\Zd,\ z+\cu_{n+3} \subseteq U \right\}.
\end{aligned} 
\end{equation*}
Since $U_0$ is Lipschitz, we have
\begin{equation} \label{e.mesobndry}
\left| U \setminus V^\circ \right| \leq C 3^{l-(n+m)} |U|.
\end{equation}
Since $n+3\leq l$, we have
\begin{equation}
\label{e.VWfit}
V \subseteq \bigcup_{z\in 3^n\Zd\cap W} (z+\cu_{n}).
\end{equation}
Observe that $\dist(V^\circ,\partial V) \geq 3^l$ and select a cutoff function $\eta \in C^\infty_c(U)$ satisfying
\begin{equation} \label{e.etacutoff}
0 \leq \eta \leq 1, \quad \eta \equiv 1 \ \mbox{on} \  V^\circ + \cu_n, \quad \eta \equiv 0 \ \mbox{on} \ U\setminus V, \quad \sup_{x\in U} \left| \nabla \eta(x) \right| \leq C3^{-l}.
\end{equation}

\smallskip

We next construct the approximating pair $(\widetilde u,\widetilde \g) \in (f+H^1_0(U)) \times \Ls(U)$ by modifying $(u,\g)$. Define the mesoscopic spatial averages of $(u,\g)$ by 
\begin{equation} \label{e.sa}
u_{\mathrm{sa}} (y):= \fint_{y+\cu_n} u(x)\,dx \quad \mbox{and} \quad \g_{\mathrm{sa}} (y):= \fint_{y+\cu_n} \g(x)\,dx, \quad y\in V.
\end{equation}
Observe that $(u_{\mathrm{sa}} , \g_{\mathrm{sa}}) \in H^1(V)\times \Ls(V)$. We get our approximate pair by smoothly interpolating between $(u_{\mathrm{sa}} , \g_{\mathrm{sa}})$ and $(u,\g)$ near the boundary of $V$:
\begin{equation*} \label{}
\left\{ \begin{aligned}
& \widetilde u(x) := \eta(x)u_{\mathrm{sa}} (x)+ (1-\eta(x)) u(x),  \\
&  \tilde\g(x) :=\eta(x)\g_{\mathrm{sa}}(x) + (1-\eta(x)) \g(x) - \nabla h(x),
\end{aligned} \right.
\end{equation*}
where $h\in H^1(U)$ is the (unique, up to a constant) solution of the Neumann problem 
\begin{equation} \label{e.h}
\left\{ 
\begin{aligned}
& \Delta h = \nabla \cdot \left( \eta\g_{\mathrm{sa}} + (1-\eta) \g \right)  & \mbox{in} & \ U, \\
& \partial_\nu h = 0 & \mbox{on} & \ \partial U.
\end{aligned}
\right.
\end{equation}
Observe that we have $(\tilde u,\tilde \g) \in (f+H^1_0(U)) \times \Ls(U)$ and $\tilde \g - \g \in \Lso(U)$.

\smallskip

We conclude this subsection by giving some $L^p$ estimates needed below. First, according to the Meyers estimate (Proposition~\ref{p.meyersbndry}), there exists an exponent~$\delta_0(d,\Lambda,\delta) \in (0,\delta]$ such that 
\begin{equation} \label{e.Meyers}
\left( \fint_U \left| \nabla u(x) \right|^{2+\delta_0} \,dx  \right)^{1/(2+\delta_0)} + \left( \fint_U \left| \nabla \uhom (x) \right|^{2+\delta_0} \,dx  \right)^{1/(2+\delta_0)} \leq CM.
\end{equation}
Note that $\g(x) = \a(\nabla u(x),x)$ and $\ghom = \overline \a(\nabla \uhom)$ by Lemma~\ref{l.werepresent} and Proposition~\ref{p.yesvariational}. By~\eqref{e.UC},~(P1), and~\eqref{e.abarprops}, we have that
\begin{equation*} \label{}
\left| \a(p,x) \right| \leq C\left( K_0 + |p| \right) \quad \mbox{and} \quad \left| \overline\a(p) \right| \leq C\left( K_0 + |p| \right).
\end{equation*}
We deduce that, a.e.~in~$U$,
\begin{equation} \label{e.fluxbnds}
|\g| \leq C\left( K_0+ |\nabla u| \right) \quad \mbox{and} \quad |\ghom| \leq C\left( K_0+ |\nabla \uhom| \right)
\end{equation}
and thus 
\begin{equation} \label{e.Meyersg}
\left( \fint_U \left| \g(x) \right|^{2+\delta_0} \,dx  \right)^{1/(2+\delta_0)} + \left( \fint_U \left| \ghom (x) \right|^{2+\delta_0} \,dx  \right)^{1/(2+\delta_0)} \leq CM.
\end{equation}

\smallskip

The inequalities above, together with the H\"older inequality,  give a bound on the energy density of $(u,\g)$ and $(\uhom,\ghom)$ in the boundary strip $U\setminus V^\circ$. Indeed, we have
\begin{multline}\label{e.strips}
\frac{1}{|U|}\int_{U\setminus V^\circ} \left( \left| \nabla u(x)\right|^2 + \left| \nabla \uhom(x)\right|^2   + \left| \g(x)\right|^2 + \left| \ghom(x)\right|^2 \right)\,dx \\
\leq CM^2 \left( \frac{|U\setminus V^\circ|}{|U|} \right)^{\delta_0/(2+\delta_0)} \leq CM^2 3^{-\alpha(n+m-l)}.
\end{multline}
They also imply, by the Cauchy-Schwarz inequality, that for every $y\in V$,
\begin{equation} \label{e.Linftybymeso}
\left| \nabla u_{\mathrm{sa}} (y)\right|^2 \leq \left( \fint_{y+\cu_n}\left| \nabla u(x) \right|\,dx \right)^2 \leq \frac{|U|}{|\cu_n|}  \fint_{U}\left| \nabla u(x) \right|^2\,dx \leq C3^{dm} M^2,
\end{equation}
and similarly, using \eqref{e.Meyersg},
\begin{equation} \label{e.Linftybymeso2}
\left| \g_{\mathrm{sa}} (y)\right| \leq  C3^{dm/2}M.
\end{equation}

\subsection{The energy estimate}

As we will see, Proposition~\ref{p.blackbox} is essentially a consequence of the uniform convexity of~$\overline F$ and the following estimate, which we prove in this subsection: 
\begin{multline}\label{e.easyway}
\fint_U \overline F\left(\nabla \tilde u(x),\tilde \g(x) \right)\,dx - \fint_U F\left( \nabla u(x),\g(x),x  \right)\,dx \\ \leq  CM^2 \left( \fint_{U} \, \left( \sup_{p,q \in B_{CM3^{dm/2}}} \,\frac{\mathcal E(x+\cu_n,p,q) }{K_0^2+|p|^2+|q|^2}  \right)^{\frac{2+\delta_0}{\delta_0}}\,dx  \right)^{\frac{\delta_0}{2+\delta_0}} + CM^2 3^{-\alpha(n+m-l)}.
\end{multline}
Here $\delta_0(d,\Lambda,\delta)>0$ is defined in the previous subsection and given by the Meyers estimate. 

\smallskip

We now present the derivation of the claimed energy estimate.

\begin{proof}[{Proof of~\eqref{e.easyway}}]
We break the argument into several steps. 

\smallskip

\emph{Step 1.} We begin with the main part of the argument for~\eqref{e.easyway}. For $x\in V$, let $(p^*(x),q^*(x))$ denote the dual pair (via $\overline F$) to $(\nabla u_{\mathrm{sa}}(x), \g_{\mathrm{sa}}(x) - \nabla h(x))$, that is:
\begin{equation*}
\left( p^*(x) , q^*(x) \right) := \nabla \overline F \left( \nabla u_{\mathrm{sa}}(x) , \g_{\mathrm{sa}}(x) -\nabla h(x)\right). 
\end{equation*}
Using $(u,\g)$ as a minimizer candidate for $\mu(y+\cu_n,p^*(y),q^*(y))$, we have 
\begin{align*}
\lefteqn{  \mu(y+\cu_n,p^*(y),q^*(y)) }\qquad & \\
& \leq \fint_{y+\cu_n} \left( F \left( \nabla u(x),\g(x),x \right) - q^*(y) \cdot \nabla u(x) - p^*(y) \cdot \g(x) \right)\,dx    \\
& = \fint_{y+\cu_n}  F \left( \nabla u(x),\g(x),x \right)\,dx - \left( q^*(y),p^*(y) \right) \cdot  \left( \nabla u_{\mathrm{sa}}(y) , \g_{\mathrm{sa}}(y)  \right).
\end{align*}
By~\eqref{e.dualatlimit}, 
\begin{multline*}
\left( q^*(y),p^*(y) \right) \cdot  \left( \nabla u_{\mathrm{sa}}(y) , \g_{\mathrm{sa}}(y) \right) = \overline F\left( \nabla u_{\mathrm{sa}}(y) , \g_{\mathrm{sa}}(y) -\nabla h(y)\right) \\
+ p^*(y)\cdot \nabla h(y) - \overline \mu(p^*(y),q^*(y)) 
\end{multline*}
and therefore we deduce that, for every $y\in V$,
\begin{align}\label{e.easyheart}
\lefteqn{ \overline F\left( \nabla u_{\mathrm{sa}}(y) , \g_{\mathrm{sa}}(y)  - \nabla h(y) \right) - \fint_{y+\cu_n}  F \left( \nabla u(x),\g(x),x \right)\,dx } \qquad & \\
 & \leq - p^*(y) \cdot \nabla h(y) + \overline \mu(p^*(y),q^*(y)) - \mu(y+\cu_n,p^*(y),q^*(y)) \notag \\
 & \leq   - p^*(y) \cdot \nabla h(y) + \mathcal E(y+\cu_n,p^*(y),q^*(y)).\notag
\end{align}
Integrating this over $V^\circ$, we obtain
\begin{multline*}
\int_{V^\circ} \overline F\left( \nabla u_{\mathrm{sa}}(x) , \g_{\mathrm{sa}}(x)  - \nabla h(x)\right)\, dx - \int_{V^\circ} \fint_{x+\cu_n}  F \left( \nabla u(y),\g(y),y \right)\,dy\,dx \\
 \leq \int_{V^\circ} \left| p^*(x) \right| \left| \nabla h(x) \right| \, dx+ \int_{V^\circ} \mathcal E(x+\cu_n,p^*(x),q^*(x)) \,dx.
\end{multline*}
This is already quite close to~\eqref{e.easyway}. The rest of the argument is mainly concerned with using~\eqref{e.Meyers} to show that the contributions to the energy in the mesoscopic boundary strip~$U\setminus V^\circ$ are negligible, to bound~$|\nabla h|$ and to put the last term involving~$\mathcal E(x+\cu_n,p^*(x),q^*(x))$ into a form resembling the right side of~\eqref{e.easyway}.

\smallskip

In fact, to complete the proof of~\eqref{e.easyway}, it suffices to verify the following three inequalities:
\begin{multline} \label{e.boundarycleanup}
\left| \fint_U  F\left(\nabla  u(x), \g(x),x\right)\,dx- \frac{1}{|U|}\int_{V^\circ} \fint_{x+\cu_n}  F \left( \nabla u(y),\g(y),y \right)\,dy\,dx \right| \\ +  \left| \fint_U \overline F\left(\nabla \tilde u(x),\tilde \g(x)\right)\,dx - \frac{1}{|U|}\int_{V^\circ} \overline F\left( \nabla u_{\mathrm{sa}}(x) , \g_{\mathrm{sa}}(x)  - \nabla h(x) \right)\,dx \right|\\
 \leq C M^2 3^{-\alpha(n+m-l)},
\end{multline}
\begin{equation} \label{e.boundnablah}
\fint_U \left| p^*(x) \right| \left| \nabla h(x) \right| \, dx \leq CM^2 3^{-\alpha(n+m-l)},
\end{equation}
and
\begin{multline}\label{e.funnyhaha}
\frac{1}{|U|} \int_{V^\circ} \mathcal E(x+\cu_n,p^*(x),q^*(x)) \,dx \\
\leq CM^2 \left( \fint_{U} \, \left( \sup_{p,q \in B_{CM3^{dm/2}}} \frac{\mathcal E(x+\cu_n,p,q) }{K_0^2+|p|^2+|q|^2}    \, \right)^{\frac{2+\delta}{\delta}}\,dx  \right)^{\frac\delta{2+\delta}}.
\end{multline}

\smallskip

\emph{Step 2.} 
We next show that 
\begin{equation} \label{e.poin}
\fint_V \left( 3^{-n}\left| u_{\mathrm{sa}}(x) - u(x)\right|+ \left|  \nabla u_{\mathrm{sa}}(x)\right| + \left| \g_{\mathrm{sa}}(x)\right|  \right)^{2+\delta_0} \,dx  \leq C M^{2+\delta_0}.
\end{equation}
We first estimate the second term in the integrand using~\eqref{e.Meyers}: 
\begin{multline*}
\fint_V \left| \nabla u_{\mathrm{sa}}(x) \right|^{2+\delta_0} \,dx  = \fint_V \left| \fint_{x+\cu_n}  \nabla u(y)\,dy\right|^{2+\delta_0} \,dx \leq \fint_V \fint_{x+\cu_n} \left| \nabla u(y)\right|^{2+\delta_0} \,dy\,dx \\
 \leq \frac{|U|}{|V|} \fint_U  \left| \nabla u(x)\right|^{2+\delta_0} \,dx \leq CM^{2+\delta_0}.
\end{multline*}
The third term in the integrand is handled similarly.

To estimate the first term in the integrand, we use the expression
\begin{align*}
\lefteqn{ \fint_{y+\cu_n} \left| u_{\mathrm{sa}}(x) - u(x)\right|^{2+\delta_0} \,dx } \qquad  & \\
& = \fint_{y+\cu_n} \left| u(x) - \fint_{x+\cu_n} u(z)\,dz \right|^{2+\delta_0} \,dx \\
& \begin{multlined}[.8\textwidth]
\leq C\fint_{y+\cu_n} \left| u(x) - \fint_{y+\cu_n} u(z)\,dz \right|^{2+\delta_0} \,dx \\
 + C\fint_{y+\cu_n} \left| \fint_{x+\cu_n} u(z)\,dz - \fint_{y+\cu_n} u(z)\,dz \right|^{2+\delta_0} \,dx.\end{multlined}
\end{align*}
We use the Poincar\'e inequality to bound the first term on the right side:
\begin{equation*} \label{}
\fint_{y+\cu_n} \left| u(x) - \fint_{y+\cu_n} u(z)\,dz \right|^{2+\delta_0} \,dx \leq C3^{n(2+\delta_0)} \fint_{y+\cu_n} \left| \nabla u(x) \right|^{2+\delta_0}\,dx,
\end{equation*}
and to bound the second term on the right side, we compute, for all $y\in V$ and $x\in y+\cu_n$,
\begin{align*}
\left| \fint_{x+\cu_n} u(z)\,dz - \fint_{y+\cu_n} u(z)\,dz \right| & = \left| \fint_{\cu_n} \int_0^1 (x-y) \cdot \nabla u(tx+(1-t)y+z)\,dt\,dz \right| \\
& \leq C3^{n} \fint_{y+\cu_{n+1}} \left| \nabla u(z) \right| \, dz.
\end{align*}
Assembling these yields
\begin{equation*} \label{}
\fint_{y+\cu_n} \left| u_{\mathrm{sa}}(x) - u(x)\right|^{2+\delta_0} \,dx \leq C3^{n(2+\delta_0)} \fint_{y+\cu_{n+2}} \left| \nabla u(x) \right|^{2+\delta_0}\,dx
\end{equation*}
and then integrating over $y\in V$ yields the desired estimate of the first term in the integrand in~\eqref{e.poin}. This completes the proof of~\eqref{e.poin}. 

\smallskip

\emph{Step 3.}  We use~\eqref{e.Meyers} and~\eqref{e.poin} to obtain
\begin{equation} \label{e.Meyerstilde}
\fint_U \left( \left| \nabla \tilde u(x) \right|^{2+\delta_0} + \left| \tilde\g(x)\right|^{2+\delta_0} \right)\,dx \leq CM^{2+\delta_0}.
\end{equation}
Differentiating the expression for $\tilde u$, we get
\begin{equation*} \label{}
\nabla \tilde u(x) = \nabla \eta(x) \left( u_{\mathrm{sa}}(x)-u(x) \right) + \eta(x) \left( \nabla u_{\mathrm{sa}}(x)-\nabla u(x)\right) + \nabla u(x).
\end{equation*}
Thus the desired estimate for $\nabla \tilde u$ follows from~\eqref{e.Meyers},~\eqref{e.etacutoff} (recall that $n\leq l$) and~\eqref{e.poin}. The estimate for $\tilde \g$ is similar, but we have the extra term~$\nabla h$. To estimate this, we note that 
\begin{equation*} \label{}
\nabla \cdot\left( \eta \g_{\mathrm{sa}} + (1-\eta) \g \right) = \nabla \cdot \left( (1-\eta) \left( \g_{\mathrm{sa}} - \g \right) \right)
\end{equation*}
and apply for instance~\cite[Theorem 1.2]{GS}, in view of~\eqref{e.poin}, to get
\begin{equation} \label{e.nablaheasy}
\fint_U \left| \nabla h(x)  \right|^{2+\delta_0} \, dx \leq C \fint_U \left|\g_{\mathrm{sa}} (x)- \g(x) \right|^{2+\delta_0} \,dx \leq C M^{2+\delta_0}.
\end{equation}
This completes the proof of~\eqref{e.Meyerstilde}. For future reference we observe also that, by~\eqref{e.Fbarnablabnd},~\eqref{e.poin} and~\eqref{e.nablaheasy}, we have 
\begin{equation} \label{e.popsicle}
\fint_{V} \left( \left| p^*(x) \right| + \left| q^*(x) \right| \right)^{2+\delta_0}  \,dx \leq CM^{2+\delta_0}.
\end{equation}

\emph{Step 4.} We prove~\eqref{e.boundnablah}. First we notice that we can bound $\nabla h$ more brutally (compared to~\eqref{e.nablaheasy}) in $L^2$ by exploiting that $(1-\eta)$ vanishes in $V$ and using the H\"older inequality. Using~\cite[Theorem 1.2]{GS} with exponent~$2$ rather than~$2+\delta_0$ and~\eqref{e.mesobndry}, we get
\begin{align} \label{e.nablahbrutality}
\fint_U \left| \nabla h(x)  \right|^{2} \, dx & \leq C \frac{1}{|U|} \int_{U\setminus V} \left|\g_{\mathrm{sa}} (x)- \g(x) \right|^{2} \,dx \\
& \leq  C \left(\frac{|U\setminus V|}{|U|} \right)^{\delta_0/(2+\delta_0)} \left( \fint_U \left|\g_{\mathrm{sa}} (x)- \g(x) \right|^{2+\delta} \,dx \right)^{2/(2+\delta_0)} \notag\\
& \leq C 3^{ -\alpha(n+m-l)} M^{2}. \notag
\end{align}
The  previous inequality,~\eqref{e.popsicle} and the H\"older inequality yield~\eqref{e.boundnablah}.

\smallskip

\emph{Step 5.} We give the proof of~\eqref{e.funnyhaha}, which follows from an application of the H\"older inequality and~\eqref{e.popsicle}:
\begin{align*}
\lefteqn{ \frac1{|U|} \int_{V^\circ} \mathcal E(x+\cu_n,p^*(x),q^*(x)) \,dx } \qquad & \\
& \leq  \frac1{|U|}  \int_{V^\circ}\frac{\mathcal E(x+\cu_n,p^*(x),q^*(x))}{K_0^2+|p^*(x)|^2+|q^*(x)|^2} \left( K_0^2+|p^*(x)|^2+|q^*(x)|^2 \right)\,dx \\
& \leq C  \frac1{|U|}  \left( \int_{V^\circ} \left( K_0^2+|p^*(x)|^2+|q^*(x)|^2 \right)^{\frac{2+\delta}{2}}\,dx \right)^{\frac2{2+\delta}} \\
& \qquad \qquad \times \left(  \int_{V^\circ}\left( \frac{\mathcal E(x+\cu_n,p^*(x),q^*(x))}{K_0^2+|p^*(x)|^2+|q^*(x)|^2} \right)^{\frac{2+\delta}{\delta}} \,dx \right)^{\frac\delta{2+\delta}} \\
& \leq CM^2 \left( \fint_{U} \, \left( \sup_{p,q \in B_{CM3^{dm/2}}} \frac{\mathcal E(x+\cu_n,p,q) }{K_0^2+|p|^2+|q|^2}    \, \right)^{\frac{2+\delta}{\delta}}\,dx  \right)^{\frac\delta{2+\delta}}.
\end{align*}
In the last line, we introduced the supremum with the help of~\eqref{e.Linftybymeso} and~\eqref{e.Linftybymeso2}, but we also of course need a similar estimate for $|\nabla h|$. To get the latter, we use the fact that $h$ is harmonic in $V^\circ+\cu_n$ and then apply Cauchy-Schwarz with the help of~\eqref{e.nablaheasy}, similar to the derivation of~\eqref{e.Linftybymeso} and~\eqref{e.Meyersg} above, to obtain, for every $y\in V^\circ$,
\begin{equation*}
\left| \nabla h(y) \right|^2  \leq C\left( \fint_{y+\cu_n} \left| \nabla h(x) \right| \,dx \right)^2 \leq CM^2 3^{md}.
\end{equation*}
This completes the proof of~\eqref{e.funnyhaha}.

\smallskip

\emph{Step 6.} We give the proof of~\eqref{e.boundarycleanup}. To estimate the second term on the left side of~\eqref{e.boundarycleanup}, we use the identity
\begin{multline*}
\fint_U \overline F\left(\nabla \tilde u(x),\tilde \g(x)\right)\,dx - \frac{1}{|U|} \int_{V^\circ} \overline F\left( \nabla u_{\mathrm{sa}}(x) , \g_{\mathrm{sa}}(x) - \nabla h(x) \right)\,dx  \\
 = \frac{1}{|U|} \int_{U\setminus V^\circ} \overline F\left(\nabla \tilde u(x),\tilde \g(x)\right)\,dx
\end{multline*}
and by~(P1),~\eqref{e.Meyerstilde} and the H\"older inequality,
\begin{equation*} \label{}
\frac{1}{|U|} \left| \int_{U\setminus V^\circ} \overline F\left(\nabla \tilde u(x),\tilde \g(x)\right)\,dx \right| 
\leq CM^2\left( \frac{\left| U\setminus V^\circ \right|}{|U|} \right)^{\delta_0/(2+\delta_0)} = CM^23^{-\alpha(n+m-l)}.
\end{equation*}
The first term on the left side of~\eqref{e.boundarycleanup} is handled similarly, using~\eqref{e.Meyers} in place of~\eqref{e.poin} and~\eqref{e.Meyerstilde}. Set 
\begin{equation}
\label{e:def:Vcc}
V^{\circ\circ}:= \{ x\in V^\circ \,:\, x+\cu_n \subseteq V^\circ\},
\end{equation}
and observe that $|U\setminus V^{\circ\circ}| \leq C3^{-(n+m-l)}|U|$.
We obtain
\begin{align*}
& \left| \fint_U  F\left(\nabla  u(x), \g(x),x\right)\,dx - \frac{1}{|U|}\int_{V^\circ} \fint_{x+\cu_n}  F \left( \nabla u(y),\g(y),y \right)\,dy\,dx \right|  \\
&\qquad  \leq  \frac{1}{|V^{\circ}|} \int_{U\setminus V^{\circ\circ}} \left| F\left( \nabla u(x) , \g(x),x \right) \right|\,dx \\
& \qquad \leq  CM^23^{-\alpha(n+m-l)},
\end{align*}
which completes the proof of~\eqref{e.boundarycleanup} and hence of~\eqref{e.easyway}.
\end{proof}

\subsection{The proof of Proposition~\ref{p.blackbox}}

We now derive Proposition~\ref{p.blackbox} as a consequence of~\eqref{e.easyway}.
\begin{proof}[{Proof of Proposition~\ref{p.blackbox}}]
By  construction, $\g-\tilde\g \in \Lso(U)$ and $u,\tilde u \in f+H^1_0(U)$.
This implies that 
\begin{equation*} \label{}
\int_{U} \nabla u(x) \cdot \g(x)\,dx = \int_{U} \nabla \tilde u(x) \cdot \tilde\g(x)\,dx = \int_U \nabla f(x) \cdot \g(x)=\int_U \nabla f(x) \cdot \tilde\g(x).
\end{equation*}
According to the previous line, \eqref{e.def.g} and~\eqref{e.easyway}, 
\begin{multline}\label{e.bbtest1}
\fint_U  \left( \overline F\left( \nabla\tilde u(x),\tilde\g(x)\right) -\nabla f(x)\cdot\tilde\g(x) \right) \,dx \\ \leq \fint_U \left( F\left( \nabla u(x),\g(x),x \right) - \nabla f(x) \cdot \g(x) \right)\,dx+ C\mathcal E'' = C\mathcal E'',
\end{multline}
where $\mathcal E'_{n,m,M}$ is as in~\eqref{e.Eprimestat} with $\rho:= (2+\delta_0) / \delta_0$ and we denote
\begin{equation*}
\mathcal E'':=  M^2\left(\mathcal E'_{n,m,M} +3^{-\alpha(m+n-l)} \right).
\end{equation*}
By~\eqref{e.def.ghom}, Proposition~\ref{p.yesvariational} and the fact that $\uhom \in f+H^1_0(U)$, we have
\begin{align*}
0 
&= \fint_U  \left( \overline F\left( \nabla\uhom(x),\ghom(x)\right) -\nabla \uhom(x)\cdot\ghom(x) \right) \,dx \\
& = \fint_U  \left( \overline F\left( \nabla\uhom(x),\ghom(x)\right) -\nabla f(x)\cdot\ghom(x) \right) \,dx  \\
& = \inf_{(u',\g') \in (f+H^1_0(U))\times \Lso(U)} \,  \fint_U  \left( \overline F\left( \nabla u'(x),\g'(x)\right) -\nabla f(x)\cdot\g'(x) \right) \,dx.
\end{align*}
The functional inside the infimum on the last line is uniformly convex in the variable $(u',\g')\in (f+H^1_0(U))\times \Lso(U)$, and therefore the previous display and~\eqref{e.bbtest1}  imply that
\begin{equation*}
\fint_{U} \left(   \left| \nabla \uhom(x) -\nabla \tilde u(x) \right|^2 +\left| \ghom(x) - \tilde \g(x) \right|^2\right)\,dx \leq C\mathcal E''
\end{equation*}
(c.f.~the proof of Lemma~\ref{l.unifconv}). The Poincar\'e inequality then yields
\begin{equation*}
3^{-2(n+m)} \fint_{U} \left|  \uhom(x) - \tilde u(x) \right|^2 \,dx \leq C\mathcal E''.
\end{equation*}
Using the identity $u-\tilde u = \eta \left( u_{\mathrm{sa}} - u \right)$ and applying~\eqref{e.poin}, we also have 
\begin{equation*}
3^{-2n}\fint_{U} \left|  u(x) - \tilde u(x) \right|^2 \,dx \leq CM^2.
\end{equation*}
The previous two lines and the triangle inequality imply
\begin{equation*}
3^{-2(n+m)} \fint_{U}  \left|  u(x) - \uhom(x) \right|^2 \,dx \leq C\mathcal E'' + CM^2 3^{-2m},
\end{equation*}
which is~\eqref{e.exhale}. This completes the proof of the proposition.
\end{proof}

\smallskip

\appendix

\section{Consequences of the mixing conditions}
\label{s.mixing}

In this section we review some basic consequences of the mixing conditions which are used in several key arguments in the paper. So that we may handle both mixing conditions at once, we assume that $\P$ is $\alpha$-mixing with a rate given by a decreasing function $\phi:[0,\infty) \to [0,1]$: for all Borel subsets $U,V\subseteq \Rd$ and events $A\in \F_U$ and $B\in \F_V$, we have 
\begin{equation}
\label{e.mixing}
\left| \P \left[A \cap B \right] - \P \left[A \right] \P \left[B \right] \right| \leq \phi\left(\dist(U,V) \right).
\end{equation}
We next reformulate this condition equivalently as a statement about the difference between the expectation of a product and the product of expectations. Throughout, we denote the $\P$-essential supremum of a random variable~$|X|$ by
\begin{equation*}
\| X \|_\infty := \inf\left\{ \lambda > 0\,:\, \P \left[ |X| > \lambda \right] = 0 \right\}.
\end{equation*}

\begin{lemma}
\label{l.equivmixing}
Assume that $\P$ satisfies~\eqref{e.mixing}. Fix $M,D>0$. Consider Borel subsets $U_1,\ldots,U_k\subseteq \Rd$ such that $\dist(U_i,U_j) \geq D$ for every $i\neq j$. 
Let $X_1,\ldots,X_k$ be random variables on $\Omega$ such that, for each $i\in \{1,\ldots,k\}$, $X_i$ is $\F_{U_i}$--measurable. Then
\begin{equation*}
\left| \E \left[ \prod_{i=1}^k X_i \right] -  \prod_{i=1}^k \E\left[ X_i \right] \right| \leq 4(k-1) \left( \prod_{i=1}^k \| X_i\|_\infty\right) \phi(D).
\end{equation*}
\end{lemma}
\begin{proof}
It suffices by induction to prove the result for $k=2$. We need to show that
\begin{equation}\label{e.covalphamix}
\cov\left[ X,Y \right]  \leq 4 \| X \|_\infty \|Y\|_{\infty} \phi\left(\dist(U,V)\right),
\end{equation}
provided that $X$ is $\F_U$--measureable and $Y$ is $\F_V$--measureable. To get this, we compute
\begin{multline*}
\cov\left[ X,Y \right] = \E \left[ XY \right] - \E\left[ X \right]\E\left[Y\right] \\
= \int_{-\| X\|_{\infty}}^{\| X\|_{\infty}}\int_{-\| Y\|_{\infty}}^{\| Y\|_{\infty}} \left( \P \left[ X>s\, \mbox{and} \, Y>t \right] -\P \left[ X>s\right]\P \left[ Y>t\right] \right) \, dt \,ds
\end{multline*}
and observe that the integrand is bounded by $\phi(\dist(U,V))$. 
\end{proof}

In the next two lemmas, we put Lemma~\ref{l.equivmixing} in a more convenient form for its application in the proof of Lemma~\ref{l.stochasticupgrade}. The notation here for the cubes is the same as in Section~\ref{ss.cubes}.

\begin{lemma}
\label{l.finitemoments}
Assume that $\P$ satisfies~(P2) and~\eqref{e.mixing}. Fix $n,k\in\N$ and let $X$ be an $\F_{\cut_n}$--measurable random variable satisfying $0 \le X \leq 1$. Let $X_z$ denote the translation of $X$ by $z\in\Zd$. 
Then there exists $C(d,k) \geq 1$ such that, for every $m\in\N$, 
\begin{equation*} \label{}
\E \left[ \left( 3^{-dm}\sum_{z\in 3^n\Zd\cap\, \cut_{n+m}} X_z \right)^{2k} \right]  \leq  C \left(  \max\left\{ \E \left[ X \right], 3^{-dm} \right\} ^{2k} +  \phi(3^n) \right).
\end{equation*}
\end{lemma}
\begin{proof}
We first separate the subcubes into $3^{d}$ distinct groups. We have
\begin{align*} \label{}
\E \left[ \left( \sum_{z\in 3^n\Zd\cap\, \cut_{n+m}} X_z \right)^{2k} \right]^{\frac1{2k}} & = \E \left[ \left( \sum_{y\in \{ -3^n, 0, 3^n\}^d}  \sum_{z\in 3^{n+1}\Zd\cap\, \cut_{n+m}} X_{y+z} \right)^{2k} \right]^{\frac1{2k}} \\
& \leq \sum_{y \in\{ -3^n, 0, 3^n\}^d }  \E \left[ \left( \sum_{z\in 3^{n+1}\Zd\cap\, \cut_{n+m}} X_{y+z} \right)^{2k} \right]^{\frac1{2k}} \\
& = 3^d \,  \E \left[ \left( \sum_{z\in 3^{n+1}\Zd\cap\, \cut_{n+m}} X_{z} \right)^{2k} \right]^{\frac1{2k}}.
\end{align*}
Therefore it suffices to analyze the sum on the right side. The benefit of the previous computation is that each of the distinct cubes in the collection $\{ z+\cu_{n} \,:\, z\in 3^{n+1}\Zd \}$ are separated by a distance of more than~$3^n$. 

\smallskip

We follow the classic method of moments: we expand the sum by writing
\begin{equation*} \label{}
\E \left[ \left( \sum_{z\in 3^{n+1}\Zd\cap\, \cut_{n+m}} X_z \right)^{2k} \right]  
= \sum_{z_1,\ldots,z_{2k} \in 3^{n+1}\Zd\cap\, \cut_{n+m}} \E \left[ X_{z_1} \cdots X_{z_{2k}} \right]
\end{equation*}
and proceed by analyzing each term in the sum according to how many distinct entries it has. We write $(z_{1},\ldots,z_{2k}) \in H_j$ if it contains exactly $j$ distinct entries. For convenience, set $N:= \left| 3^{n+1}\Zd\cap \cut_{n+m} \right|$ and note that $c3^{dm} \leq N \leq C3^{dm}$. 

\smallskip

We next apply Lemma~\ref{l.equivmixing} to each element $(z_{1},\ldots,z_{2k}) \in H_j$ to separate its~$j$ distinct entries and use the crude bound $\E\left[ X^s \right] \leq \E\left[ X \right]$, for $s\geq 1$, which follows from the~$\P$--a.s.~bound $|X|\leq 1$. We get: 
\begin{equation*} \label{}
\E \left[ X_{z_1} \cdots X_{z_{2k}} \right] \leq \E[ X ]^{j} + 4j \phi(3^n) \quad \mbox{for every} \ (z_{1},\ldots,z_{2k}) \in H_j.
\end{equation*}
We next estimate the number of elements in~$H_j$, which is a simple combinatorics exercise. A crude upper bound is 
\begin{equation*} \label{}
\left| H_j \right| \leq  \binom{N}{j}  \, j^{2k}   \leq  \frac{N^{j} j^{2k}}{j!},
\end{equation*}
which we see from the fact that there are $\binom{N}{j}$ different ways to select $j$ distinct elements from a finite set of size~$N$, and then at most $j^{2k}$ ways to make a $2k$-tuple from them. Using Stirling's inequality 
\begin{equation*} \label{}
j! \geq j^j \exp(-j),
\end{equation*}
we obtain
\begin{equation*} \label{}
\left| H_j \right| \leq \exp(j) N^j (2k)^{2k-j}.
\end{equation*}
We deduce that 
\begin{align*} \label{}
 \sum_{z_1,\ldots,z_{2k} \in 3^{n+1}\Zd\cap\, \cut_{n+m}} \E \left[ X_{z_1} \cdots X_{z_{2k}} \right]  & =\sum_{j=1}^{2k} \, \sum_{(z_1,\ldots,z_{2k}) \in H_j} \E \left[ X_{z_1} \cdots X_{z_{2k}} \right] \\
& \leq \sum_{j=1}^{2k} |H_j| \left( \E[ X ]^{j} + 4j\phi(3^n) \right) \\
& \leq  \sum_{j=1}^{2k} \exp(j) N^j (2k)^{2k-j}\left( \E[ X ]^{j} + 4j\phi(3^n) \right).
\end{align*}
Set $\theta:= \max\left\{ 4kN^{-1} \exp(-1),\E \left[ X \right] \right\}$ and observe that 
\begin{align*}
\lefteqn{ \sum_{j=1}^{2k} \exp(j) N^j (2k)^{2k-j}\left( \E[ X ]^{j} + 4j\phi(3^n)  \right) } \qquad \qquad & \\
& \leq \sum_{j=1}^{2k} \exp(j) N^j (2k)^{2k-j} \left( \theta^{j} +4(2k) \phi(3^n) \right) \\
& \leq 2 N^{2k} \exp(2k) \left( \theta^{2k} + 4k\phi(3^n) \right).
\end{align*}
This completes the argument.
\end{proof}

We next give the analogous statement to Lemma~\ref{l.finitemoments} for exponential moments.

\begin{lemma}
\label{l.concentration}
Assume that $\P$ satisfies~(P2) and~\eqref{e.mixing}. Fix $n\in\N$ and let $X$ be an $\F_{\cut_n}$--measurable random variable satisfying $0 \le X\le 1$. Let $X_z$ denote the translation of $X$ by $z\in\Zd$. 
Then there exists $C(d) \geq 1$ such that, for every $m\in\N$ and $t\in [0,1]$, 
\begin{equation*} \label{}
 \log \E \left[ \prod_{z\in 3^n\Zd\cap\, \cut_{n+m}} \exp\left( t X_z\right) \right]  \leq C3^{dm} \left( t\E\left[ X \right] + \phi(3^n) \exp\left( Ct3^{dm} \right) \right).
 \end{equation*}
\end{lemma}
\begin{proof}
As in the previous lemma, we use that the family $\{ z+\cut_n\,:\, z\in 3^{n+1} \Zd \}$ consists of disjoint cubes separated by a distance greater than~$3^n$. Fix $m\in\N$ and $t\ge0$ and compute:
\begin{multline*}
\log \E \left[   \prod_{z\in 3^{n+1}\Zd\cap\, \cut_{n+m}} \exp\left( t X_z \right)   \right] \\
\begin{aligned}
& \leq\log \left( \prod_{z\in 3^{n+1}\Zd\cap\, \cut_{n+m}} \E \left[ \exp(tX_z) \right] + C3^{dm} \exp\left(Ct3^{dm} \right) \phi(3^n) \right) \\
& \leq  \sum_{z\in 3^{n+1}\Zd\cap\, \cut_{n+m}} \log \E \left[ \exp(tX_z) \right] + C3^{dm} \exp\left(Ct3^{dm} \right) \phi(3^n) \\
& = C3^{dm} \log\E \left[ \exp(tX) \right] + C3^{dm} \exp\left(Ct3^{dm} \right) \phi(3^n).
\end{aligned}
\end{multline*}
In the above string of inequalities, we used Lemma~\ref{l.equivmixing} in the first line, the elementary inequality~$\log (s+t) \leq t+ \log s$ (which is valid for every~$s\geq 1$ and~$t\ge0$) in the second line and stationarity in the third line. Next we use H\"older's inequality and stationarity once more to obtain, for every~$t\ge0$,
\begin{align*}
\log \E \left[   \prod_{z\in 3^{n}\Zd\cap\, \cut_{n+m}} \exp\left( t X_z \right)   \right] & = \log \E \left[ \prod_{y\in \{ -3^n,0,3^{n} \}^d} \ \prod_{z\in 3^{n+1}\Zd\cap\, \cut_{n+m}} \exp\left( t X_{y+z} \right)   \right]  \\
& \leq \log \prod_{y\in \{ -3^n,0,3^{n} \}^d} \E \left[ \prod_{z\in 3^{n+1}\Zd\cap\, \cut_{n+m}} \exp\left( 3^d t X_{y+z} \right)   \right]^{3^{-d}} \\
& =  \log \E \left[ \prod_{z\in 3^{n+1}\Zd\cap\, \cut_{n+m}} \exp\left( 3^d t X_{z} \right)   \right].
\end{align*}
Combining the above inequalities yields, for every $m\in\N$ and $t\ge0$, 
\begin{equation*} \label{}
 \log \E \left[ \prod_{z\in 3^n\Zd\cap\, \cut_{n+m}} \exp\left( t X_z\right) \right]  \leq  C3^{dm} \left( \log \E \left[ \exp\left( t3^d X\right) \right] + \phi(3^n) \exp\left( Ct3^{dm} \right) \right).
\end{equation*}
We now use the elementary inequalities
\begin{equation*} \label{}
\left\{ \begin{aligned} 
& \exp(s) \leq 1  + Cs & & \mbox{for every} \ 0\leq s \leq 3^d, \\
& \log(1+s) \leq s  & & \mbox{for every} \ s \geq 0, 
\end{aligned} \right.
\end{equation*}
and the fact that $0 \le X \leq 1$ to obtain, for every $m\in\N$ and $0\le t \leq1$, 
\begin{equation*}
 \log \E \left[ \prod_{z\in 3^n\Zd\cap\, \cut_{n+m}} \exp\left( t X_z\right) \right]  \leq C3^{dm} \left( t\E\left[ X \right] + \phi(3^n) \exp\left( Ct3^{dm} \right) \right).
\end{equation*}
This completes the proof. 
\end{proof}

\section{Basic energy estimates}
\label{s.reg}

In this appendix we record some regularity estimates needed in the paper, in particular the Caccioppoli inequality and both local and global versions of the Meyers estimate.

\smallskip

The statements here are entirely deterministic, so throughout we fix $F\in \Omega$ such that $F$ satisfies the inequality in~(P1), that is,
\begin{align} \label{e.P1again}
 \ \quad \frac1{2\Lambda}\left( |p|^2+|q|^2 \right)- K_0(1+|p|+|q|) & \leq F(p,q,x) \\
& \leq \frac\Lambda2\left(|p|^2+|q|^2\right) + K_0(1+|p|+|q|). \notag
\end{align}
We also fix a bounded Lipschitz domain $U$ and let $\J:H^1(U) \times H^{-1}(U)\to\R$ be the functional defined in~\eqref{e.J}.

\smallskip

We begin with a simple $L^2$ energy estimate. 

\begin{proposition}
\label{p.L2estimate}
Suppose $u,v\in H^1(U)$ and $u^*,v^*\in H^{-1}(U)$ are such that $u-v\in H^1_0(U)$ and
\begin{equation} \label{e.cmptminim}
 \J\left[ u,u^* \right] = \J\left[ v,v^* \right]= 0.
\end{equation}
Then there exists a constant~$C(d,\Lambda,U)\geq 1$ such that
\begin{equation*} \label{}
\|  u-v \|_{H^1(U)} \leq C \left\| u^*-v^* \right\|_{H^{-1}(U)}.
\end{equation*}
\end{proposition}
\begin{proof}
Select $\h \in L^2(U;\Rd)$ to satisfy
\begin{equation}
\label{e.chh}
-\nabla \cdot \h = v^*  \quad \mbox{and} \quad  0 = \fint_{U} \left( F\left( \nabla v(x),\h(x),x \right) - \nabla v(x) \cdot \h(x) \right)\,dx
\end{equation}
and recall that, by Proposition~\ref{p.yesvariational}, 
\begin{equation}
\label{e.chh2}
F\left( \nabla v(x),\h(x),x \right) - \nabla v(x) \cdot \h(x) \quad \mbox{a.e. in} \ U.
\end{equation}
Let $h\in H^1(U)$ denote the (unique up to a constant) solution of 
\begin{equation*}
\left\{ \begin{aligned}
& -\Delta h = u^*-v^* & \mbox{in} & \ U, \\
& \partial_\nu h = |\partial U|^{-1} \langle 1, u^*-v^*\rangle  & \mbox{on} & \ \partial U.
\end{aligned} \right.
\end{equation*}
According to~\cite[Theorem 1.2]{GS}, we have
\begin{equation}\label{e.nablaagain}
\| \nabla h \|_{L^2(U)} \leq C \left\| u^*-v^* \right\|_{H^{-1}(U)}.
\end{equation}
Using~\eqref{e.P1again},~\eqref{e.infFpq} and~\eqref{e.chh2}, we have that 
\begin{align*}
\J\left[ v,u^* \right] & \leq \fint_{U} \left( F\left( \nabla v(x),\h(x)-\nabla h(x) ,x \right) - \nabla v(x) \cdot \left( \h(x)-\nabla h(x) \right) \right) \,dx  \\
& \leq C\fint_U \left| \nabla h(x) \right|^2\,dx \leq  C\left\| u^*-v^* \right\|_{H^{-1}(U)}^2. \notag
\end{align*}
Now we use~\eqref{e.P1again}, the uniform convexity of~$w \mapsto \J\left[ w,u^* \right]$ on $u+H^1_0(U)$, the assumption~$\J\left[ u,u^*\right]=0$ and the Poincar\'e inequality to obtain the lemma. 
\end{proof} 

We next verify the interior Caccioppoli inequality. The variational formulation yields a particularly simple proof.

\begin{proposition}
\label{p.caccioppoli}
Fix a domain $V$ such that $\overline V\subseteq U$. Suppose that $u\in H^1(U)$ and $u^*\in H^{-1}(U)$ satisfy
\begin{equation*} \label{}
\J\left[ u,u^* \right] = 0.
\end{equation*}
Then there exists a constant~$C(d,\Lambda,U,V)\geq 1$ such that
\begin{equation} \label{e.catch}
\| \nabla u \|_{L^2(V)} \leq C\left( K_0+ \| u \|_{L^2(U)} + \| u^* \|_{H^{-1}(U)} \right).
\end{equation}
\end{proposition}
\begin{proof}
By Proposition~\ref{p.L2estimate}, it suffices to consider the case $u^*=0$. Select a solenoidal vector field $\g \in \Ls(U)$ to satisfy
\begin{equation*}
0 = \fint_{U} \left( F\left( \nabla u(x),\g(x),x \right) - \nabla u(x) \cdot \g(x) \right) \,dx.
\end{equation*}
By Proposition~\ref{p.yesvariational},
\begin{equation} \label{e.pinpoint1}
F \left( \nabla u(x),\g(x),x\right)  =  \nabla u(x) \cdot \g(x) \quad \mbox{a.e. in} \ U.
\end{equation}
Set $v:= (1+\eta^2) u$, where $\eta\in C^\infty_c(\Rd)$ is a smooth test function satisfying 
\begin{equation} \label{e.etaagain}
0 \leq \eta \leq 1, \quad \eta \equiv 1 \ \mbox{on} \ \overline V,  \quad \supp \eta \subseteq U,  \quad \sup_{x\in U} \left| \nabla \eta(x) \right| \leq C.
\end{equation}
Using that $\g$ is solenoidal and $u-v\in H^1_0(U)$ as well as~\eqref{e.pinpoint1}, we obtain
\begin{align*}
0 & = \int_U \g(x) \cdot \left( \nabla v(x) - \nabla u(x) \right) \, dx \\
& =  \int_U \g(x) \cdot  \left( \eta^2(x) \nabla u(x) + 2u(x)\eta(x) \nabla \eta(x) \right)\,dx \\
& \geq  \int_U \left( F(\nabla u(x),\g(x),x) \eta^2(x) - C\eta(x)\left| \g(x) \right| |u(x)| \right)\,dx.
\end{align*}
Using that $F$ satisfies~\eqref{e.P1again} and applying Young's inequality, we get
\begin{align*}
\int_U \left( \left| \nabla u(x)\right|^2+\left|\g(x)\right|^2\right)  \eta^2(x) \,dx & \leq CK_0^2 +C \int_U  F(\nabla u(x),\g(x),x) \eta^2(x) \,dx \\
& \leq CK_0^2+C\int_U \eta(x) \left| \g(x) \right| |u(x)| \,dx \\
& \leq CK_0^2 + \frac12\int_U\left|\g(x)\right|^2 \eta^2(x) + C\int_U \left| u(x) \right|^2\,dx.
\end{align*}
A rearrangement yields the proposition. 
\end{proof}

The proof of the Meyers estimate follows the argument of Giaquinta and Giusti~\cite{GG} which requires a version of the Gehring lemma in order to obtain some improvement of integrability. We use the one given in Giaquinta and Modica~\cite[Proposition 5.1]{GiaMod} which can also be found in~\cite[Theorem 6.6]{Giu}.

\begin{lemma}
\label{l.Gehring}
Fix a domain $V$ such that $\overline V\subseteq U$. 
Let $K>0$, $\delta > 0$, $0<m<1$,  $f \in L^{1}(U)$ and $g\in L^{1+\delta}(U)$. Suppose that, for every $r>0$ and $x\in \Rd$ such that $B_{2r}(x) \subseteq U$, we have
\begin{equation*}
\fint_{B_r(x)} \left|f(x) \right| \,dx \leq K \Ll( \left( \fint_{B_{2r}(x)} \left|f(x) \right|^m \,dx\right)^{\frac1m} +\fint_{B_{2r}(x)} |g(x)| \, dx \Rr). 
\end{equation*}
Then there exist~$\delta_0(d,m,K,\delta)\in (0,\delta]$ and a constant $C(d,m,K,U,V)\geq 1$ such that~$f\in L^{1+\delta_0}(V)$ and 
\begin{equation*}
\int_{V} \left|f(x) \right|^{1+\delta_0} \,dx \leq C \left(   \left( \int_{U} \left|f(x) \right|   \,dx \right)^{1+\delta_0} + \int_U \left| g(x) \right|^{1+\delta_0} \,dx \right).
\end{equation*}
\end{lemma}

We next present a version of the interior Meyers estimate.

\begin{proposition}
\label{p.interiorMeyers}
Suppose that $u\in H^1(U)$  satisfies
\begin{equation*} \label{}
\J\left[ u,0 \right] = 0.
\end{equation*}
Fix a domain $V\subseteq\Rd$ such that $\overline V\subseteq U$.  Then there exist $\delta_0(d,\Lambda)>0$ and~$C(d,\Lambda,U,V)\geq 1$ such that
\begin{equation*} \label{}
\| \nabla u \|_{L^{2+\delta_0}(V)} \leq C \left( K_0 +  \| u \|_{L^2(U)} \right).
\end{equation*}
\end{proposition}
\begin{proof}
Fix $x\in U$ and $r>0$ such that $x+2r\cu_0\subseteq U$. By (a properly rescaled version of) Proposition~\ref{p.caccioppoli} and the Poincar\'e-Sobolev inequality (cf.~\cite[Theorem 3.15]{Giu}), 
\begin{align*}
\fint_{x+r\cu_0} \left|  \nabla u(x) \right|^2\,dx & \leq C\left( K_0^2+ r^{-2} \fint_{x+2r\cu_0} \left| u(x) - \fint_{x+2r\cu_0} u(y)\,dy \right|^2\,dx \right)\\
& \leq C\left( K_0^2+ \left(  \fint_{x+2r\cu_0} \left| \nabla u(x) \right|^{\frac{2d}{d+2}} \,dx  \right)^{\frac{d+2}{d}} \right).
\end{align*}
Note that the constant $C$ here depends only on $(d,\Lambda)$. Thus the hypotheses of Lemma~\ref{l.Gehring} hold for the function $(K_0^2+\left| \nabla u\right|^2)$ with $m=d/(d+2)$ and $K=C(d,\Lambda)$, and so an application of the lemma yields the result. 
\end{proof}

The rest of this appendix is concerned with obtaining a version of Meyers' estimate which is valid up to the boundary of a Lipschitz domain. For this we adapt the classical argument which can be found for example in Giusti~\cite{Giu}. We begin with a variant of the Caccioppoli inequality which holds for balls centered in $U$ but which may intersect~$\partial U$.

\begin{proposition}
\label{p.caccioppolibdry}
Suppose that $f\in H^1(U)$ and that $u\in f+H^1_0(U)$ satisfies
\begin{equation*} \label{}
\J\left[ u, 0 \right] = 0.
\end{equation*}
Then there exists a constant~$C(d,\Lambda,U)\geq 1$ such that, for every $z\in U$ and $r>0$ 
\begin{equation} \label{e.catchbdry}
\| \nabla u \|_{L^2(B_r(z)\cap U)} \leq C\left( K_0+ r^{-1} \| u - f \|_{L^2(B_{2r}(z)\cap U)} +  \| \nabla f \|_{L^2(B_{2r}(z) \cap U)} \right).
\end{equation}
\end{proposition}
\begin{proof}
Select a solenoidal vector field $\g \in \Ls(U)$ to satisfy
\begin{equation*}
0 = \fint_{U} \left( F\left( \nabla u(x),\g(x),x \right) - \nabla u(x) \cdot \g(x) \right) \,dx
\end{equation*}
and recall that, by Proposition~\ref{p.yesvariational},
\begin{equation} \label{e.pinpoint1b}
F \left( \nabla u(x),\g(x),x\right)  =  \nabla u(x) \cdot \g(x) \quad \mbox{a.e. in} \ U.
\end{equation}
We consider the test function
$
v:=  u + \eta^2 (u-f),
$
where $\eta\in C^\infty_c(B_{2r}(z))$ is a smooth cutoff function satisfying 
\begin{equation*} 
0 \leq \eta \leq 1, \quad \eta \equiv 1 \ \mbox{on} \ \overline B_r(z),  \quad \supp \eta \subseteq B_{2r}(z),  \quad \sup_{x\in B_{2r}(z)} \left| \nabla \eta(x) \right| \leq Cr^{-1}.
\end{equation*}
Using that $\g$ is solenoidal, $u-v\in H^1_0(B_{2r}(z)\cap U)$ and~\eqref{e.pinpoint1b}, we get
\begin{align*}
0 & = \int_{B_{2r}(z)\cap U} \g(x) \cdot \left( \nabla v(x) - \nabla u(x) \right) \, dx \\
& =  \int_{B_{2r}(z)\cap U} \g(x) \cdot  \left( \eta^2(x) ( \nabla u(x)-\nabla f(x)) + 2(u(x)-f(x))\eta(x) \nabla \eta(x) \right)\,dx \\
& \geq \int_{B_{2r}(z)\cap U} F(\nabla u(x),\g(x),x) \eta^2(x)\,dx  - \int_{B_{2r}(z) \cap U} \eta^2(x) \left| \g(x) \right| |\nabla f(x)|\,dx \\
& \qquad  - Cr^{-1} \int_{B_{2r}(z)\cap U}  \eta(x)\left| \g(x) \right| |u(x)-f(x)| \,dx.
\end{align*}
Using~\eqref{e.P1again} and Young's inequality, we get
\begin{align*}
\lefteqn{ \int_{B_{2r}(z)\cap U} \left( \left| \nabla u(x)\right|^2+\left|\g(x)\right|^2\right)  \eta^2(x) \,dx } \qquad  & \\
& \leq CK_0^2 +C \int_{B_{2r}(z)\cap U}  F(\nabla u(x),\g(x),x) \eta^2(x) \,dx \\
& \leq CK_0^2+C \int_{B_{2r}(z) \cap U} \eta^2(x) \left| \g(x) \right| |\nabla f(x)|\,dx\\
& \qquad +Cr^{-1}\int_{B_{2r}(z)\cap U} \eta(x) \left| \g(x) \right| |u(x)-f(x)| \,dx \\
& \leq CK_0^2 + C\int_{B_{2r}(z)\cap U}\left|\g(x)\right|^2 \eta^2(x) \, dx + C\int_{B_{2r}(z) \cap U} \eta^2(x) |\nabla f(x)|^2\,dx  \\ 
& \qquad + Cr^{-2} \int_{B_{2r}(z) \cap U} \left| u(x)-f(x) \right|^2\,dx.
\end{align*}
This completes the proof of the proposition. 
\end{proof}

\begin{proposition}
\label{p.meyersbndry}
Fix $\delta > 0$, $f\in W^{1,2+\delta}(U)$ and~$u\in f+H^1_0(U)$ such that
\begin{equation*} \label{}
\J\left[ u,0 \right] = 0.
\end{equation*}
Then there exist $\delta_0(d,\Lambda,\delta)\in (0,\delta]$ and~$C(d,\Lambda,\delta,U)\geq 1$ such that
\begin{equation*} \label{}
\| \nabla u \|_{L^{2+\delta_0}(U)} \leq C \left( \| u \|_{L^2(U)} + \| f\|_{W^{1,2+\delta}(U)} \right).
\end{equation*}

\end{proposition}

\begin{proof}
According to the Sobolev extension theorem for bounded Lipschitz domains (cf.~\cite[Theorem 4.32]{Adams}), we may assume that $f\in W^{1,2+\delta}(\Rd)$ and
\begin{equation}\label{e.extend}
\| f \|_{W^{1,2+\delta}(\Rd)} \leq C \| f \|_{W^{1,2+\delta}(U)}.
\end{equation}
We may also extend $u$ to $\Rd$ by setting $u:=f$ in $\Rd \setminus U$.

\smallskip

We first claim that, for every $z\in\Rd$ and $r>0$, we have
\begin{multline}\label{e.setupGehring}
\fint_{B_r(z)} \left|  \nabla u(x) -\nabla f(x) \right|^2\,dx \\ 
 \leq C\left( K_0^2+ \left(  \fint_{B_{2r}(z)} \left| \nabla u(x) -\nabla f(x) \right|^{\frac{2d}{d+2}} \,dx  \right)^{\frac{d+2}{d}}  + \fint_{B_{2r}(z)} \left| \nabla f(x) \right|^2\,dx \right),
\end{multline}
We split the argument into two cases: $B_{2r}(z) \setminus U = \emptyset$ or $B_{2r}(z) \setminus U \neq \emptyset$. In the former case, we use the interior Caccioppoli estimate to obtain, as in the proof of Proposition~\ref{p.interiorMeyers}, that
\begin{equation*}
\fint_{B_r(z)} \left|  \nabla u(x) \right|^2\,dx  \leq C\left( K_0^2+ \left(  \fint_{B_{2r}(z)} \left| \nabla u(x) \right|^{\frac{2d}{d+2}} \,dx  \right)^{\frac{d+2}{d}} \right),
\end{equation*}
which of course implies~\eqref{e.setupGehring}. If $B_{2r}(z) \setminus U \neq \emptyset$, then, since $U$ is Lipschitz, we have 
\begin{equation*} \label{}
\left| B_{4r}(z) \setminus U \right| \geq c \left|B_{4r}(z)\right|.
\end{equation*}
Since $u-f \equiv 0$ on $\Rd\setminus U$, the Sobolev-Poincar\'e inequality (the version we need can be found in~\cite[Theorem 3.16]{Giu}) gives
\begin{equation*}
 r^{-2} \fint_{B_{4r}(z) } \left| u(x) - f(x) \right|^2\,dx \leq C \left( \fint_{B_{4r}(z)} \left| \nabla u(x) - \nabla f(x)  \right|^{\frac{2d}{d+2}}   \right)^{\frac{d+2}{2d}}.
\end{equation*}
Then according to~\eqref{e.catchbdry}, for every $z\in\Rd$ and $r>0$,
\begin{equation*}
\| \nabla u-\nabla f \|_{L^2(B_r(z))} \leq C\left( K_0+ r^{-1} \| u - f \|_{L^2(B_{2r}(z))} +  \| \nabla f \|_{L^2(B_{2r}(z))} \right).
\end{equation*}
The combination of the previous two inequalities gives~\eqref{e.setupGehring}, except that the balls on the right side have radius $4r$ rather than $2r$. This can of course be removed by a simple covering argument. 

\smallskip

Applying Lemma~\ref{l.Gehring} (with $V=U$ and $U=U+B_1$), in view of~\eqref{e.extend} and~\eqref{e.setupGehring}, now gives the result. 
\end{proof}

\bibliographystyle{plain}
\bibliography{monotone}
\end{document}